\documentclass[10pt,a4,leqno]{amsart}
\usepackage{amsthm}
\usepackage{amsmath}
\usepackage{amssymb}
\usepackage{bm}
\usepackage{amsfonts}
\usepackage{graphicx}
\usepackage {arydshln}
\usepackage[cmtip,arrow]{xy}
\usepackage{pb-diagram,pb-xy}
\usepackage{bbm}
\usepackage{comment} 
\usepackage{color}
\allowdisplaybreaks

\makeatletter

\@addtoreset{equation}{section}
\makeatother

\newtheoremstyle{my theoremstyle}
{1.0em}                    % Space above
    {1.0em}                    % Space below
    {\itshape}                   % Body font
    {}                           % Indent amount
    {\scshape}                   % Theorem head font
    {.}                          % Punctuation after theorem head
    {.5em}                       % Space after theorem head
    {}  % Theorem head spec (can be left empty, meaning ‘normal’)

\newtheoremstyle{dfn}
{1.0em}                    % Space above
    {1.0em}                    % Space below
    {}                   % Body font
    {}                           % Indent amount
    {\scshape}                   % Theorem head font
    {.}                          % Punctuation after theorem head
    {.5em}                       % Space after theorem head
    {}  % Theorem head spec (can be left empty, meaning ‘normal’)
    
\theoremstyle{my theoremstyle}
   \newtheorem{thm}{Theorem}[section]
   \newtheorem*{thm*}{Theorem}
   \newtheorem{introthm}{Theorem}
   
   \newtheorem{lem}[thm]{Lemma}
   \newtheorem{prop}[thm]{Proposition}
   
\theoremstyle{dfn}
   \newtheorem{dfn}[thm]{Definition}
   
\theoremstyle{remark}   
   
   \newtheorem{rem}[thm]{{\scshape Remark}}

\renewcommand{\1}{\mathbbm{1}}
\newcommand{\C}{\mathbb{C}}
\newcommand{\Q}{\mathbb{Q}}
\newcommand{\Z}{\mathbb{Z}}

\newcommand{\D}{\Delta}
\renewcommand{\S}{\mathfrak{S}}
\renewcommand{\P}{\mathbb{P}}
\renewcommand{\a}{\alpha}
\renewcommand{\b}{\beta}
\renewcommand{\c}{\gamma}
\newcommand{\e}{\varepsilon}
\renewcommand{\d}{\delta}
\renewcommand{\k}{\kappa}
\renewcommand{\l}{\lambda}
\newcommand{\khat}{\widehat{\k^*}}
\newcommand{\p}{\varphi}

\newcommand{\ol}[1]{\overline{#1}}

\newcommand{\hFF}[5]{{}_{#1}F_{#2}\left({#3\atop#4};#5\right)}
\newcommand{\hF}[3]{F\left({#1\atop#2};#3\right)}

\newcommand{\FA}[4]{F_A^{(#1)}\left({#2\atop #3};#4\right)}
\newcommand{\FB}[4]{F_B^{(#1)}\left({#2\atop #3};#4\right)}
\newcommand{\FC}[4]{F_C^{(#1)}\left({#2\atop #3};#4\right)}
\newcommand{\FD}[4]{F_D^{(#1)}\left({#2\atop #3};#4\right)}

\newcommand{\bmat}[1]{\begin{pmatrix} #1 \end{pmatrix}}
\newcommand{\mat}[1]{\begin{matrix} #1 \end{matrix}}

\begin{document}
\title[Symmetry of hypergeometric functions over finite fields]{Symmetry of hypergeometric functions over finite fields and geometric interpretation}
\author{Akio Nakagawa}
\date{\today}
\email{akio.nakagawa.math@icloud.com}
\address{Faculty of Engineering, Department of General Education, Chiba Institute of Technology, 2-1-1 Shibazono, Narashino, Chiba, 275-0023, JAPAN}

\maketitle
\begin{abstract}
We study a symmetry of general hypergeometric functions over finite fields. We extend the symmetry of the general hypergeometric functions over finite fields to the confluent setting and realize a broad class of the resulting transformation formulas by explicit isomorphisms of algebraic varieties. This gives a unified geometric interpretation of the non-confluent and confluent cases.
\end{abstract}
\section{Introduction}
%\begin{comment}

The main result of this paper is that a broad class of transformation formulas for hypergeometric functions over finite fields, including formulas of confluent type, can be realized by explicit isomorphisms of algebraic varieties.
The transformation formulas considered here are therefore not merely identities between character sums: they have a geometric origin. 
Our construction provides a common geometric framework for both the non-confluent and confluent cases.

Over $\C$, hypergeometric functions ${}_mF_n$, including Gauss' function ${}_2F_1$ and Kummer's function ${}_1F_1$, as well as the Appell-Lauricella and Humbert's hypergeometric functions, are classical and well-studied special functions (cf. \cite{Slater, Humbert}).
By extending their integral representations, the general hypergeometric function $\Phi_\D^\C(\chi;z)$ was introduced in \cite{K-H-T}.
Here, $z \in M(d,n;\C)$ is a $d \times n$ matrix, $\D=(N_1,\dots,N_l)$ is a partition of $n$ and the parameter $\chi$ is a character of a group $H_\D^\C \cong \prod_{i=1}^l (\C^* \times \C^{N_i-1})$.
The classical functions above are recovered as special cases.
The following functions are said to be {\it of confluent type}: ${}_mF_n$-functions with $n \neq m-1$, Humbert's functions and general hypergeometric functions with $\Delta \neq (1, \dots, 1)$.
Kimura and Koitabashi \cite{K-K} introduced a group $W_\D^\C \subset GL_n(\C)$ and established a symmetry of the form  $\Phi_\D^\C(\chi {}^tw;z) = \Phi_\D^\C(\chi;zw)$, where $w \in W_\D^\C$ and $\chi{}^tw$ is a suitably transformed character.
Classical transformation formulas (cf. Pfaff's and Euler's formulas) arise as special cases of this symmetry.

Let $q$ be a power of a prime $p$, and let $\k$ be a finite field with $q$ elements.
Over $\k$, analogues of ${}_mF_n$, Appell-Lauricella and Humbert's functions have been defined and discussed by several authors (e.g. \cite{BCM,CK,FST,FLRST,Greene,He,IKNN,Katz,Koblitz,Mc,Otsubo,TB}).
We use Otsubo's definition, which is convenient for treating confluent type (see \cite[Remark 2.13]{Otsubo} for the relation with other definitions). 
The functions over $\k$ satisfy finite field analogues of many formulas and the integral representations.
A non-trivial additive character $\psi$ of $\k$ is an analogue of the exponential function, and appears naturally in the analogues of the integral representations in the confluent case.
Gel'fand-Graev-Retakh \cite{GGR} defined a finite field analogue of the general hypergeometric function for $\D=(1,\dots,1)$, and obtained its symmetry.
We will extend their construction and symmetry to arbitrary partitions $\Delta$, and then realize the resulting symmetry geometrically by explicit isomorphisms of algebraic varieties.

The geometric motivation for this extension is particularly clear in the confluent case.
Whereas the standard integral representations in the non-confluent case involve only algebraic factors, those in the confluent case also contain exponential factors.
Over $\C$, non-confluent hypergeometric functions with rational parameters can often be interpreted as complex periods of algebraic varieties (cf. \cite{Koblitz, Archinard}).
For functions of confluent type, however, the exponential factors in their integral representations prevent the same direct interpretation as ordinary periods.
Over $\k$, non-confluent functions such as \({}_mF_{m-1}\) and the Appell-Lauricella functions have already been expressed in terms of the numbers of rational points on algebraic varieties (cf. \cite{Koblitz, N}).
Additive characters can be realized using Artin-Schreier varieties over $\k$.
This makes it possible to extend the point-counting interpretation from the non-confluent case to the confluent case.
In this paper, we combine this Artin-Schreier construction with the geometry appearing in the non-confluent case to construct a family of hypergeometric varieties including the confluent case.
We then show that the symmetry of hypergeometric functions is realized by explicit isomorphisms between these varieties.
Thus, our construction not only extends the geometric interpretation to the confluent case, but also explains the transformation formulas themselves geometrically.

%we combine this ingredient with the multiplicative part in a family of hypergeometric varieties and show that explicit isomorphisms between these varieties realize the corresponding hypergeometric symmetries. To the author's knowledge, such a unified geometric interpretation, including the confluent case, is specific to the finite field setting.

We now describe our construction and main results.
We introduce a finite field analogue $H_\D\cong \prod_{i=1}^l (\k^* \times \k^{N_i-1})$ of $H_\D^\C$.
For a character $\chi$ of $H_\D$ and a matrix $z \in M(d,n; \k)$, we define the general hypergeometric function $\Phi_\D(\chi;z)$ over $\k$.
We also construct a group $W_\D \subset GL_n(\k)$.
Our first result is the following analogue of the symmetry due to Kimura-Koitabashi.

\begin{introthm}\label{eq. intro-sym2}
Let $\chi$ be a character of $H_\D$. For $z \in M(d,n;\k)$ and $w \in W_\D$, 
$$\Phi_\D(\chi{}^tw;z) = \Phi_\D(\chi;zw). $$
See Subsection \ref{subsec. of Def. and Prop.} for the definition of $\chi {}^tw$. 
\end{introthm}

For suitable choices of $d, n, \D, \chi$ and $z$, the classical hypergeometric functions $F(\l)$ over $\k$, where $\l $ is associated with $z$, are recovered from $\Phi_\D(\chi;z)$, and Theorem \ref{eq. intro-sym2} specializes to transformation formulas between $F(\l)$ and $F(\l_w)$, where $\l_w$ is associated with $zw$.
For example, when $(d,n,\D)=(2,4,(1,1,2))$, it gives the following analogue of Kummer's formula, due to Otsubo \cite[Theorem 6.1 (i)]{Otsubo}:
\begin{equation}
  \psi(\l)\hFF{1}{1}{\a^{-1}\b}{\b}{\l}_\psi =\hFF{1}{1}{\a}{\b}{-\l}_\psi. \label{Kummer prod}
\end{equation}
Here, the parameters $\a,\b$ are multiplicative characters of $\k^*$, $\psi$ is a fixed non-trivial additive character of $\k$ and $\l \in \k^*$. 

The central geometric result of this paper is that the symmetry in Theorem \ref{eq. intro-sym2} is induced by isomorphisms of varieties.
For each $z \in M(d,n;\k)$, we construct an affine algebraic variety $X_{\D,z} \subset \mathbbm A^{d+n}$ equipped with a natural action of $H_\D$.
If a variety $X$ over $\k$ is equipped with an action of a finite group $G$, the number $\#X(\k)$ of $\k$-rational points decomposes into $\chi$-components $N(X;\chi)$ for characters $\chi$ of $G$.
The geometric interpretation of Theorem A is summarized as follows.
\begin{introthm}[Theorems \ref{N of X_D(z)-1}, \ref{isom fw} and \ref{N of X_D(z)} \ref{N of X_D(z)-4}]\label{thmA'}$ $
\begin{enumerate}
\item For a character $\chi$ of $H_\D$, we have
$$ N(X_{\D,z};\chi) = \Phi_\D(\chi;z).$$
\item For $w \in W_\D$, we have an explicit isomorphism
$$  X_{\D,z} \longrightarrow X_{\D,zw}. $$
\item The isomorphism in (ii) gives, for a character $\chi $ of $H_\D$ and $w \in W_\D$, 
$$ N(X_{\D,z}; \chi{}^t w) = N(X_{\D,zw};\chi). $$
\end{enumerate}
\end{introthm}
In particular, the symmetry in Theorem \ref{eq. intro-sym2} is a consequence of an isomorphism of algebraic varieties.

For each of the classical hypergeometric functions $F$ over $\k$ and $\l \in \k^*$, we construct a simpler variety $X_{F,\l}$, equipped with an action of a finite abelian group.
Its character components are expressed in terms of $F(\l)$.
The isomorphisms in Theorem \ref{thmA'} (ii) induce, after a suitable finite extension $\k'/\k$, explicit isomorphisms
$$ X_{F,\l} \otimes \k' \longrightarrow X_{F,\l_w} \otimes \k'.$$
The transformation formulas between $F(\l)$ and $F(\l_w)$ are recovered by comparing the character components of the point counts via the isomorphisms.
In particular, this construction gives geometric realizations of the 24 transformation formulas for Gauss' function in \cite[Corollary 3.16]{Otsubo}.
It also applies to Kummer's, Lauricella's and Humbert's functions.

We illustrate the confluent case with Kummer's function ($\D=(1,1,2)$).
Define the affine variety ${}_1X_{2,\l} \subset \mathbbm A^4$ ($\l \in \k^*$) by
$$ \begin{cases} x^{q-1} + y^{q-1} = 1 \\ t^q - t = z^{q-1} \\ \l x^{q-1} = (yz)^{q-1} \\ xyz \neq 0. \end{cases} $$
The group $G = (\k^*)^3 \times \k$ acts naturally on ${}_1X_{2,\l}$.
In this case, $W_\D \cong \S_2\times \k^*$, where $\S_2$ is the symmetric group of degree $2$.
The following theorem illustrates how a transformation formula for a confluent hypergeometric function is realized geometrically.
\begin{introthm}[Theorems \ref{N of X_mn} and \ref{N of X12}]\label{intro ThmC}
Let $w=(\sigma,c) \in W_{\D}$, where $\sigma \in \S_2$ and $c \in \k^*$.
\begin{enumerate}
\item For a character $\chi = (\a,\b,\e,\psi)$ of $G$, if $\a\b \neq \e$, we have
\begin{equation*}
N({}_1X_{2,\l};\chi) = -j(\a,\b)\hFF{1}{1}{\a}{\b^{-1}}{\l}_\psi. 
\end{equation*}
Here, $\e$ is the trivial character and $j(\a,\b)$ is the Jacobi sum.

\item 
For $\l_w := {\rm sgn}(\sigma) c \l$, we have an explicit isomorphism 
$${}_1X_{2,\l} \otimes \k'\longrightarrow {}_1X_{2,\l_w} \otimes \k',$$
where $[\k':\k]=p(q-1)$.

\item  The isomorphism above gives, for $\chi = (\a, \b, \e,\psi)$,
\[
\b({\rm sgn}(\sigma) ) \psi(c_\sigma \l) N({}_1X_{2,\l}; \chi_w) = N({}_1X_{2,\l_w};\chi),
\]
where $c_\sigma = c$ if $\sigma \neq {\rm id}$ and $c_\sigma = 0$ otherwise.
For $\sigma \neq {\rm id}$ and $c=1$, $\chi_w = ( (\a\b)^{-1}, \b, \e,\psi)$ and hence, (iii) recovers \eqref{Kummer prod}.
\end{enumerate}
\end{introthm}

The paper is organized as follows. Section 2 reviews hypergeometric functions over finite fields and the discrete Fourier transform formulas used later. Section 3 defines \(\Phi_\Delta(\chi;z)\), constructs the symmetry group \(W_\Delta\), and states Theorem \ref{eq. intro-sym2}. Section 4 constructs the associated varieties and identifies the character components of their numbers of rational points with the corresponding hypergeometric functions. Section 5 constructs the explicit isomorphisms, proves the symmetry geometrically, and treats the cases of Gauss', Kummer's, Lauricella's, and Humbert's functions.

\section{Preliminaries}
Throughout this paper, for a group $G$, we write $\widehat{G} = {\rm Hom}(G, \C^*)$ for the character group, and let $\d \colon G \rightarrow \{0,1\}$ denote the characteristic function of the identity element.
Let $\e \in \khat$ be the trivial character,  set $\a(0) = 0 $ and put $\ol\a = \a^{-1}$ for all $ \a \in \khat$.
Fix a non-trivial additive character $\psi \in \widehat{\k}$, and for each $a \in \k$, define $\psi_a \in \widehat{\k}$ by $ \psi_a(x) = \psi(ax)$ $( x \in \k ).$
Write $\S_n$ for the symmetric group of degree $n$,  and for $\sigma \in \S_n$, let $ P_\sigma = (e_{\sigma(1)}, \dots, e_{\sigma(n)}) \in GL_n(\Z)$ be the permutation matrix, where $e_i$ is the $i$th standard unit vector.
Write $O_n$ and $I_n$ for the zero and identity matrices, respectively, of size $n$.

\subsection{Review of hypergeometric functions over finite fields}\label{subsec. recall for HGF}
For multiplicative characters $\eta,\eta_1,\dots,\eta_n  \in \widehat{\k^*}$ ($n\geq 2$), {\it the Gauss sum} and {\it the Jacobi sum} are 
\begin{align*}
g(\eta)&=-\sum_{x\in \k^*} \psi(x)\eta(x)\ \in\Q(\mu_{p(q-1)}),\\
j(\eta_1,\dots,\eta_n)&=(-1)^{n-1}\sum_{\substack{x_i\in\k^*\\ x_1+\dots+x_n=1}}\prod_{i=1}^n \eta_i(x_i)\ \in\Q(\mu_{q-1}).
\end{align*}
Note that $g(\e)=1$. 
Put $g^\circ(\eta)=q^{\d(\eta)}g(\eta)$. Then, one shows (cf. \cite[Proposition 2.2 (iii)]{Otsubo})
\begin{equation}\label{Gauss sum thm}
g(\eta)g^\circ(\ol{\eta})=\eta(-1)q.
\end{equation}
For $\eta_1,\dots,\eta_n\in\widehat{\k^*}$, we have (cf. \cite[Proposition 2.2 (iv)]{Otsubo})
\begin{equation}\label{J=G}
j(\eta_1,\dots,\eta_n)=
\begin{cases}
\dfrac{1-(1-q)^n}{q}&(\eta_1=\cdots=\eta_n=\e),\vspace{5pt}\\
\dfrac{g(\eta_1)\cdots g(\eta_n)}{g^\circ(\eta_1\cdots\eta_n)}&({\rm otherwise}).
\end{cases}
\end{equation}
As an analogue of the Pochhammer symbol $(a)_n=\Gamma(a+n)/\Gamma(a)$, where $\Gamma$ is the gamma function, put 
\begin{align*}
(\a)_\nu=\dfrac{g(\a\nu)}{g(\a)},\ \ \ \ (\a)_\nu^\circ=\dfrac{g^\circ(\a\nu)}{g^\circ(\a)},
\end{align*}
for $\a,\ \nu\in\widehat{\k^*}$. 
One shows (cf. \cite[Lemma 2.5 (ii)]{Otsubo})
\begin{equation}
(\a)_\nu (\ol\a)_{\ol\nu}^\circ  = \nu(-1). \label{Poch formula}
\end{equation}
%One shows 
%\begin{align}\label{Poch formula}
%(\a)_{\nu\mu}=(\a)_\nu(\a\nu)_\mu,\ \ \ (\a)_{\nu\mu}^\circ=(\a)_\nu^\circ(\a\nu)_\mu^\circ
%\end{align}
%and 
%\begin{equation}
%(\a)_\nu = \nu(-1) \dfrac{1}{(\ol{\a})_{\ol{\nu}}^\circ}. \label{inversion Poch}
%\end{equation}
%If $p \neq 2$, we will use the duplication formulas (cf. \cite[Theorem 3.10 and Corollary 3.11]{Otsubo})
%\begin{equation}
%g(\a^2) = \a(4) \dfrac{g(\a)g(\a\phi)}{g(\phi)},\quad g^\circ(\a^2) = \a(4) \dfrac{g^\circ(\a) g^\circ(\a\phi)}{g(\phi)}\label{dup. form. gauss}
%\end{equation}
%and
%\begin{equation}
%(\a^2)_{\nu^2} = \nu(4) (\a)_\nu (\a\phi)_\nu,\quad (\a^2)_{\nu^2}^\circ = \nu(4) (\a)_\nu^\circ (\a\phi)_\nu^\circ. \label{dup. form. Poch}
%\end{equation}

For $\a_1,\dots,\a_m,\b_1,\dots,\b_{n+1}\in\widehat{\k^*}$ and $\l \in \k$, define
\begin{equation*}
\hF{\a_1,\dots,\a_m}{\b_1,\dots,\b_{n+1}}{\lambda}_\psi=\dfrac{1}{1-q}\sum_{\nu\in\widehat{\k^*}}\dfrac{(\a_1)_\nu\cdots(\a_m)_\nu}{(\b_1)_\nu^\circ\cdots(\b_{n+1})_\nu^\circ}\nu(\lambda) \ \in \Q(\mu_{p(q-1)}).
\end{equation*}
We often omit writing the subscript $\psi$.
When $\b_{n+1}=\e$, we use the classical notation:
$$ \hFF{m}{n}{\a_1,\dots,\a_m}{\b_1,\dots,\b_{n}}{\l} := \hF{\a_1,\dots, \a_m}{\b_1,\dots,\b_n,\e}{\l}.$$
Analogues of Gauss' function and Kummer's function are $\hFF{2}{1}{\a, \b}{\c}{\l} = \hF{\a,\b}{\c,\e}{\l}$ and $\hFF{1}{1}{\a}{\c}{\l} = \hF{\a}{\c,\e}{\l}$, respectively.
For $\l \in \k^*$, one shows (cf. \cite[Proposition 2.9 (i) and Corollary 3.4]{Otsubo})
\begin{equation}
\hFF{0}{0}{}{}{\l} = \hF{}{\e}{\l} =  \psi(-\l) \label{0F0 int}
\end{equation} 
and
\begin{equation}
\hFF{1}{0}{\a}{}{\l} = \hF{\a}{\e}{\l} = \ol{\a}(1-\l) \quad (\a\neq\e).\label{1F0 int.}
\end{equation}

For $\bm\l=(\l_1,\dots,\l_n) \in \k^n$, the Appell-Lauricella hypergeometric functions 
\begin{align*}
& \FA{n}{\a\,;\, \b_1,\dots,\b_n}{\c_1,\dots,\c_n ;  \d_1,\dots,\d_n}{\bm\l}_\psi, \quad  \FB{n}{\a_1,\dots,\a_n; \b_1,\dots,\b_n}{\c \, ;\,  \d_1,\dots,\d_n}{\bm\l}_\psi,\\
& \FC{n}{\a \, ;\,  \b}{\c_1,\dots,\c_n \, ;\, \d_1,\dots,\d_n}{\bm\l}_\psi, \quad \FD{n}{\a\, ;\, \b_1,\dots,\b_n}{\c\, ;\, \d_1,\dots,\d_n}{\bm\l}_\psi,
\end{align*}  
over $\k$ are similarly defined (cf. \cite[Subsection 2.5]{Otsubo}).
For example, 
\begin{align*}
 &\FA{n}{\a\,;\, \b_1,\dots,\b_n}{\c_1,\dots,\c_n ;  \d_1,\dots,\d_n}{\bm\l}_\psi
 := \dfrac{1}{(1-q)^n}\sum_{\nu_i \in \khat} \dfrac{ (\a)_{\nu_1\cdots \nu_n} \prod_{i=1}^n (\b_i)_{\nu_i}}{ \prod_{i=1}^n (\c_i)_{\nu_i}^\circ (\d_i)_{\nu_i}^\circ}\prod_{i=1}^n \nu_i(\l_i),\\
 & \FD{n}{\a\, ;\, \b_1, \dots, \b_n}{\c\, ;\, \d_1, \dots, \d_n}{\bm\l}_\psi
  :=  \dfrac{1}{(1-q)^n}\sum_{\nu_i \in \khat} \dfrac{ (\a)_{\nu_1\cdots \nu_n} \prod_{i=1}^n (\b_i)_{\nu_i}}{ (\c)_{\nu_1\cdots \nu_n}^\circ \prod_{i=1}^n (\d_i)_{\nu_i}^\circ}\prod_{i=1}^n \nu_i(\l_i).
\end{align*}
We often omit writing the subscript $\psi$.
When $\d_i = \e$ for all $i$, we use the classical notation (we omit writing $\d_1,\dots,\d_n$).
When $n=2$, we use Appell's notation $F_2, F_3, F_4, F_1$ for $F_A^{(2)}, F_B^{(2)}, F_C^{(2)}, F_D^{(2)}$, respectively.
Note that $\c_i$ and $\d_i$ play symmetric roles in $F_A$ and $F_C$.

\begin{rem}
The functions ${}_mF_{m-1}$ and Appell-Lauricella functions over $\k$ are $\Q(\mu_{q-1})$-valued and independent of the choice of $\psi$ (see \cite[Lemma 2.4 (iii)]{Otsubo}).
\end{rem}

Similarly, we define analogues of Humbert's hypergeometric functions (\cite{Humbert}).
\begin{dfn}
\begin{align*}
& \Phi_1\left( {\a; \b \atop \c; \d_1, \d_2}; \l_1,\l_2\right)_\psi := \dfrac{1}{(1-q)^2}\sum_{\mu,\nu \in \khat} \dfrac{(\a)_{\mu\nu} (\b)_\mu}{(\c)_{\mu\nu}^\circ (\d_1)_\mu^\circ (\d_2)_\nu^\circ}\mu(\l_1)\nu(\l_2),\\
& \Phi_2\left( {\b, \b' \atop \c; \d_1,\d_2}; \l_1,\l_2\right)_\psi := \dfrac{1}{(1-q)^2}\sum_{\mu,\nu \in \khat} \dfrac{(\b)_{\mu} (\b')_\nu}{(\c)_{\mu\nu}^\circ (\d_1)_\mu^\circ (\d_2)_\nu^\circ}\mu(\l_1)\nu(\l_2),\\
& \Phi_3\left( {\b \atop \c; \d_1,\d_2}; \l_1,\l_2\right)_\psi := \dfrac{1}{(1-q)^2}\sum_{\mu,\nu \in \khat} \dfrac{(\b)_{\mu}}{(\c)_{\mu\nu}^\circ (\d_1)_\mu^\circ (\d_2)_\nu^\circ}\mu(\l_1)\nu(\l_2).
\end{align*}
We often omit writing the subscript $\psi$.
When $\d_1=\d_2=\e$, we use the classical notation (i.e. we omit writing $\d_1,\d_2$).
These three functions are particular cases of Kamp\'e de F\'eriet hypergeometric functions over $\k$ (\cite{IKNN}).
\end{dfn}

\begin{rem}\label{rem of parameters sift}
By the shift of parameters \cite[Proposition 2.10]{Otsubo}, any hypergeometric function can be written in the classical notation.
For example,
$$ \hF{\a_1,\a_2}{\b_1,\b_2}{\l} = \dfrac{(\a_1)_{\ol{\b_2}} (\a_2)_{\ol{\b_2}}}{(\b_1)^\circ_{\ol{\b_2}} (\b_2)^\circ_{\ol{\b_2}}}\ol{\b_2}(\l) \hFF{2}{1}{\a_1\ol{\b_2}, \a_2\ol{\b_2}}{\b_1\ol{\b_2}}{\l}.$$
\end{rem}
%\begin{rem}
%Over $\k$, one-variable hypergeometric functions were defined by Koblitz \cite{Koblitz}, Greene \cite{Greene}, McCarthy \cite{Mc} and Fuselier-Long-Ramakrishna-Swisher-Tu \cite{FLRST} when $m=n+1$, and by Katz \cite{Katz} and Otsubo \cite{Otsubo} in general.
%Appell's functions over $\k$ were defined by Li-Li-Mao \cite{Li-Li-Mao}, He \cite{He}, He-Li-Zhang \cite{He-Li-Zhang} and Ma \cite{Ma}, Tripathi-Barman \cite{TB} and Tripathi-Saikia-Barman \cite{TSB}.
%For general $n$, $F_D^{(n)}$ were defined by Frechette-Swisher-Tu \cite{FST} and He \cite{He}, and $F_A^{(n)}$ were defined by Chetry-Kalita \cite{CK}. Otsubo \cite{Otsubo} gave a definition of all the Lauricella functions.
%See \cite[Remark 2.13]{Otsubo} for the relation between Otsubo's definition and other definitions.
%\end{rem}

\subsection{Discrete Fourier transform}
For a function $f \colon (\k^*)^n \rightarrow \C$, its {\it Fourier transform} $\widehat{f}$ is a function on $(\widehat{\k^*})^n$ defined by
\begin{equation*}
\widehat{f}(\nu_1,\dots,\nu_n) :=\sum_{t_i\in\k^*}f(t_1,\dots,t_n)\prod_{i=1}^n\ol{\nu_i}(t_i).
\end{equation*}
Then, 
\begin{equation}\label{Fourier trans.}
f(\lambda_1,\dots,\lambda_n)=\dfrac{1}{(q-1)^n}\sum_{\nu_i\in\widehat{\k^*}}\widehat{f}(\nu_1,\dots,\nu_n)\prod_{i=1}^n \nu_i(\lambda_i).
\end{equation}
\begin{rem}
For $f(\l) = \hF{\a_1,\dots, \a_m}{\b_1,\dots,\b_n}{\l}$, one shows
$$ \widehat{f}(\nu) = - \dfrac{(\a_1)_\nu \cdots (\a_m)_\nu}{(\b_1)_\nu^\circ \cdots (\b_n)_\nu^\circ}.$$
The Fourier transforms of the Appell-Lauricella and Humbert's functions over $\k$ have analogous forms.
\end{rem}

\begin{prop}\label{iteration}
For $\a, \b, \b_1, \dots, \b_n \in \khat$, we have the following.
\begin{enumerate}
\item 
If $\ol\a \b_1\cdots\b_i \neq \e$, then for each $i = 1, \dots, n$,
\begin{align*}
& \dfrac{j(\ol\a\b_1\cdots\b_i, \ol{\b_1}, \dots, \ol{\b_i})}{(-1)^{i} (q-1)^n} \sum_{\nu_1, \dots, \nu_n \in \khat} \widehat{f}(\nu_1, \dots, \nu_n) \dfrac{(\a)_{\nu_1 \cdots \nu_i}}{(\b_1)_{\nu_1}^\circ \cdots (\b_i)_{\nu_i}^\circ} \prod_{j=1}^n \nu_j(\l_j)\\
& = \sum_{u_1, \dots, u_i \in \k^*} f(\dfrac{\l_1}{u_1}, \dots, \dfrac{\l_i}{u_i}, \l_{i+1}, \dots, \l_n) \ol\a\b_1\cdots\b_i(1-\sum_{j=1}^i u_j)\prod_{j=1}^i\ol{\b_j}(u_j).
\end{align*}

\item If  $ \ol\a\b_1\cdots\b_i \neq \e$, then for each $ i = 1, \dots, n$, 
\begin{align*}
& \dfrac{j(\a\ol{\b_1\cdots\b_i}, \b_1, \dots, \b_i)}{(-1)^{i}(q-1)^n} \sum_{\nu_1, \dots, \nu_n \in \khat} \widehat{f}(\nu_1, \dots, \nu_n) \dfrac{(\b_1)_{\nu_1} \cdots (\b_i)_{\nu_i}} {(\a)_{\nu_1 \cdots \nu_i}^\circ}\prod_{j=1}^n \nu_j(\l_j)\\
& = \sum_{u_1, \dots, u_i \in \k^*} f(\l_1u_1, \dots, \l_iu_i, \l_{i+1}, \dots, \l_n) \a\ol{\b_1\cdots\b_i}(1-\sum_{j=1}^i u_j)\prod_{j=1}^i\b_j(u_j).
\end{align*}

\item If $\a\ol\b\neq\e$, then for each $i = 1, \dots, n$, 
\begin{align*}
& -\dfrac{j(\a, \ol\a\b)}{(q-1)^n}\sum_{\nu_1, \dots, \nu_n \in \khat} \widehat{f}(\nu_1, \dots, \nu_n) \dfrac{(\a)_{\nu_1\cdots \nu_i}}{(\b)_{\nu_1\cdots\nu_i}^\circ} \prod_{j=1}^n \nu_j(\l_j)\\
& = \sum_{u \in \k^*} f(\l_1u , \dots, \l_iu, \l_{i+1},\dots, \l_n) \a(u) \ol\a\b(1-u).
\end{align*}

\item For each $i=1, \dots, n$, 
\begin{align*}
& -\dfrac{g(\ol\a)}{(q-1)^n} \sum_{\nu_1, \dots, \nu_n \in \khat} \widehat{f}(\nu_1, \dots, \nu_n) \dfrac{1}{(\a)_{\nu_1\cdots \nu_i}^\circ} \prod_{j=1}^n \nu_j(\l_j)\\
& = \sum_{u \in \k^*} f(-\dfrac{\l_1}{u}, \dots, -\dfrac{\l_i}{u}, \l_{i+1}, \dots, \l_n) \ol\a(u) \psi(u).
\end{align*}
\end{enumerate}
\end{prop}
\begin{proof}
We prove only (i) (the others can be shown similarly).
If we write $F(\l_1, \dots, \l_n)$ for the right-hand side, then
\begin{align*}
& \widehat{F}(\nu_1, \dots, \nu_n) \\
& = \sum_{u_1, \dots, u_i}\sum_{t_1, \dots, t_n} f(\dfrac{t_1}{u_1}, \dots \dfrac{t_i}{u_i}, t_{i+1}, \dots, t_n) \ol\a\b_1\cdots\b_i(1-\sum_{j=1}^i u_j)\prod_{j=1}^i\ol{\b_j}(u_j)\prod_{j=1}^n \ol{\nu_j}(t_j)\\
& = \sum_{u_1, \dots, u_i}\ol\a\b_1\cdots \b_i(1-\sum_{j=1}^i u_j) \prod_{j=1}^i \ol{\b_j\nu_j}(u_j) \sum_{s_1, \dots, s_n\in\k^*}f(s_1, \dots, s_n) \prod_{j=1}^n \ol{\nu_j}(s_j)\\
& = (-1)^i j(\ol{\a}\b_1\dots\b_i, \ol{\b_1\nu_1}, \dots, \ol{\b_i\nu_i}) \widehat{f}(\nu_1, \dots, \nu_n).
\end{align*}
Here, we put $ s_j = t_j/u_j$ for $1 \leq j \leq i$ and $s_j=t_j$ for $i+1 \leq j \leq n$.
Therefore, by \eqref{J=G} and \eqref{Gauss sum thm}, we have
$$ \widehat{F}(\nu_1, \dots, \nu_n) = (-1)^i j(\ol\a\b_1\cdots\b_i, \ol{\b_1}, \dots, \ol{\b_i})\widehat{f}(\nu_1, \dots, \nu_n) \dfrac{(\a)_{\nu_1 \cdots \nu_i}}{(\b_1)_{\nu_1}^\circ \cdots (\b_i)_{\nu_i}^\circ}.$$ 
Thus, we obtain the proposition by \eqref{Fourier trans.}.
\end{proof}

\section{General hypergeometric functions over finite fields}
\subsection{Definition and properties}\label{subsec. of Def. and Prop.}
Fix a positive integer $m$, and let $\Lambda$ be the shift matrix of size $m$:
$$ \bmat{ 0 & 1      & 0      & \cdots & 0 \\ 
                      & \ddots & \ddots & \ddots & \vdots \\
                      &   & \ddots & \ddots & 0 \\
                      &   &        & \ddots & 1\\
                      &  &  &  & 0}.$$
For $h_0 , \dots, h_{m-1} \in \k$, write
$$ [h_0, \dots, h_{m-1}] =  \sum_{i=0}^{m-1} h_i \Lambda^i.$$
Define the group
$$ J(m) = \left\{ [h_0, \dots, h_{m-1}] \, \middle| \, h_0 \in \k^*, h_1, \dots, h_{m-1} \in \k \right\} \subset GL_m(\k).$$
We recall the explicit form of characters of $J(m)$ in analogy with the complex numbers case (cf. \cite{K-H-T}).
For indeterminates $x=(x_0,x_1,\dots)$, define the polynomial $ \theta_i(x) \in \Q[ X_1, \dots, X_i]$, where $X_j := x_j/x_0$, by
$$ \theta_i(x) = \sum_{\substack{k_1,\dots,k_i \geq 0 \\ k_1+2k_2+\cdots+ik_i=i}} (-1)^{k_1+\cdots +k_i-1}\dfrac{(k_1+\cdots +k_i -1)!}{k_1! \cdots k_i!} X_1^{k_1} \cdots X_i^{k_i}.$$
It is known that (cf. \cite[p. 417]{K-T})
$$ \log( x_0 + x_1T + x_2T^2 + \cdots ) = \log x_0 + \sum_{i \geq 1} \theta_i(x) T^i,$$
and hence, we have
\begin{align*}
\log( \sum_{i \geq 0} x_iT^i \sum_{i \geq 0} y_iT^i ) = \log(x_0y_0) + \sum_{i \geq 1} (\theta_i(x) + \theta_i(y)) T^i.
\end{align*}
Therefore, for $x=(x_0,x_1,\dots)$ and $y=(y_0,y_1,\dots)$, if we define $z = (z_0,z_1,\dots)$ by 
$$ \sum_i z_iT^i = ( \sum_i x_iT^i ) ( \sum_i y_iT^i ),$$
then 
\begin{equation}
 \theta_i(z) = \theta_i(x) + \theta_i(y). \label{additivity of theta}
\end{equation}
For $h = [h_0, \dots, h_{m-1}] \in J(m)$,  we write $\theta_i(h)$ for $\theta_i(h_0, \dots, h_{i})$.
The following proposition is a finite field analogue of a well-known fact over $\C$.
\begin{prop}\label{char of J(m)}
Suppose that $p \geq m$.
\begin{enumerate}
\item \label{char of J(m)-1} We have the isomorphism
$$ \iota \colon J(m) \longrightarrow \k^* \times \k^{m-1}\, ;\, h \longmapsto (h_0, \theta_1(h), \dots, \theta_{m-1}(h)).$$

\item \label{char of J(m)-2} We have
$$ \widehat{J(m)} = \left\{ (\a,a_1, \dots, a_{m-1}) \, \middle| \, \a \in \khat, a_i \in \k \right\},$$
where 
$$ (\a, a_1, \dots, a_{m-1}) := (\a, \psi_{a_1}, \dots, \psi_{a_{m-1}}) \circ \iota. $$
\end{enumerate}
\end{prop}
\begin{proof}
\ref{char of J(m)-1} Noting \eqref{additivity of theta}, for $h, h' \in J(m)$, one shows
$$ \iota(hh') = (h_0h_0', \theta_1(h)+\theta_1(h'), \dots, \theta_{m-1}(h)+\theta_{m-1}(h') ) = \iota(h) \iota(h').$$
Therefore, $\iota$ is a homomorphism.
For indeterminates $y=(y_1, y_2,\dots)$, define a polynomial $p_i(y) \in \Q[y_1, \dots, y_i]$ by
$$ p_i(y) = \sum_{\substack{ k_1,\dots,k_i \geq 0 \\ k_1+2k_2+\cdots + ik_i=i}} \dfrac{1}{k_1! \cdots k_i !} y_1^{k_1} \cdots y_i^{k_i}, \quad p_0(y):=1.$$
Then, it is well known that (cf. \cite[p. 30, Example 11]{Macdonald})
$$ \exp( y_1T + y_2T^2 + \cdots ) = \sum_{i \geq 0} p_i(y)T^i. $$
Thus, if $ y = (\theta_1(x), \theta_{2}(x), \dots )$, then
$$ x_0 \sum_{i \geq 0} p_i(y)T^i = x_0 + x_1T + x_{2}T^{2} + \cdots .$$
On the other hand, if $x = (y_0, y_0p_1(y),  y_0 p_{2}(y), \dots )$, then 
$$ \sum_{i \geq 1} \theta_i (x) T^i =  y_1T + y_{2}T^{2} + \cdots. $$
Therefore, the inverse morphism $\iota^{-1}$ is given by
$$ (a_0, a) \longmapsto [a_0, a_0p_1(a), \dots , a_0p_{m-1}(a)] \quad (a_0 \in \k^*, a \in \k^{m-1}).$$

\ref{char of J(m)-2} This follows from \ref{char of J(m)-1} and the fact $\widehat{\k} = \{ \psi_a \mid a \in \k\}$.
\end{proof}

Let $\D = (N_1, \dots, N_l)$ be a partition of $n$, where $N_1 \leq \cdots \leq N_l$, and suppose that $p \geq N_l$.
Define a group $H_\D \subset GL_n(\k)$ by
$$ H_\D = \{ {\rm diag}(h_1, \dots, h_l) \mid h_i \in J(N_i) \} \cong \prod_{i=1}^l J(N_i).$$
By Proposition \ref{char of J(m)} \ref{char of J(m)-1}, we have the isomorphism
\begin{equation}
 \widetilde\iota \colon H_\D \overset{\cong}{\longrightarrow} \prod_{i=1}^l (\k^* \times \k^{N_i-1}). \label{isom iota tilde}
\end{equation}
Note that the character group is
$$ \widehat{H_\D} = \left\{ \chi := (\chi_1, \dots, \chi_l) \mid \chi_i \in \widehat{J(N_i)} \right\}. $$
Let $z \in M(d,n;\k)$ be a matrix, and write
\begin{equation} 
z = (z^{(1)}, \dots, z^{(l)}), \quad z^{(i)} = (z_0^{(i)}, \dots, z_{N_i-1}^{(i)} ),\label{z rep. 1}
\end{equation}
where $z_j^{(i)}$ are the columns of $z$.
For $s = (s_1, \dots, s_d) \in \k^d$ and $\chi_i \in \widehat{J(N_i)}$, put
$$ \chi_i(sz^{(i)}) = \chi_i([ sz_0^{(i)}, \dots, sz_{N_i-1}^{(i)}]),$$
where if $sz_0^{(i)} = 0$, then $\chi_i(sz^{(i)}) :=0$.
For $\chi = (\chi_1, \dots, \chi_l) \in \widehat{H_\D}$, define
$$ \chi(sz) = \prod_{i=1}^l \chi_i(sz^{(i)}). $$

\begin{dfn}\label{def. of gen. HGF}
Let $\D$ be a partition of $n$ and $\chi \in \widehat{H_\D}$.
For $z \in M(d,n;\k)$, define
$$ \Phi_\D(\chi;z) = \sum_{ s \in \k^d} \chi(sz) \in \ol\Q. $$
We call $ \Phi_\D(\chi;z)$ {\it the general hypergeometric function} {\it over }$\k$.
\end{dfn}
\begin{rem}
Gel'fand-Graev-Retakh \cite[Section 11]{GGR} defined hypergeometric functions associated with Grassmannians over finite fields.
Their functions coincide with our functions $\Phi_{(1,\dots,1)}$.
\end{rem}

We now establish some properties of general hypergeometric functions.
First, we can obtain the following proposition.
\begin{prop}\label{action of GL and H on Phi} Let $\chi \in \widehat{H_\D}$.
\begin{enumerate}
\item \label{action of GL and H on Phi-GL} For $g \in GL_d(\k)$, we have
$$ \Phi_\D(\chi;gz) = \Phi_\D(\chi;z) . $$

\item \label{actino of GL and H on Phi-H} For $h \in H_\D$, we have
$$ \Phi_\D(\chi;zh) = \chi(h) \Phi_\D(\chi;z).$$ 
\end{enumerate}
\end{prop}
This proposition can be proved by direct computations of the character sums, but in this paper, we will give a geometric proof (see the proof following Theorem \ref{N of X_D(z)}).
%\begin{proof}
%\ref{action of GL and H on Phi-GL} This follows by the change of variables $s \mapsto sg^{-1}$.
%\ref{actino of GL and H on Phi-H} One shows $\chi(szh) = \chi(h)\chi(sz)$ since $\chi$ is a character of $H_\D$. Thus, we obtain the proposition.
%\end{proof}

Secondly, we see an analogue of a symmetry for the general hypergeometric function over $\C$ (cf. \cite{K-K}).
Define the polynomial $\mu_{i,j}(y) \in \Z[y_1, \dots, y_j]$ by
$$ \mu_{i,j}(y) 
= \begin{cases} 
      0 & ( i > j ) \vspace{3pt} \\
      1 & ( i = j = 0 ) \vspace{3pt}\\
      \displaystyle \sum_{\substack{ j_1, \dots, j_i \geq 1 \\ j_1 + \cdots + j_i = j}} y_{j_1} \cdots y_{j_i} & ({\rm otherwise}).
\end{cases}$$
Note that $ \mu_{0,j} = 0$ when $0 <j$.
It is known that (cf. \cite[(4.2)]{K-K}), for an indeterminate $T$,
\begin{equation}
 \sum_{j \geq 0} \mu_{i,j}(y)T^j = ( y_1T + y_2T^2 + \cdots )^i.  \label{eq. of mu(c)}
\end{equation}

For a fixed positive integer $m$, define upper triangular matrices
$$ \mu(y) = (\mu_{i,j}(y))_{0\leq i,j \leq m-1} ,\quad  \mu(y)' = (\mu_{i,j}(y))_{1 \leq i,j \leq m-1}. $$
For indeterminates $x = (x_0, x_1, \dots, x_{m-1})$, we have (\cite[(5.9)]{K-K})
\begin{equation}\label{theta mu = theta(mu)}
 (x_0,\theta_1(x), \dots, \theta_{m-1}(x))\mu(y) = (x_0, \theta_1(x'), \dots, \theta_{m-1}(x')),
\end{equation}
where $x' = x\mu(y)$ (note that $x'_0 = x_0$ since the first column of $\mu(y)$ is ${}^t(1,0,\dots, 0)$).
Note that the diagonal components of $\mu(y)$ are $1, y_1, \dots, y_1^{m-1}$.
Define the set 
$$ W(m) = \left\{ \mu(c) \mid c = (c_1, \dots, c_{m-1}) \in \k^{m-1}, c_1 \neq 0 \right\} \subset GL_m(\k),$$
where $W(1)=\{1\}$.
One shows $\mu(1,0,\dots,0) = I_m$.
For $a=(a_1, a_2,\dots )$ and $b=(b_1,b_2, \dots) $, define $c=(c_1, c_2 , \dots)$ by
$$ c_k = \sum_{j \geq 1} a_j \mu_{j,k}(b).$$
When $i \geq 1$, noting \eqref{eq. of mu(c)}, we have 
\begin{align*}
\sum_{k \geq 1} \mu_{i,k}(c) T^k 
& = \Big( \sum_{k \geq 1} c_k T^k \Big)^i \\
& = \Big( \sum_{j \geq 1} a_j \sum_{k \geq 1} \mu_{j,k}(b) T^k \Big)^i\\
& = \Big( \sum_{j \geq 1} a_j ( b_1T + b_2 T^2 + \cdots)^j \Big)^i\\
& =  \sum_{j \geq 1} \mu_{i,j}(a) (b_1T + b_2T^2 +  \cdots)^j \\
& = \sum_{j \geq 1} \mu_{i,j}(a) \sum_{k \geq 1} \mu_{j,k}(b) T^k 
= \sum_{k \geq 1} \Big( \sum_{j \geq 1} \mu_{i,j}(a) \mu_{j,k}(b) \Big) T^k.
\end{align*}
Hence, for $a=(a_1, \dots, a_{m-1})$ and $b=(b_1, \dots, b_{m-1})$, 
$$ \mu_{i,k}(c) = \sum_{j=1}^{m-1} \mu_{i,j}(a)\mu_{j,k}(b) \quad (1 \leq i,k\leq m-1).$$
Therefore, we have $\mu(a)\mu(b) = \mu(c)$, and hence, $W(m)$ is a subgroup of $GL_m(\k)$.

From now on, we suppose that the partition $\D$ of $n$ is of the following form:
\begin{equation}
\Delta = (\overbrace{n_1, \dots, n_1}^{p_1}, \dots , \overbrace{n_k, \dots, n_k}^{p_k})\quad (n_1 < n_2 < \cdots < n_k\leq p).\label{refinment partition}
\end{equation}
Then, $H_\D = \prod_i J(n_i)^{p_i}$.
For each $i = 1, \dots, k$, define
$$ \mathcal{P}_i = \left\{ \widetilde{P}_\sigma := (E_{\sigma(1)}, \dots, E_{\sigma(p_i)}) \in GL_{n_ip_i}(\k) \middle|\, \sigma \in \S_{p_i} \right\},$$
where
$$  E_j := {}^t (E_{j1}, \dots , E_{jp_i}), \quad E_{jk}=\begin{cases} O_{n_i} &(k \neq j) \\ I_{n_i} & (k=j). \end{cases}$$
If $n_i = 1$, then $\widetilde P_\sigma = P_\sigma$.
As with permutation matrices, one shows $\widetilde{P}_\sigma \widetilde{P}_{\sigma'} = \widetilde P_{\sigma \sigma'}$ and ${}^t \widetilde P_\sigma = \widetilde P_{\sigma^{-1}} = \widetilde P_\sigma^{-1}$.
Hence, the set $\mathcal{P}_i$ is a subgroup of $GL_{n_ip_i}(\k)$ and is isomorphic to the symmetric group $\S_{p_i}$.
The group $\mathcal{P}_i$ acts on the group
$$ W(n_i)^{p_i} := \left\{{\rm diag}(\mu(c_1), \dots, \mu(c_{p_i})) \middle| \, \mu(c_j) \in W(n_i) \right\} \subset GL_{n_ip_i}(\k)$$
by 
\begin{align*} 
\widetilde{P}_\sigma \cdot {\rm diag}(\mu(c_1), \dots, \mu(c_{p_i})) 
& := \widetilde{P}_\sigma {\rm diag}(\mu(c_1), \dots, \mu(c_{p_i})) \widetilde{P}_\sigma^{-1} \\
& = {\rm diag}(\mu(c_{\sigma^{-1}(1)}), \dots , \mu(c_{\sigma^{-1}(p_i)}) ).
\end{align*}
Define the group
$$ W_\Delta = \prod_{i=1}^k \Big( W(n_i)^{p_i} \rtimes \mathcal{P}_i \Big) \subset GL_n(\k).$$
For a character $(\a, a) \in \widehat{J(n_i)}$ ($a \in \k^{n_i-1}$) and an element $\mu(c) \in W(n_i)$, define
$$ (\a, a){}^t\mu(c) = (\a, a{}^t\mu(c)') \in \widehat{J(n_i)}.$$
Each element $w_i \in W(n_i)^{p_i} \rtimes \mathcal{P}_i $ can be uniquely written as 
$$ w_i = {\rm diag}(\mu(c_1), \dots, \mu(c_{p_i}) ) \widetilde P_\sigma,$$
where $\mu(c_j) \in W(n_i)$ and $\sigma \in \S_{p_i}$.
For a character $\chi_i:=(\chi_{i,1}, \dots, \chi_{i,p_i}) \in \widehat{J(n_i)^{p_i}}$, define
$$ \chi_i {}^tw_i = (\chi_{i,\sigma^{-1}(1)}{}^t\mu(c_{1}), \dots , \chi_{i,\sigma^{-1}(p_i)}{}^t\mu(c_{p_i}) ) \in \widehat{J(n_i)^{p_i}}.$$
Applying this construction component-wise, define $\chi {}^tw =(\chi_i {}^t w_i)_i \in \widehat{H_\Delta}$ for $\chi =(\chi_1,\dots,\chi_k) \in \widehat{H_\Delta}$ and $w = {\rm diag}(w_1,\dots,w_k) \in W_\Delta$.
The following is a finite field analogue of the symmetry of general hypergeometric functions over $\C$ (\cite[Theorem 5.3]{K-K}).
\begin{thm}\label{sym of Phi}
Let  $\chi \in \widehat{H_\Delta}$.
For $w \in W_\Delta$, we have 
$$ \Phi_\Delta (\chi {}^t w ; z) = \Phi_\Delta(\chi ;zw).$$
%In particular, if we put $\mathcal{P} = \prod_{i=1}^k \mathcal{P}_i \subset W_\Delta$, then for $w \in \mathcal{P}$, 
%$$ \Phi_\Delta( \chi w ; zw) = \Phi_\Delta(\chi;z).$$
\end{thm}
We can prove this theorem by direct computations of the character sums, but in this paper, we will give a geometric proof (see the proof following Theorem  \ref{N of X_D(z)}).
\begin{rem}
For $\D=(1,\dots,1)$, note that $W_\D=\{ P_\sigma \mid \sigma \in \S_n\}$ and the formulas in Proposition \ref{action of GL and H on Phi} and Theorem \ref{sym of Phi} are given by Gel'fand-Graev-Retakh \cite[Section 11]{GGR}.
\end{rem}

\subsection{Particular cases}\label{particular cases}
In this subsection, we discuss several special cases to explain how the general hypergeometric functions recover the classical hypergeometric functions over $\k$.
When $(d,n) = (2,4)$, the general hypergeometric functions can be written in terms of one-variable functions ${}_2F_1, {}_1F_1$ and ${}_0F_1$, and the symmetry induces certain well-known formulas for each of these functions over $\k$ (see Remarks \ref{rem of Phi(1,1,1,1)} and \ref{rem of Phi(1,1,2)}).
When $(d,n)=(2,5)$, the general hypergeometric functions can be expressed in terms of Appell's $F_1$ and Humbert's functions, and the symmetry induces certain formulas for each of these functions over $\k$ (see Remark \ref{rem of Phi(1,1,...,1)}).
We now treat the case $(d,n)=(2,4)$ in this subsection.

For $z \in M(d,n;\k)$ and a fixed partition $\D$ of $n$, we denote $z \sim z'$ when $z' = gzh$ for some $g \in GL_d(\k)$ and $h \in H_\Delta$.
Over the complex numbers, the following confluent diagram for classical special functions is known (cf. \cite[Figure 1 and Subsection 2.1]{Ohyama}):
\begin{equation*}
\begin{diagram}
\node{} \node{} \node{{}_0F_1} \node{}\\
\node{{}_2F_1} \arrow{e,t}{\rm c.o.} \node{{}_1F_1} \arrow{ne,l}{\rm c.o.} \arrow{se,r}{\rm c.o.} \node{} \node{\mbox{Airy}}\\
\node{} \node{} \node{\mbox{Hermite-Weber}} \arrow{ne,r}{\footnotesize{\mbox{c.o.}}} \node{}
\end{diagram}
\end{equation*}
Here, c.o. denotes a limiting procedure known as confluence.
These functions arise from general hypergeometric functions $\Phi_\Delta(z)$ over $\C$ with $z \in M(2,4;\C)$.
The functions ${}_2F_1, {}_1F_1, {}_0F_1$, Hermite-Weber and Airy correspond to $\Delta=(1,1,1,1)$, $(1,1,2), (2,2), (1,3)$ and $(4)$, respectively.

We now describe an analogous correspondence for the first three functions over $\k$  as follows. 
For $z=(z_{ij}) \in M(2,n;\k)$, put
$$\quad [i\, j]:= \det\bmat{z_{1i}& z_{1j}\\ z_{2i}&z_{2j}}.$$
First, let $\Delta=(1,1,1,1)$.
We have 
$$z \sim z'= \bmat{1&1&1&0 \\ -1&-\l&0&1} \quad (\l \in \k^*),$$
when $[i\, 3], [i\, 4], [3\,4]\neq 0$ for $i =1,2$.
For $\chi = (\a_1,\a_2,\a_3,\a_4) \in \widehat{H_\Delta}$ ($\a_i \in \khat$ with $\a_1,\a_2 \neq \e$), we have
\begin{align*}
\Phi_\Delta(\chi;z') 
& = \sum_{s \in \k^2} \a_1(s_1-s_2)\a_2(s_1-\l s_2)\a_3(s_1)\a_4(s_2)\\
& = \sum_{s_2' \in \k^*} \a_2(1-\l s_2') \a_4(s_2') \a_1(1-s_2') \sum_{s_1} \a_1\a_2\a_3\a_4(s_1)\\
& = \delta(\a_1\a_2\a_3\a_4)(q-1) \sum_{s_2'} \a_2(1-\l s_2') \a_4(s_2') \a_1(1-s_2')\\
& = -\delta(\a_1\a_2\a_3\a_4)(q-1) j(\a_1,\a_4) \hFF{2}{1}{\ol{\a_2}, \a_4}{\a_1\a_4}{\l}.
\end{align*}
Here, we put $s_2' = s_2/s_1$ and, in the last equality, we used \eqref{1F0 int.} and Proposition \ref{iteration} (ii).
\begin{rem}\label{rem of Phi(1,1,1,1)}
In this case, $W_\D = \{ P_\sigma \mid \sigma \in \S_4\}$.
Suppose $ [i\, j] \neq 0 \ (i \neq j)$.
The symmetry in Theorem \ref{sym of Phi} is
$$ \Phi_\D(\chi {}^tP_\sigma;z') = \Phi_\D(\chi; z'P_\sigma) \quad (\sigma \in \S_4),$$
where $\chi=(\a_i)_i \in \widehat{H_\D}$ and $\chi {}^tP_\sigma = (\a_{\sigma^{-1}(i)})_i$.
This symmetry induces a transformation formula between ${}_2F_1(\l)$ and ${}_2F_1(\l_\sigma)$, where $z'P_\sigma \sim \bmat{1 & 1 & 1 & 0 \\ -1 & -\l_\sigma & 0 & 1}$.
In fact, the transformation formulas are the relations among 24 ${}_2F_1$-functions over $\k$ \cite[Corollary 3.16]{Otsubo}.
Indeed, all the relations in \cite[Corollary 3.16]{Otsubo} were obtained by composing three formulas \cite[Theorems 3.14 and 3.15]{Otsubo}, which are induced by the symmetry for $\sigma = (1\, 3)(2\, 4), (1\, 4)$ and $(1\, 3)$ and the shift of parameters.
\end{rem}
%For example, when $\sigma = (1\, 3) \in \S_4$, then the formula above induces
%\begin{equation}
% \hF{\a_1, \a_2}{\a_1\a_2\b_1, \ol{\b_2}}{\l} = \a(-1)\dfrac{j(\a_1,\b_1)}{j(\a_1, \ol{\a_1\a_2\b_1})} \b_2\Big( \dfrac{\l}{\l-1} \Big)\hF{\a_1,\a_2}{\ol{\b_1},\ol{\b_2}}{1-\l}, \label{example of Kummer 24}
%\end{equation}

Secondly, let $\Delta=(1,1,2)$ and suppose $[i\, j], [3\, 4] \neq0$ for $1 \leq i \neq  j \leq 3$. 
We have $z \sim z'=\bmat{-1 & 1 & 0 & -\l \\ 1 & 0 & 1 & 0}$ ($\l \in \k^*$).
For $\chi = (\a_1,\a_2,\a_3,a) \in \widehat{H_\Delta}$ ($\a_i \in \khat, a \in \k^*$ with $\a_1 \neq \e$), we have
\begin{align*}
\Phi_\Delta(\chi;z') 
& = \sum_{s \in \k^2} \a_1(s_2-s_1) \a_2(s_1) \a_3(s_2) \psi_a(-\l s_1/s_2)\\
& = \delta(\a_1\a_2\a_3)(q-1) \sum_{s_1'} \psi( -a\l s_1')  \a_2(s_1') \a_1(1-s_1') \\
& = -\delta(\a_1\a_2\a_3)(q-1) j(\a_1,\a_2) \hFF{1}{1}{\a_2}{\a_1\a_2}{a\l}_\psi.
\end{align*}
Here, we put $s_1' = s_1/s_2$ and used \eqref{0F0 int} and Proposition \ref{iteration} (ii).
\begin{rem}\label{rem of Phi(1,1,2)}
We have 
$$ W_\D = \left\{ w_{\sigma, c}:= \bmat{P_\sigma & O_2 \\ O_2 & \mu(c)} \in GL_4(\k)\middle|\, \sigma \in \S_2, c \in \k^*\right\},\quad \mu(c) = \bmat{1 & 0 \\ 0 & c}.$$
When $\sigma$ is the permutation $(1\,2)$, the symmetry $\Phi_\D(\chi {}^tw_{\sigma,c};z') = \Phi_\D(\chi  ; z' w_{\sigma, c})$ induces an analogue \eqref{Kummer prod} of Kummer's first product formula.
Here, note that $\chi {}^t w_{\sigma,c}= (\a_{\sigma(1)}, \a_{\sigma(2)}, \a_3, ca)$.
\end{rem}
%(cf. \cite[Theorem 6.1 (i)]{Otsubo}):
%\begin{equation}
% \hFF{1}{1}{\a}{\b}{\l} = \psi(-\l)\hFF{1}{1}{\ol{\a}\b}{\b}{-\l}. \label{Kummer prod}
%\end{equation}
Finally, we consider the case when $\Delta=(2,2)$.
For $z = \bmat{1 & 0 & 0 & \l \\ 0 & -1 & 1 & 0}$ and $\chi = (\a_1, a_1, \a_2, a_2) \in \widehat{H_\Delta}$ ($a_i \in \k^*$), we have by putting $s_2' = -a_1s_2/s_1$,
\begin{align*}
\Phi_{\Delta}(\chi; z) 
& = \sum_{s}\a_1(s_1)\psi_{a_1}(-s_2/s_1) \a_2(s_2)\psi_{a_2}(\l s_1/s_2)\\
& = \ol{\a_2}(-a_1)\sum_{s_2'} \psi(-a_1a_2 \l/ s_2') \a_2(s_2') \psi( s_2') \sum_{s_1} \a_1\a_2(s_1)\\
& = -\d(\a_1\a_2)(q-1) \a_1(-a_1)g(\ol{\a_1}) \hFF{0}{1}{}{\a_1}{-a_1a_2\l}.
\end{align*}
Here, we used \eqref{0F0 int} and Proposition \ref{iteration} (iv).

\begin{rem}\label{rem of Phi(1,1,...,1)}
Similarly, when $d=2, n=5$, Appell's $F_1$ corresponds to $\Delta=(1,1,1,1,1)$, Humbert's functions $\Phi_1$ and $\Phi_2$ correspond to $\Delta= (1,1,1,2)$ and $\Phi_3$ corresponds to $\Delta = (1,2,2)$ (by Proposition \ref{action of GL and H on Phi}, one can check that $\Phi_1$ and $\Phi_2$ are essentially equivalent).
For $\D=(1,1,1,1,1)$, recall $W_\D=\S_5$.
When $[i\, j]\neq0$ for all $i \neq j$,  Theorem \ref{sym of Phi} induces $120$ transformation formulas for Appell's $F_1$ over $\k$.
The author has checked some of them, and it seems that these formulas are relations among finite field analogues of the $60$ $F_1$-functions listed in \cite[(110)--(121)]{Vidunas} and the trivial relation $ F_1\left( { \a\, ;\, \b_1,\b_2 \atop \c}; x,y \right) = F_1\left( { \a\, ;\, \b_2, \b_1 \atop \c}; y,x \right)$.
Some of these formulas were essentially obtained by Li-Li-Mao \cite[Theorems 3.2 and 3.3]{Li-Li-Mao}.
For $\D=(1,1,1,2)$ and $ (1,2,2)$, Theorem \ref{sym of Phi} induces transformation formulas for Humbert's functions over $\k$, which have not been obtained as far as the author knows.
\end{rem}

\begin{rem}
When $d =2, n \geq 5$ and $\D=(1, \dots, 1)$, one shows that if $[i\, j] , [n-1\, n]\neq 0$ for $i \in \{1, \dots, n-2\}$ and $  j\in\{n-1,n\}$, then  
$$ z \sim z' = \bmat{ 1 & 1 & 1 & \cdots & 1 & 1 & 0 \\ 
                     -1 & -\l_1 & -\l_2 & \cdots & -\l_{n-3} & 0 & 1} \quad (\l_i \in \k^*),$$
and the function $\Phi_\D(\chi; z')$  can be written in terms of Lauricella's $F_D^{(n-3)}(\l_1, \dots, \l_{n-3})$, similarly to the case when $n=4$.
Also Lauricella's $F_A$ and $F_B$ can be written in terms of general hypergeometric functions (see Subsection \ref{subsec. of X_A} and Remark \ref{rem for FA and FB}). 
On the other hand, Lauricella's $F_C$ does not appear to be expressible in terms of general hypergeometric functions, as far as the author knows. 
\end{rem}

%\subsubsection{$k=3, n=6$}

\section{Hypergeometric varieties}
Hereafter, put $N=q-1$.
Let $\ol\k$ be an algebraic closure of $\k$, and let $\k_r \subset \ol\k$ be the degree $r$ extension of $\k$ for $r \geq 1$.

\subsection{The number of rational points}
Let $X$ be an algebraic variety over $\k$ and suppose that a finite abelian group $G$ acts on $X$ over $\k$.
We write $X \otimes \k_r$ for the variety $X \times_{{\rm Spec}(\k)} {\rm Spec}(\k_r)$ over $\k$.
Let $ {\rm Frob} \colon X \rightarrow X$ be the $q$-Frobenius morphism.
For each character $\chi \in \widehat{G}$, {\it the number of rational points on }$X$ {\it associated with }$\chi$ is 
$$N(X;\chi) := \dfrac{1}{\#G} \sum_{g \in G} \chi(g) \Lambda_g,$$
where 
$$ \Lambda_g  := \#\{P \in X(\ol\k) \mid {\rm Frob}(P)=g\cdot P\}. $$

\begin{rem}
By the orthogonality of characters, we have
$$ \# X(\k) = \sum_{\chi \in \widehat{G}} N(X;\chi).$$
Accordingly, the congruence zeta function $Z(X;t)$ decomposes into the product of Artin $L$-functions $L(X,\chi;t)$ ($\chi \in \widehat{G}$) (cf. \cite{Serre}).
%When $X$ is smooth, by the Grothendieck-Lefschetz trace formula, we have
%$$ N(X;\chi) = \sum_{i=0}^{2\dim X} (-1)^i {\rm Tr}( {\rm Frob}^{-1} \mid H^i(X \otimes \ol \k, \ol{\Q_l})(\chi)),$$
%where $H^i(X \otimes \ol\k, \ol{\Q_l})(\chi)$ is the $\chi$-eigenspace of $l$-adic \' etale cohomology with compact support.
\end{rem}
For an extension $\k' / \k$, the Galois group ${\rm Gal} := {\rm Gal}(\k'/\k)$ acts on ${\rm Spec}(\k')$ over $\k$ by $a \mapsto e(a)$ on the structure sheaf $\mathcal{O}_{{\rm Spec}(\k')}$, where $e \in {\rm Gal}$.
One shows (\cite[(12)]{Serre}) 
\begin{equation*}
 N({\rm Spec} (\k'); \rho) = \rho(F),
\end{equation*}
where $\rho \in \widehat{{\rm Gal}}$ and $F \in {\rm Gal}$ is the Frobenius automorphism.
Furthermore, the group $G \times {\rm Gal}$ acts naturally on $X \otimes \k'$ over $\k$, and we have
\begin{equation}
N(X \otimes \k'; (\chi, \rho)) = \rho(F) N(X;\chi),\label{N of base change}
\end{equation}
for $(\chi, \rho) \in \widehat{G} \times \widehat{\rm Gal}$.

Let $Y$ be an algebraic variety equipped with an action of a finite abelian group $G'$ over $\k$.
Suppose that there is an isomorphism $f \colon X \rightarrow Y$ over $\k$ and a group isomorphism $\pi \colon G \rightarrow G'$ such that
$$ f \circ g  = \pi(g) \circ f $$
for all $g \in G$.
One shows the following lemma.
\begin{lem}\label{Comparison of N by isom}
 For $\chi \in \widehat{G'}$, we have
$$N(X; \pi^*\chi) = N(Y;\chi),$$
where $\pi^*\chi := \chi \circ \pi \in \widehat G$.
\end{lem}

\begin{rem}
The $(n-1)$-dimensional Fermat hypersurface 
$$ Fer_n \colon x_1^N + \cdots + x_n^N=1$$ 
has $(\k^*)^n$-action by $(\xi_i)_i\cdot (x_i)_i := (\xi_ix_i)_i$ for $(\xi_i)_i \in (\k^*)^n$.
On the other hand, the Artin-Schreier curve 
$$AS \colon t^q-t = z^N$$
has $(\k^* \times \k)$-action by $(\xi,a) \cdot (z,t) := (\xi z, t+a)$ for $(\xi,a) \in \k^*\times \k$.
We write $Fer_n^* = Fer_n-\{x_1\cdots x_n=0\}$ and $AS^* = AS-\{z=0\}$.
For $\a, \a_1, \dots, \a_n \in \khat$ and $\psi \in \widehat{\k}$, one shows the well-known relations 
\begin{equation}
 N(Fer_n^*; (\a_1, \dots, \a_n) ) = (-1)^{n-1}j(\a_1, \dots, \a_n) \label{N of Fer}
\end{equation}
and
$$ N(AS^*; (\a,\psi)) = -g(\a). $$
\end{rem}
%If $G' \subset G$ is a normal subgroup, $\pi\colon G \rightarrow G/G'$ is the natural projection and $\chi \in \widehat{G/G'}$, then (cf. \cite[2.3]{Serre})
%\begin{equation}
%N(X;\pi^*\chi) = N(X/G' ; \chi).\label{Comparison of N by quotient}
%\end{equation}

\subsection{General hypergeometric functions}\label{subsection of X_D,z}

Let $z \in M(d,n;\k)$ and let $\D=(N_1,\dots,N_l)$ be a partition of $n$, where $N_1 \leq \cdots \leq N_l$.
Write $z = (z^{(1)}, \dots, z^{(l)})$ as in \eqref{z rep. 1}.
%$$\D = (\overbrace{n_1, \dots, n_1}^{p_1}, \dots, \overbrace{n_k, \dots, n_k}^{p_k})$$
%be a partition of $n$, where $n_1 < n_2 < \cdots < n_k $.
%We write
%$$ z = (z^{(1)}, \dots, z^{(k)}),\quad z^{(i)} = (z^{(i,1)}, \dots, z^{(i, p_i)}) \in M(d, n_ip_i;\k) ,$$
%where $z^{(i,j)} = ( z^{(i,j)}_0, \dots , z^{(i,j)}_{n_i-1}) \in M(d, n_i; \k)$ and $z^{(i,j)}_k$ are the column vectors.
Put 
$$\ol\theta_i(x) : = x_0^i \theta_i(x)$$ 
for $x = (x_0, x_1, \dots )$ (then, $\ol\theta_i(x) \in \Q[x_0, \dots, x_{i}]$). 
Suppose that $p \geq N_l$.
Define an affine variety $X_{\D,z} \subset \mathbbm{A}^{n+d}$ over $\k$ by the following equations for $(t_{i}, u_{(i,j)})$ and $s=(s_1, \dots, s_d)$: 
$$ \begin{cases} 
t_{i}^N = sz^{(i)}_0 \vspace{3pt}\\
t_{i}^N (u_{(i,1)}^q -u_{(i,1)}) =\ol\theta_1(sz^{(i)}) \vspace{3pt}\\
\hspace{30pt}\vdots\\
t_{i}^{N (N_i-1)} (u_{(i,N_i-1)}^q - u_{(i,N_i-1)}) = \ol\theta_{N_i-1}(sz^{(i)}) \vspace{3pt}\\
t_{i} \neq 0
\end{cases}(1 \leq i \leq l).$$
For brevity, we denote the coordinates on $X_{\D,z}$ by $ \big( (t_{i}, u_{i}), s \big)$, where $u_{i} := (u_{(i,1)}, \dots, u_{(i,N_i-1)})$.
% and $(t_{(i,j)}, u_{(i,j)})$ means 
%$$ (t_{(1,1)}, u_{(1,1)}, \dots, t_{(1,p_1)}, u_{(1,p_1)}, \dots\dots , t_{(k,1)}, u_{(k,1)}, \dots, t_{(k,p_k)}, u_{(k,p_k)}).$$
The group $G_\D := \prod_{i=1}^l \big( (\k^*) \times \k^{N_i-1} \big)$ acts on $X_{\D,z}$, similarly to the $(\k^*\times\k)$-action on the Artin-Schreier curve.
The group $H_\D$ also acts on $X_{\D,z}$ through the isomorphism $\widetilde \iota \colon H_\D \rightarrow G_\D$ defined in \eqref{isom iota tilde}.

\begin{thm}\label{N of X_D(z)-1}
Let $\chi \in \widehat{H_\D}$ be a character.
We have
$$ N(X_{\D,z} ; \chi) = \Phi_\D(\chi;z). $$

\end{thm}

\begin{proof}
%Note that
%$$ \Phi_\Delta(\chi;z) = \prod_{i=1}^k \Big( \sum_{s \in \k^d} \chi_i(sz^{(i)}) \Big) =  \prod_{i=1}^k \prod_{j = 1}^{p_i} \Big( \sum_{s \in \k^d} \chi_{i,j}(sz^{(i,j)}) \Big),$$
%and 
Note that, for $\chi' \in \widehat{G_\D}$ such that $\chi = \chi' \circ \widetilde \iota \in \widehat{H_\D}$, clearly we have
\begin{equation}
N(X_{\D,z}; \chi) = N(X_{\D,z} ; \chi').\label{eq.2 in pr of N of X_D(z)}
\end{equation}
For $g = ( \xi_{i}, a_{(i,1)}, \dots , a_{(i,N_i-1)} )_{i} \in G_\D$, we have
\begin{align*}
\Lambda_g 
& =  \#\{ P \in X_{\D,z} (\ol\k) \mid {\rm Frob}(P) = g\cdot P\}\\
& =  \#\{ ( (t_{i}, u_{i}) ,s) \in X_{\D,z}(\ol\k) \mid t_{i}^N = \xi_{i},\,  u_{(i,j)}^q-u_{(i,j)} = a_{(i,j)}, \, s \in \k^d \}\\
& = \# G_\D \times \#\{ s \in \k^d \mid \xi_{i} = sz_0^{(i)}, a_{(i,j)} = \theta_j (sz^{(i)}) \}\\
& =  \# G_\D \times \#\{ s \in \k^d \mid g = \widetilde\iota([sz])\}.
\end{align*}
Thus, 
\begin{align*}
N(X_{\D,z} ; \chi') 
 = \dfrac{1}{\# G_\D} \sum_{g \in G_\D} \chi'(g) \Lambda_g 
 = \sum_{s \in \k^d} \chi' \circ \widetilde \iota ([sz])
 = \Phi_\D(\chi; z).
\end{align*}
Here, we put $ \chi' \circ \widetilde \iota ([sz])=0$ if $sz_0^{(i)}=0$ for some $i$.
Hence, $N(X_{\D,z} ; \chi) = \Phi_\D( \chi ; z)$ by \eqref{eq.2 in pr of N of X_D(z)}.
\end{proof}

For particular hypergeometric functions such as ${}_mF_n$, Appell-Lauricella and Humbert's functions, simpler varieties correspond to the functions as described in the following subsections.
\subsection{One-variable hypergeometric functions}\label{subsec. of var. for one-variable functions}
The contents of this subsection are due to Otsubo (private communications).
I would like to thank him for his permission to include his results in this paper.

For the function $F(\l):=\d(1-\l)$ ($\l \in \k$), one shows
\begin{equation}
 \widehat{F}(\nu) = 1 \quad (\nu \in \khat).\label{F empty = delta}
\end{equation}
Let $0 \leq  m \leq n$ be integers and put $l=n-m$.
%One shows
%$$ \hF{\a_1,\dots, \a_m}{\b_1, \dots, \b_n}{\l} = \hF{\ol{\b_1}, \dots, \ol{\b_n}}{\ol{\a_1}, \dots, \ol{\a_m}}{(-1)^{m-n}\dfrac{1}{\l}}.$$
%Hence, we only consider the case when $m\leq n$, .
For $\l \in \k^*$, let ${}_mX_{n,\l} \subset Fer_2^{m} \times AS^{l}$ be an affine variety over $\k$  defined by the equations
\begin{align*}
\begin{cases}
 x_i^N + y_i^N = 1 & (i = 1, \dots, m) \vspace{3pt}\\
 t_j^q - t_j = z_j^N & (j=1, \dots, l) \vspace{3pt}\\
 (-1)^n \l \prod_{i=1}^m x_i^N = \prod_{i=1}^m y_i^N \prod_{j=1}^l z_j^N\vspace{5pt}\\
 \prod_{i=1}^m x_iy_i \prod_{j=1}^l z_j \neq 0. 
\end{cases}
\end{align*}
We denote the coordinates of ${}_mX_{n,\l}$ by 
$$(x_i, y_i, z_j, t_j) := (x_1, \dots, x_m, y_1, \dots, y_m, z_1, \dots, z_l, t_1, \dots, t_l).$$
One shows that the rank of the Jacobian matrix of ${}_mX_{n,\l}$ is $n+1$ at any point, and hence, the variety ${}_mX_{n,\l}$ is smooth and $\dim {}_mX_{n,\l} = n-1$ for all $n \geq 1$.
The finite abelian group $G := (\k^*)^{2m+l} \times \k^l$ acts on ${}_mX_{n,\l}$ over $\k$ by 
$$ (\xi_i, \xi_i', \zeta_j, a_j)\cdot (x_i, y_i, z_j, t_j) := (\xi_i x_i, \xi_i'y_i, \zeta_j z_j, t_j + a_j)\quad ((\xi_i, \xi_i', \zeta_j, a_j) \in G). $$
%Write $\bm \psi$ for $(\psi, \dots, \psi) \in \widehat{\k}^l$.

When $m=n$, Otsubo expressed $N({}_mX_{n,\l};\chi)$ ($\chi \in \widehat G$) in terms of functions  ${}_mF_{m-1}(\l)$ over $\k$, and he suggested using the Artin-Schreier curve to relate confluent type ${}_mF_{n}$ (i.e. $n \neq m-1$) to the number of rational points.
The following is the result.
\begin{thm}\label{N of X_mn}
Let $\chi = ((\a_i)_{i=1}^m, (\b_i)_{i=1}^m, (\c_j)_{j=1}^l, (\psi_{c_j})_{j=1}^l) \in \widehat{G}$, where $c_j \in \k^*$, and put $ c = \prod_{j=1}^l c_j$.
If $\a_i\b_i \neq \e$ for all $i$, then  we have
\begin{align*}
 N&({}_mX_{n,\l}; \chi)  \\
 & = (-1)^{n+1}\Big( \prod_{j=1}^l g(\c_j)\ol{\c_j}(c_j)\prod_{i=1}^m j(\a_i, \b_i) \Big) 
  \hF{\a_1,\dots,\a_m}{\ol{\b_1},\dots,\ol{\b_m}, \ol{\c_1},\dots,\ol{\c_l}}{ c \l}_\psi.
\end{align*}
\end{thm}

\begin{proof}
For $ g = (\xi_i, \xi_i', \zeta_j, a_j) \in G$, define $\d'(g) \in \{0,1\}$ by $\d'(g)=1$ if and only if $g$ satisfies
$$ \begin{cases} \xi_i + \xi_i' =1 & (i=1, \dots, m) \vspace{3pt} \\ a_j = \zeta_j & (j=1, \dots, l) \\ (-1)^n \l \prod_{i=1}^m \xi_i = \prod_{i=1}^m \xi_i' \prod_{j=1}^l \zeta_j. \end{cases} $$
If we put $u_i = x_i^N, v_i = y_i^N, w_j = z_j^N$ and $s_j = t_j^q-t_j$, we have
\begin{align*}
\Lambda_g 
& = \# \{ (x_i, y_i, z_j, t_j) \in {}_mX_{n,\l}(\ol\k) \mid (x_i^N, y_i^N, z_j^N,t_j^q-t_j) = g \}\\
& = \# G \times \left\{ (u_i, v_i, w_j, s_j) \in G\,  \middle| \, {\d'(u_i, v_i, w_j, s_j)=1, \atop (u_i,v_i,w_j,s_j) = g}\right\}\\
& = \# G \times  \d'(g).
\end{align*}
Therefore, (by putting $\zeta_j' = c_j \zeta_j$)
\begin{align*}
& N({}_mX_{n,\l}; \chi) \\
& = \sum_{g \in G} \chi(g) \d'(g)\\
& = \sum_{(\xi_i, \xi_i', \zeta_j,a_j) \in G} \d'(\xi_i, \xi_i', \zeta_j, a_j) \prod_{i=1}^m \a_i(\xi_i)\b_i(\xi_i')\prod_{j=1}^l \c_j(\zeta_j) \psi(c_ja_j) \\
& = ( \prod_j \ol{\c_j}(c_j) ) \sum_{\xi_i , \zeta_j' \in \k^*} \d( 1- \l \prod_i \dfrac{\xi_i}{\xi_i-1} \prod_j \dfrac{c_j}{(-\zeta_j')}) \prod_i \a_i(\xi_i) \b_i(1-\xi_i) \prod_j \c_j(\zeta_j') \psi(\zeta_j').
\end{align*}
Letting $\omega_i = \xi_i/(\xi_i-1)$, the last right-hand side above is equal to
\begin{align*}
&  \Big( \prod_i \a_i(-1) \prod_j \ol{\c_j}(c_j) \Big)  \sum_{\omega_i, \zeta_j' \in \k^*} \d(1-c\l\dfrac{ \prod_i \omega_i}{\prod_j (-\zeta_j')}) \prod_i \a_i(\omega_i)\ol{\a_i\b_i}(1-\omega_i) \prod_j \c_j(\zeta_j')\psi(\zeta_j').
\end{align*}
Thus, we obtain the theorem by \eqref{F empty = delta} and Proposition \ref{iteration} (iii) and (iv),
where noting that $\a_i(-1) j(\a_i, \ol{\a_i\b_i}) = j(\a_i,\b_i)$ by \eqref{J=G} and \eqref{Gauss sum thm}.
\end{proof}

\begin{rem}\label{rem of decomposition of mXn}
Let $\k' \supset \k$ be an extension such that $\{ \tau \mid \tau^N = (-1)^n \l \} \subset \k'$.
We can decompose ${}_mX_{n,\l} \otimes \k'$ into the following disjoint union:
$$ {}_mX_{n,\l} \otimes \k' = \bigsqcup_{\tau^N = (-1)^n\l} {}_mX_{n,\l}^\tau,$$
where
$$ {}_mX_{n,\l}^\tau := \{ (x_i, y_i, z_j, t_j) \in {}_mX_{n,\l} \otimes \k' \mid \tau \prod_{i=1}^m x_i = \prod_{i=1}^m y_i \prod_{j=1}^l z_j \}.$$
\end{rem}

\begin{rem}$ $
\begin{enumerate}
\item Let $Y_\l$ be an affine hypersurface defined by the equation 
$$v^N = (1-\l u_1\cdots u_{m-1})^a \prod_{i=1}^{m-1} u_i^{b_i} (1-u_i)^{c_i},\ v \neq 0,$$
where $a, b_i, c_i \in\Z_{>0}$.
Koblitz \cite[Theorem 3 and Remark 2]{Koblitz} computed the number of $\k$-rational points on $Y_\l$ and expressed it in terms of his ${}_mF_{m-1}$-function over $\k$.
The following morphism $ {}_mX_{m,\l}  \otimes \k_N \rightarrow Y_\l \otimes \k_N $ relates ${}_mX_{m,\l}$ and $Y_\l$: 
$$  u_i = -\Big(  \dfrac{x_i}{y_i} \Big)^N , \quad v = x_m^{-a} \prod_{i=1}^{m-1} \dfrac{(\sqrt[N]{-1} x_i)^{b_i}} {y_i^{b_i+c_i}} .$$

\item Let $C_\l$ be the affine curve defined by $(1-u^N)(1-v^N) = \l u^Nv^N$.
Asakura-Otsubo \cite[Theorem 4.2]{Asakura-Otsubo} express the number of rational points on $C_\l$ in terms of $\hFF{2}{1}{\a,\b}{\e}{\l}$.
Clearly, we have the projection
$$ {}_2X_{2,\l} \longrightarrow C_\l \, ;\, (x_i,y_i) \longmapsto (u,v) = (x_1,x_2). $$
\end{enumerate}
\end{rem}

\begin{rem}\label{rem on inverse relation}
When $m > n$, we have to consider ${}_nX_{m,(-1)^{m-n}/\l}$ since
$$\hF{\a_1,\dots, \a_m}{\b_1, \dots, \b_n}{\l} = \hF{\ol{\b_1}, \dots, \ol{\b_n}}{\ol{\a_1}, \dots, \ol{\a_m}}{\dfrac{(-1)^{m-n}}{\l}}.$$
This identity can be easily checked by \eqref{Poch formula}.
\end{rem}

\subsection{Appell-Lauricella functions}\label{subsec. of var. for AL functions}
We can also define corresponding varieties for Appell-Lauricella functions.
Let $\bm\l=(\l_1, \dots, \l_n) \in (\k^*)^n$. 
Define affine varieties $X_{F_D^n,\bm\l} \subset Fer_2^{n+1}$, $X_{F_A^n,\bm\l}\subset Fer_{n+1} \times Fer_2^n$  and $X_{F_C^n,\bm\l}\subset Fer_{n+1}^2$ by the equations
$$ X_{F_D^n,\bm\l} \colon \begin{cases} x_i^N + y_i^N = 1 &( i = 0, \dots, n) \vspace{3pt} \\ \l_i x_0^Nx_i^N = y_0^Ny_i^N & ( i= 1, \dots, n) \vspace{3pt} \\ \prod_{i=0}^n x_iy_i \neq 0, \end{cases}$$

$$ X_{F_A^n,\bm\l} \colon \begin{cases} x_0^N + \cdots +x_n^N=1 \vspace{3pt} \\ y_i^N + z_i^N = 1 & (i=1, \dots, n) \vspace{3pt} \\ \l_i x_0^N y_i^N = x_i^N z_i^N & ( i=1, \dots, n) \vspace{3pt} \\ \prod_{i=0}^n x_i \prod_{i=1}^n y_iz_i \neq 0, \end{cases}$$
and
$$ X_{F_C^n,\bm\l} \colon \begin{cases} x_0^N + \cdots + x_n^N = 1 \vspace{3pt} \\ y_0^N + \cdots + y_n^N = 1 \vspace{3pt} \\ \l_i x_0^Ny_0^N = x_i^N y_i^N & (i = 1, \dots, n) \vspace{3pt} \\ \prod_{i=0}^n x_iy_i \neq 0. \end{cases}$$
When $n=2$, we write respectively $X_{F_1,\bm\l}$, $X_{F_2,\bm\l}$ and $X_{F_4,\bm\l}$ for these varieties.
%Similarly to ${}_mX_{n,\l}$, these varieties are smooth, $\dim X_{F_D^n,\bm\l} = 1$ and $\dim X_{F_A^n,\bm\l} = \dim X_{F_C^n,\bm\l} = n$.
As in the case of ${}_mX_{n,\l}$, the groups $(\k^*)^{2n+2}$, $(\k^*)^{3n+1}$ and $(\k^*)^{2n+2}$ act on $X_{F_D^n,\bm\l}$, $X_{F_A^n,\bm\l}$ and $X_{F_C^n,\bm\l}$, respectively.
\begin{rem}\label{rem for FA and FB}
As in Remark \ref{rem on inverse relation}, Lauricella's function $F_B^{(n)}$ is essentially equal to $F_A^{(n)}$.
Therefore, we omit $X_{F_B^n,\bm\l}$ in this paper.
\end{rem}

Note that the following is an involution:
$$
\begin{array}{rccc}
        & S            &\longrightarrow& S            \\
        & \rotatebox{90}{$\in$}&               & \rotatebox{90}{$\in$} \\
        & (x_i)_i                  & \longmapsto   & \Big( -\dfrac{x_i}{1-\sum_{j=1}^n x_j} \Big)_i,
\end{array}
$$
where $S := (\k^*)^n - \{(x_i)_i \mid \sum_i x_i = 1\}$.
By this involution and the same argument as in the proof of Theorem \ref{N of X_mn}, we obtain the following theorem.
\begin{thm}\label{N of X_Lauricella}
$ $
\begin{enumerate}
\item \label{N of X_Lauricella-D} Let $\chi = ((\a_i)_{i=0}^n, (\b_i)_{i=0}^n) \in (\khat)^{2(n+1)}$. 
If $\a_i\b_i \neq \e $ for all $i =0, \dots, n$, then we have
\begin{align*}
& N(X_{F_D^n,\bm\l};\chi) 
 = -\big( \prod_{i=0}^n j(\a_i, \b_i) \big) \FD{n}{\a_0 \, ;\, \a_1, \dots, \a_n}{\ol{\b_0}\, ;\, \ol{\b_1}, \dots, \ol{\b_n}}{\bm\l}.
\end{align*}

\item \label{N of X_Lauricella-A} Let $\chi =( (\a_i)_{i=0}^n, (\b_i)_{i=1}^n, (\c_i)_{i=1}^n ) \in (\khat)^{3n+1}$.
If $\e \not\in \{ \a_0\cdots \a_n , \b_i\c_i\, (i=1,\dots,n) \}$, then we have
\begin{align*}
 & N(X_{F_A^n,\bm\l} ; \chi) \\
 & = (-1)^n j(\a_0,\dots,\a_n)\big( \prod_{i=1}^n j(\b_i,\c_i) \big) \FA{n}{\a_0\, ;\, \b_1,\dots,\b_n}{\ol{\a_1},\dots, \ol{\a_n}\, ;\, \ol{\c_1}, \dots,\ol{\c_n}}{\bm\l}. 
\end{align*}

\item \label{N of X_Lauricella-C} Let $\chi = ((\a_i)_{i=0}^n, (\b_i)_{i=0}^n) \in (\khat)^{2(n+1)}$.
If $\e \not \in \{\a_0\cdots\a_n, \b_0\cdots\b_n\}$, then we have
\begin{align*}
& N(X_{F_C^n,\bm\l}; \chi) \\
& = (-1)^n j(\a_0, \dots, \a_n)j(\b_0,\dots,\b_n) \FC{n}{\a_0\, ;\, \b_0}{\ol{\a_1},\dots,\ol{\a_n}\, ;\,  \ol{\b_1}, \dots,\ol{\b_n}}{\bm\l}.
\end{align*}
\end{enumerate}
\end{thm}

\begin{rem}
Let $Y_{\bm\l}$ be a curve over $\k$ defined by
$$ v^N = u^a (1-u)^c \prod_{i=1}^n (1-\l_i u)^{b_i} \quad (v \neq 0), $$
where $a, b_i, c$ are positive integers.
In \cite[Theorem 4.2] {N}, the number of rational points $N(Y_{\bm\l}; \chi)$, where $\chi \in \khat$, is expressed in terms of Lauricella's function $F_D^{(n)}$ over $\k$. 
There exists a morphism $X_{F_D^n,\bm\l} \otimes \k_N \rightarrow Y_{\bm\l} \otimes \k_N$ given by
$$ u = -\left( \frac{x_0}{y_0}\right)^N,\quad  v =\dfrac{(\sqrt[N]{-1} x_0 )^a}{y_0^{a+c} \prod_{i=1}^n x_i^{b_i}}.$$
\end{rem}

\subsection{Humbert's functions}\label{subsec. of var. for Humbert's functions}
Let $\bm\l=(\l_1,\l_2) \in (\k^*)^2$.
Define varieties $X_{\Phi_1,\bm\l}\subset Fer_2^2 \times AS$ and $X_{\Phi_3,\bm\l}\subset Fer_2 \times AS^2$ by
\begin{align*}
& X_{\Phi_1,\bm\l}\colon \begin{cases} x_i^N+y_i^N=1 & (i=1,2)\\ t^q-t = z^N \\ \l_1 x_1^Nx_2^N = y_1^Ny_2^N \\ \l_2 x_1^N = y_1^Nz^N \\ x_1x_2y_1y_2z \neq 0, \end{cases}\ 
X_{\Phi_3,\bm\l} \colon \begin{cases} x^N + y^N = 1 \\ t_i^q-t_i = z_i^N & (i=1,2) \\ \l_1 x^N = y^Nz_1^N \\ \l_2 = z_1^Nz_2^N \\ xyz_1z_2 \neq 0. \end{cases} 
\end{align*}
As in the previous subsections, the groups $(\k^*)^5 \times \k$ and $(\k^*)^4 \times \k^2$ act on $X_{\Phi_1,\bm\l}$ and $X_{\Phi_3,\bm\l}$, respectively.
We have the following theorem by arguments analogous to those used in the preceding subsections.

\begin{thm}\label{N of X_Humbert}$ $
\begin{enumerate}
\item Let $\chi = ( \a_1,\a_2,\b_1,\b_2, \c, \psi_c) \in \widehat{(\k^*)^5 \times \k}$, where $c \in \k^*$.
If $\e \not\in\{ \a_1\b_1, \a_2\b_2\}$, then we have
$$ N(X_{\Phi_1,\bm\l}; \chi ) = -g(\c) \ol\c(c) j(\a_1,\b_1)j(\a_2,\b_2) \Phi_1\left( {\a_1 ; \a_2 \atop \ol{\b_1}; \ol{\b_2}, \ol\c}; \l_1,c\l_2 \right)_\psi.$$

\item Let $\chi = (\a,\b,\c_1,\c_2, \psi_{c_1}, \psi_{c_2}) \in \widehat{ (\k^*)^4 \times \k^2}$, where $c_1,c_2 \in \k^*$.
If $\a\b \neq \e$, then we have
$$N(X_{\Phi_3,\bm\l}; \chi) = -(\prod_{j=1}^2 g(\c_j)\ol{\c_j}(c_j) )j(\a,\b) \Phi_3\left( {\a \atop \ol{\c_1}; \ol \b,\ol{\c_2}}; c_1\l_1, c_1c_2\l_2\right)_\psi.$$
\end{enumerate}
\end{thm}

As mentioned in Remark \ref{rem of Phi(1,1,...,1)}, $\Phi_1$ and $\Phi_2$ are essentially equivalent, and hence we omit $X_{\Phi_2,\bm\l}$ in this paper.

\section{Relations among the hypergeometric varieties}
Put $N=q-1$.
For $x=(x_1, \dots, x_n) \in (\mathbbm A^*)^n := \mathbbm A^n-\{x_1\cdots x_n=0\}$ and $A=(a_{ij}) \in M(n,m;\Z)$, we use the following notation:
$$ x * A := ( \prod_{i=1}^n x_i^{a_{i1}}, \dots , \prod_{i=1}^n x_i^{a_{im}} ) \in (\mathbbm A^*)^m.$$
For $x \in (\k^*)^n$ or $(\khat)^n$, we use the same notation.
For homogeneous coordinates $x = (x_1 : \cdots : x_n)$ with $x_1 \cdots x_n \neq 0$,  the notation $x*A$ means the above monomial expression formed from the coordinates $x_1, \dots, x_n$.
For $A \in M(n,m;\Z)$ and $B \in M(m,l;\Z)$, one shows
$$ ( x *A) * B = x*(AB), \quad x*(A+B) = (x*A)(x*B),\quad ( x y)*A = (x *A)(y *A),$$
where $xy$ is the component-wise product of $x$ and $y$.
For $ x \in (\k^*)^n$, $\chi \in (\khat)^m$ and $A \in M(n,m;\Z)$, one shows
\begin{equation}
\chi (x * A) = (\chi*{}^t A)(x). \label{character*matrices}
\end{equation}

In this section we realize the symmetry and transformation formulas induced by the symmetry geometrically by constructing explicit isomorphisms between the varieties introduced in Section 4.

\subsection{General hypergeometric functions}
Let $L$ be a field which contains the group $\mu_N$ of all $N$th roots of unity.
For $a \in L^*$, fix an $N$th root $\sqrt[N]{a}$ of $a$.
Define the homomorphism 
$$ K_a \colon {\rm Gal}( L(\sqrt[N]{a}) / L) \longrightarrow \mu_N \, ;\, e \longmapsto \dfrac{e(\sqrt[N]{a})}{\sqrt[N]{a}}.$$
For brevity, we use the same notation for the composition with the restriction map from ${\rm Gal}(L'/L)$ to ${\rm Gal}(L(\sqrt[N]{a})/L)$, where $L(\sqrt[N]{a})\subset L'$.
Note that $L( \{ \sqrt[N]{a} \mid a \in L^*\} ) = \k_N$ when $L=\k$.
For $t \in \k$, fix a root $r(t)$ of the Artin-Schreier equation $x^q-x = t$.
Define the homomorphism
\begin{align*}
A_t & \colon {\rm Gal}(\k(r(t))/\k) \longrightarrow \k \, ;\, e \longmapsto e(r(t))-r(t).
\end{align*}
As above, we use the same notation for the composition with a restriction map, and note that $\k( \{ r(t) \mid t \in \k\}) = \k_p$.
Note that the homomorphisms $K_a$ and $A_t$ are independent of the choice of $\sqrt[N]{a}$ and $r(t)$.

Hereafter, fix an $N$th root of $a$, denoted by $\sqrt[N]{a}$ for each $a \in \k^*$.
In particular, we take $\sqrt[N]{1} = 1$.
Furthermore, fix a root of the Artin-Schreier equation $x^q -x = 1$ and write $r(1)$ for it.
For $t \in \k$, put
$$ r(t) = t \times r(1) \in \k_p.$$
One shows that $r(t)$ is a root of the equation $x^q-x = t$, and that the map $t \mapsto r(t)$ defines a homomorphism $\k \rightarrow \k_p$.
For $a = (a_1, \dots, a_n) \in (\k^*)^n$, write
$$ K_{a} (e) = ( K_{a_1}(e), \dots, K_{a_n}(e) ) $$
and
$$ \sqrt[N]{a} = (\sqrt[N]{a_1}, \dots, \sqrt[N]{a_n}). $$

Let $\D=(N_1, \dots, N_l)$ ($p \geq N_l$) be a partition of $n$ as in Subsection \ref{subsection of X_D,z}.
For $h  \in H_\D $, write
$$ h = {\rm diag}(h_1, \dots, h_l), \quad h_i = [ h_{(i,0)}, \dots, h_{(i,N_i-1)} ] \in J(N_i).$$
The following theorems give geometric analogues of the formulas in Proposition \ref{action of GL and H on Phi} and Theorem \ref{sym of Phi}.
\begin{thm}\label{isom between X_D}
We have the following isomorphisms.
\begin{enumerate}
\item \label{isom Lg}
For $g \in GL_d(\k)$, 
$$ L_g \colon X_{\D,z}\longrightarrow X_{\D,gz},$$
given by
$$ \big( (t_{i}, u_{i}), s \big) \longmapsto   \big( (t_{i}, u_{i}), sg^{-1} \big). $$

\item \label{isom Rh}
For $h \in H_\D$,
$$ R_h \colon X_{\D,z} \otimes \k' \longrightarrow X_{\D,zh} \otimes \k'  \, ;\,  \big( (t_{i}, u_{i}), s \big) \longmapsto \big( (t_{i}', u_{i}'), s \big),$$
where
$$ t_{i}' = \sqrt[N]{h_{(i,0)}} t_{i}, \quad u_{(i,j)}' = u_{(i,j)} + r( \theta_j( h_i ) ) .$$
Here, $\k' = \k_N$ if $\D=(1, \dots, 1)$, and $\k'=\k_{pN}$ otherwise.
\end{enumerate}
\end{thm}
\begin{proof}
\ref{isom Lg} For $ \big( (t_{i}, u_{i}), s \big)  \in X_{\D,z}$,  one shows $L_g( (t_{i}, u_{i}), s  ) \in X_{\D,gz}$, since $sz = (sg^{-1})(gz)$. Clearly, the inverse morphism of $L_g$ is $L_{g^{-1}}$.

\ref{isom Rh} Noting that  $zh = (z^{(1)}h_1, \dots , z^{(l)} h_l)$ ($z^{(i)}$ is as in \eqref{z rep. 1}), it suffices to prove the assertion for the case when $\D=(n)$ (i.e. $l=1$ and $H_\D = J(n)$).
Write $z = (z_0, \dots, z_{n-1})$, where $z_i$ are column vectors.
Then, for $ h = [h_0, \dots, h_{n-1}] \in J(n)$,  one shows
\begin{align*}
h \cdot [sz]
& = h \cdot [sz_0, \dots, sz_{n-1}] \\
&= [s h_0z_0, s(h_1z_0 + h_0z_1), \dots, s(h_{n-1}z_0 + \cdots + h_0z_{n-1})] \\
& = [szh].
\end{align*}
Therefore, we have $\theta_i(h) + \theta_i(sz) = \theta_i(h \cdot [sz]) = \theta_i(szh)$ by \eqref{additivity of theta}.
By this, for $ \big( (t_{i}, u_{i}), s \big)  \in X_{\D,z}$, we can check that $R_h( (t_{i}, u_{i}), s  ) \in X_{\D,zh}$.
Its inverse $R_h^{-1}$ is $R_{h^{-1}}$ with compatible choices of roots, where noting that $\theta_j(h^{-1}) = -\theta_j(h)$ for $j \geq 1$ by Proposition \ref{char of J(m)} \ref{char of J(m)-1}.
\end{proof}

From now on, let $\Delta = (\overbrace{n_1, \dots, n_1}^{p_1}, \dots , \overbrace{n_k, \dots, n_k}^{p_k})$ be a partition of $n$ as in \eqref{refinment partition}.
Write
$$ z = (z^{(1)}, \dots, z^{(k)}),\quad z^{(i)} = (z^{(i,1)}, \dots, z^{(i, p_i)}) \in M(d, n_ip_i;\k) ,$$
where $z^{(i,j)} \in M(d, n_i; \k)$.
Then, the equation of $X_{\D,z}$ can be written as
$$ \begin{cases} 
t_{(i,j)}^N = sz^{(i,j)}_0 \vspace{3pt}\\
t_{(i,j)}^N (u_{(i,j,1)}^q -u_{(i,j,1)}) =\ol\theta_1(sz^{(i,j)}) \vspace{3pt}\\
\hspace{30pt}\vdots\\
t_{(i,j)}^{N(n_i-1)} (u_{(i,j,n_i-1)}^q - u_{(i,j,n_i-1)}) = \ol\theta_{n_i-1}(sz^{(i,j)}) \vspace{3pt}\\
t_{(i,j)} \neq 0
\end{cases}(1 \leq i \leq k, 1 \leq j \leq p_i),$$
and we write $u_{(i,j)} = (u_{(i,j,1)}, \dots, u_{(i,j,n_i-1)})$.

For $w \in W_\D = \prod_{i=1}^k (W(n_i)^{p_i} \rtimes \mathcal P_i) \subset GL_n(\k)$, write
$$ w = {\rm diag}( w_1, \dots, w_k), \quad w_i = {\rm diag} (\mu(c_1), \dots , \mu(c_{p_i})) \widetilde P_{\sigma_i} \in W(n_i)^{p_i} \rtimes \mathcal P_i,$$
where $  \sigma_i \in \S_{p_i}$.
Here, $\mathcal{P}_i \cong \S_{p_i}$, $W(n_i)$, $\widetilde P_\sigma \in \mathcal P_i$ and $\mu(c_j) \in W(n_i)$ are as in Subsection \ref{subsec. of Def. and Prop.}.
\begin{thm}\label{isom fw}
For $w \in W_\D$, we have the isomorphism
$$ f_w \colon X_{\D,z} \longrightarrow X_{\D,zw} \, ;\, \big( (t_{(i,j)}, u_{(i,j)}), s \big) \longmapsto \big( (t_{(i,j)}', u_{(i,j)}'), s \big),$$
given by
$$ t_{(i,j)}' =  t_{(i, \sigma_i(j))} ,\quad  u_{(i,j)}' =  u_{(i, \sigma_i(j))} \mu(c_{\sigma_i(j)})' .$$
\end{thm}
\begin{proof}
It suffices to prove the assertion for the case when $\D = (\overbrace{n', \dots, n'}^m)$ (i.e. $k=1$ and $W_\D = W(n')^m \rtimes \mathcal \S_m$).
By decomposing $w$, it suffices to treat the following two cases:
\begin{enumerate}
\item \label{case 1} $w = \widetilde P_\sigma$, where $\sigma \in \S_m$,
\item \label{case 2} $w = {\rm diag} (\mu(c_1), \dots, \mu(c_m)) \in W(n')^m$.
\end{enumerate}
In case \ref{case 1}, it is clear that $P \in X_{\D,z} \Rightarrow f_w(P) \in X_{\D,zw}$ and $f_w^{-1} = f_{w^{-1}}$.
Here, note that if we write $z = (z^{(1)}, \dots, z^{(m)})$, where $z^{(j)}$ is the $j$-th block of $n'$ columns, then $zw = (z^{(\sigma(1))}, \dots, z^{(\sigma(m))})$.

In case \ref{case 2}, it suffices to prove the assertion for the case when $m=1$ (i.e. $\D=(n)$ and $W_\D = W(n)$).
Write $z=(z_0, \dots, z_{n-1})$, where $z_0,\dots, z_{n-1}$ are the columns.
Then, the defining equation of $X_{\D,z}$ is equivalent to
\begin{equation}
 (t^N , u_1^q-u_1, \dots, u_{n-1}^q - u_{n-1}) = (sz_0, \theta_1(sz), \dots , \theta_{n-1}(sz)),\quad t \neq 0.\label{eq.2 in pf of isom fw}
\end{equation}
Thus, for $P := (t, u_1, \dots, u_{n-1}, s) \in X_{\D,z}$, and $w:= \mu(c) \in W_\D$, we have
\begin{align}
(t^N, u_1^q-u_1, \dots, u_{n-1}^q-u_{n-1})w= (sz_0, \theta_1(sz), \dots , \theta_{n-1}(sz))w. \label{eq.1 in pf of isom fw}
\end{align}
Now,  
$$ f_w(t, u_j, s) = (t, u_j', s), \quad (u_1', \dots, u_{n-1}') := (u_1, \dots, u_{n-1})\mu(c)'.$$
The left-hand side of \eqref{eq.1 in pf of isom fw} is clearly equal to 
$$ (t^N, (u_1')^{q}-u_1', \dots, (u_{n-1}')^{q}-u_{n-1}').$$
On the other hand, by \eqref{theta mu = theta(mu)}, the right-hand side of \eqref{eq.1 in pf of isom fw} is equal to
\begin{align*}
(sz_0, \theta_1(szw), \dots, \theta_{n-1}(szw)).
\end{align*}
Thus, we have
$$ (t^N, (u_1')^{q}-u_1', \dots, (u_{n-1}')^{q}-u_{n-1}') = (sz_0, \theta_1(szw), \dots, \theta_{n-1}(szw)).$$
This means $f_w(P) \in X_{\D,zw}$.
One shows that $f_w^{-1} = f_{w^{-1}}$, where noting that $(\mu(c)')^{-1} = (\mu(c)^{-1})'$.
\end{proof}
\begin{rem}
Clearly, for $g, g' \in GL_d(\k)$, we have $L_{g'g} = L_{g'} \circ L_g$.
Note that 
\begin{align*}
& {\rm diag}(\mu(c_1), \dots, \mu(c_{p_i}))\widetilde P_\sigma \times {\rm diag}(\mu(c_1'), \dots, \mu(c_{p_i}') ) \widetilde P_{\sigma'} \\
& = {\rm diag}(\mu(c_1)\mu(c_{\sigma^{-1}(1)}'), \dots, \mu(c_{p_i})\mu(c_{\sigma^{-1}(p_i)}')) \widetilde P_{\sigma\sigma'}.
\end{align*}
It follows that $f_{w w'} = f_{w'} \circ f_{w}$ for $w , w' \in W_\D$.
On the other hand, for $h, h' \in H_\D$, if we choose $\sqrt[N]{h_{(i,0)}} \sqrt[N]{h_{(i,0)}'}$ as $\sqrt[N]{h_{(i,0)}h_{(i,0)}'}$, we have $R_{hh'} = R_{h'} \circ R_{h}$.
\end{rem}

\begin{rem}
When $\D=(1, \dots, 1)$, the Artin-Schreier curves are not necessary for the definition of $X_{\D,z}$.
Then, analogous statements of Theorems \ref{isom between X_D} and \ref{isom fw} hold in characteristic zero, where we have to replace $\k_N$ with $\k( \{ \sqrt[N]{a} \mid a \in \k^*\})$.
\end{rem}

For $\chi \in \widehat{H_\D}$ and $w \in W_\D$, $\chi {}^tw \in \widehat{H_\D}$ is as in Subsection \ref{subsec. of Def. and Prop.}.
Recall the isomorphism $\widetilde \iota \colon H_\D \rightarrow G_\D = \prod_{i=1}^k \big( (\k^*) \times \k^{n_i-1} \big)^{p_i}$ by \eqref{isom iota tilde}.

\begin{thm}\label{N of X_D(z)} Let $\chi \in \widehat{H_\D}$.
\begin{enumerate}
\item \label{N of X_D(z)-2} For $g \in GL_d(\k)$,
$$ N(X_{\D,z}; \chi) = N(X_{\D,gz} ; \chi). $$ 

\item \label{N of X_D(z)-3} For $ h \in H_\D$,
$$ \chi(h) N(X_{\D,z} ; \chi) = N(X_{\D,zh} ; \chi).$$

\item \label{N of X_D(z)-4} For $w \in W_\D$, 
$$ N(X_{\D,z} ; \chi {}^t w) = N(X_{\D,zw} ; \chi). $$
%Here, $\pi := (\widetilde \iota)^{-1} \circ \pi' \circ \widetilde \iota$, where $\pi' \colon G \rightarrow G$ is an isomorphism given by, for $\big( (m_{i,j}, a_{i,j})_{j=1}^{p_i} \big)_{i=1}^k \in G$ $ (m_{i,j} \in \k^*, a_{i,j} \in \k^{n_i-1})$, 
%$$ m_{i,j} \longmapsto m_{i, \sigma_i(j)}, \quad a_{i,j} \mapsto a_{i,\sigma_i(j)} . $$
\end{enumerate}
\end{thm}
\begin{proof}
\ref{N of X_D(z)-2} The identity follows from the isomorphism $L_g$ and Lemma \ref{Comparison of N by isom} with $\pi = {\rm id}_{H_\D}$.

\ref{N of X_D(z)-3} We prove the assertion only for the case when $\D=(n)$ (the general case is similar).
Then, $X_{\D,z}$ is defined by the equation \eqref{eq.2 in pf of isom fw}. 
Let $h = [h_0, \dots, h_{n-1}] \in H_\D = J(n)$.
Recall  $R_h(t,u_j,s) = ( \sqrt[N]{h_0} t, u_j + r(\theta_j(h)), s)$.
%For short notation, write $\tau_j$ for $r(\theta_j(h))$.
For $g := (\xi, a_1, \dots, a_{n-1}) \in G_\D$ and $e \in {\rm Gal}:={\rm Gal}(\k_{pN}/\k)$, put $g_e = (\xi', a_1', \dots, a_{n-1}') \in G_\D$, where
\begin{align*}
& \xi' := \xi \times K_{h_0}(e), \\
& a_j' := a_j + A_{\theta_j(h)}(e).
\end{align*}
Then, one shows (by computing on the structure sheaves)
\begin{equation} 
R_h  \circ (g,e)  = (g_e, e) \circ R_h. \label{eq. in pf of N of X_D(z)}
\end{equation}
Define $\pi \colon G_\D \times {\rm Gal} \rightarrow G_\D \times {\rm Gal}\, ;\, (g, e) \mapsto (g_e, e)$.
Then, $\pi$ is an automorphism and we have, for $\chi' \in \widehat{ G_\D}$ and $\rho \in \widehat{\rm Gal}$, 
\begin{equation}
N(X_{\D,z} \otimes \k_{pN} ; \pi^* (\chi', \rho)) = N(X_{\D,zh} \otimes \k_{pN}; (\chi',\rho)), \label{eq.3 in pf of N of X_D(z)} 
\end{equation}
by Lemma \ref{Comparison of N by isom}, \eqref{eq. in pf of N of X_D(z)} and Theorem \ref{isom between X_D} \ref{isom Rh}.
One shows $ \pi^*(\chi',\rho) (g, e) = \chi'(g_e) \rho(e)$.
Write $\chi' = (\a, \psi_1, \dots, \psi_{n-1})$, where $\a \in \khat$ and $\psi_j \in \widehat\k$.
Then, we have
$$ \chi' (g_e) = \chi'(g) \times \a\circ K_{h_0}(e) \times  \prod_{j=1}^{n-1} \big( \psi_j \circ A_{\theta_j(h)}(e) \big).$$
Therefore, if we put $\eta = (\a\circ K_{h_0} ) \times \prod_j ( \psi_j\circ A_{\theta_j(h)}) \in \widehat {\rm Gal}$, then
$$ \pi^*(\chi', \rho) = (\chi', \eta \rho).$$
Since $\chi = \chi' \circ \widetilde\iota$ for a suitable $\chi' \in \widehat{G_\D}$, we have
$$ \eta(F) N(X_{\D,z} ; \chi) = N(X_{\D,zh}; \chi), $$
by \eqref{eq.3 in pf of N of X_D(z)}, \eqref{N of base change} and \eqref{eq.2 in pr of N of X_D(z)}.
Noting that $K_x(F) = x$ for $x \in \k^*$ and $A_{y}(F) = y$ for $y \in \k$, we have
$$ \eta(F) = \a(h_0) \prod_j \psi_j (\theta_j(h)) = \chi' \circ \widetilde \iota (h) = \chi(h).$$
Hence, we have $\chi(h) N(X_{\D,z}; \chi) = N(X_{\D,zh};\chi)$.

\ref{N of X_D(z)-4} We need only prove the assertion for the case when $\D=( \overbrace{n',\dots,n'}^m)$.
Then, $H_\D = J(n')^m$, $G_\D = (\k^* \times \k^{n'-1})^m$ and $W_\D = W(n')^m \rtimes \S_m$.
For $ g = (\xi_j, x_j)_{j=1}^m  \in G_\D$, where $x_j \in \k^{n'-1}$ and for $w = {\rm diag}(\mu(c_1), \dots, \mu(c_{m}) ) \widetilde P_\sigma \in W_\D$, define
$$ \pi_w ( (\xi_j, x_j)_j) = (\xi_{\sigma(j)}, x_{\sigma(j)}\mu(c_{\sigma(j)})' )_j. $$
Then, $\pi_w$ defines an isomorphism $G_\D \rightarrow G_\D$ and it satisfies 
$$ f_w \circ g = \pi_w(g) \circ f_w \quad (g \in G_\D). $$
Let $\chi = \chi' \circ \widetilde \iota$, where $\chi' \in \widehat{G_\D}$.
By the commutativity above and Lemma \ref{Comparison of N by isom}, we have
$$ N(X_{\D,z} ; \pi_w^* \chi') = N(X_{\D,zw}; \chi').$$
Recall that we can write $\chi' = (\a_j, \psi_{a_j})_j$, where $\a_j \in \widehat{\k^*}$, $a_j \in \k^{n'-1}$ and $\psi_{a_j} = (\psi_{a_{j,1}}, \dots, \psi_{a_{j,n'-1}})$, and then $\chi = (\a_j, a_j)_j$.
One shows
$$ \pi_w^* \chi' = (\a_{\sigma^{-1}(j)}, \psi_{a_{\sigma^{-1}(j)} {}^t \mu(c_j)'})_j,$$
where noting that $\psi_{ a}(xM) = \psi_{ a \cdot {}^tM }(x)$ for $a, x \in \k^{n}$ and $M \in M(n,n;\k)$.
Therefore %noting that ${}^t w = {\rm diag}({}^t \mu(c_{\sigma(1)}), \dots, {}^t \mu(c_{\sigma(n'-1)}) ) \widetilde P_{\sigma^{-1}}$, 
we have 
$$ \pi_w^*\chi' \circ \widetilde \iota = \chi {}^t w.$$
Thus, by \eqref{eq.2 in pr of N of X_D(z)}, we obtain the theorem.
\end{proof}

\begin{proof}[Proof of Proposition \ref{action of GL and H on Phi} and Theorem \ref{sym of Phi}]
The result follows from Theorems \ref{N of X_D(z)-1} and \ref{N of X_D(z)}.
Indeed, by Theorem \ref{N of X_D(z)} \ref{N of X_D(z)-4}, we have
$$ \Phi_\D(\chi {}^t w; z) = N(X_{\D,z};\chi {}^t w) = N(X_{\D,zw};\chi) = \Phi_\D(\chi;zw).$$
Proposition \ref{action of GL and H on Phi} can be proved similarly.
\end{proof}

In Subsection \ref{particular cases}, we showed that certain transformation formulas for Gauss', Kummer's, Appell-Lauricella and Humbert's functions are derived from the symmetry of general hypergeometric functions.
In the following subsections, we upgrade the formulas to isomorphisms among the varieties defined in Subsections \ref{subsec. of var. for one-variable functions}--\ref{subsec. of var. for Humbert's functions}.
\subsection{Gauss' function ($d=2, n=4, \D=(1,1,1,1)$)}\label{subsec. of X22}
Recall that $W_\D = \left\{ P_\sigma \mid \sigma \in \S_4\right\}$ and, for $z = (z_{ij}) \in M(2,4;\k)$, $X_{\D,z} \subset \mathbbm{A}^6 = \{(t,s):= (t_1, \dots, t_4, s_1, s_2)\}$ is defined by
$$ \left\{ \begin{array}{ll} t_i^N = s_1 z_{1i}+ s_2z_{2i} & (i=1, \dots, 4) \\ t_1\cdots t_4 \neq 0. \end{array} \right. $$
For $\l  \in \k^* - \{1\}$, fix an element $x = (x_{ij}) \in GL_2(\k)$ such that $\l = \dfrac{x_{11}x_{22}}{x_{21}x_{12}}$.
Put $z = \left( \begin{array}{c|c} x & I_2 \end{array} \right)$ (note that $[i\, j] \neq 0$ for any $i \neq j$, where $[i\, j]$ is as in Subsection \ref{particular cases}). 
Let $X_x \subset \mathbbm{P}^3 = \{  (u_1:u_2:v_1:v_2) \}$ be an algebraic variety defined by the equations
$$ \left\{ \begin{array}{l} u_1^N = x_{11}v_1^N+ x_{21} v_2^N \vspace{3pt}\\ u_2^N = x_{12}v_1^N + x_{22}v_2^N \\ u_1u_2v_1v_2 \neq 0. \end{array} \right. $$
Clearly, we have the morphism
$$ \p \colon X_{\D,z} \longrightarrow X_x \colon (t, s) \longmapsto (t_1 : t_2 : t_3: t_4). $$
Put
$$  \theta_0 = \bmat{-1 & 0 & -1 & 0 \\ 0 & 0 & 0 &  0 \\ 0 & 0 & 1 &  0 \\ 0 & -1 & 0 & 0 }.$$
Then, we have the morphism 
$$ \p_{\theta_0} \colon {}_2X_{2,\l} \otimes \k_N \longrightarrow X_x \otimes \k_N$$
given by
\begin{align*}
 (x_1, x_2, y_1, y_2)  \longmapsto & \big(\sqrt[N]{d_x}^{-1}  (x_1:x_2:y_1:y_2)  \big) * \theta_0 \\
 & = \Big( \dfrac{\sqrt[N]{x_{21}}}{x_1} : \dfrac{\sqrt[N]{x_{22}}}{y_2} : \dfrac{\sqrt[N]{x_{21}}y_1}{ \sqrt[N]{x_{11}}x_1} : 1 \Big),
\end{align*}
where $d_x = (x_{21}, x_{12}, x_{11}, x_{22})$.

\begin{rem}\label{rem of 2X2 and Xx}
As in Remark \ref{rem of decomposition of mXn}, we have the decomposition
$$ {}_2X_{2,\l} \otimes \k_N = \bigsqcup_{\tau^N = \l } {}_2X_{2,\l}^\tau,$$
where
$$  {}_2X_{2,\l}^\tau := \{ (x_i,y_i) \in {}_2X_{2,\l} \otimes \k_N \mid \tau x_1x_2 =  y_1y_2 \}.$$
As shown below, the restriction of $\p_{\theta_0}$ induces the isomorphism 
$$\p_{\theta_0}^\tau \colon {}_2X_{2,\l}^\tau \longrightarrow X_x \otimes \k_N.$$
For $(x_i,y_i) \in {}_2X_{2,\l}^\tau$, one shows
\begin{equation}
(1,1,1,\tau') \big( ( \sqrt[N]{d_x}^{-1}(x_i,y_i) )*T \big) = (1,1,1,1), \label{tau' invariant}
\end{equation} 
where
$$ T := \left( \begin{array}{c|c} \mat{O} & \mat{1\\1\\-1\\-1} \end{array} \right) \in M(4,4;\Z) ,\quad \tau':= \tau \cdot (\sqrt[N]{d_x}*\bmat{1\\1\\-1\\-1}).$$
Thus, we have
$$\p_{\theta_0}^\tau  (x_i,y_i) = (1,1,1,\tau')\big( (\sqrt[N]{d_x}^{-1}(x_i:y_i))*\theta \big),$$
where $\theta := \theta_0 + T \in GL_4(\Z)$, and 
$$ (\p_{\theta_0}^\tau)^{-1} (u_1:u_2:v_1:1) = \dfrac{\sqrt[N]{d_x}}{(1,1,1,\tau')*\theta^{-1}} ( (u_1,u_2,v_1,1)*\theta^{-1} ).$$
\end{rem}

For each $\sigma \in \S_4$, write $ (z_{ij}^\sigma) = zP_\sigma$ and let
$$ z_\sigma = \bmat{ z_{13}^\sigma & z_{14}^\sigma \\ z_{23}^\sigma & z_{24}^\sigma}^{-1} zP_\sigma. $$
%Note that $(z_{\sigma_1})_{\sigma_2} = z_{\sigma_1\sigma_2}$ for $\sigma_1, \sigma_2 \in \S_4$.
One shows that the matrix $z_\sigma$ can be written as $\left( \begin{array}{c|c} (x_{ij}^\sigma)  & I_2 \end{array} \right)$, where $x_{ij}^\sigma \in \k^*$ and $\det (x_{ij}^\sigma) \neq 0$.
Put 
$$x_\sigma = (x_{ij}^\sigma) \in GL_2(\k), \quad \l_\sigma = \dfrac{x_{11}^\sigma x_{22}^\sigma}{x_{21}^\sigma x_{12}^\sigma} \in \k^*-\{1\}. $$
\begin{rem}
For $\sigma , \sigma' \in \S_4$, the third and fourth columns of $z_\sigma P_{\sigma'} $ are equal to $\bmat{z_{13}^\sigma& z_{14}^\sigma\\ z_{23}^\sigma & z_{24}^\sigma}^{-1} \bmat{z_{13}^{\sigma\sigma'} & z_{14}^{\sigma\sigma'}\\ z_{23}^{\sigma\sigma'} & z_{24}^{\sigma\sigma'}}$. 
Thus, one shows $(z_\sigma)_{\sigma'} = z_{\sigma\sigma'}$ and $(x_\sigma)_{\sigma'}=x_{\sigma\sigma'}$.
\end{rem}
By  Theorems \ref{isom fw} and \ref{isom between X_D} \ref{isom Lg}, we have the isomorphism 
$$f_\sigma \colon X_{\D,z} \longrightarrow X_{\D,z_\sigma} \, ;\, (t,s) \mapsto (t*P_\sigma, s') \quad ( s':=  s\bmat{ z_{13}^\sigma & z_{14}^\sigma \\ z_{23}^\sigma & z_{24}^\sigma}).$$
Clearly, 
$$g_\sigma \colon X_{x} \longrightarrow X_{x_\sigma}\, ;\, (u_1 : u_2 : v_1 : v_2) \mapsto (u_1 : u_2 : v_1 : v_2)*P_\sigma$$
is an isomorphism and it satisfies $\p \circ f_\sigma = g_\sigma \circ \p$.

Define
$$ Q_\sigma = (\theta_0 P_\sigma M + T)\theta^{-1} \in M(4,4;\Z),$$
where
$$ M= \bmat{1 & 0 & 0 & 0 \\ 0 & 1 & 0 & 0 \\ 0 & 0 & 1 & 0 \\ -1 & -1 & -1 & 0} .$$

\begin{lem}\label{lem of Qsigma}$ $
For any $\sigma, \sigma' \in \S_4$, we have 
$$ Q_\sigma Q_{\sigma'} = Q_{\sigma\sigma'}.$$
In particular, $Q_\sigma \in GL_4(\Z)$.
\end{lem}

\begin{proof}
One shows $ T\theta^{-1} \theta_0 = O_4$, $ M \theta^{-1} \theta_0 = M$ and $ M P_\sigma M = P_\sigma M$, and hence, we have
\begin{equation*}\label{eq. in pf. of lem of Qsigma}
Q_\sigma \theta_0 P_{\sigma'} M = \theta_0 P_{\sigma\sigma'} M.
\end{equation*}
Furthermore, one shows $Q_\sigma T = T$ for any $\sigma\in\S_4$, where noting that $\theta^{-1}T = \left( \begin{array}{c|c} O & e_4 \end{array} \right)$.
These identities give
\begin{align*}
Q_\sigma Q_{\sigma'} = Q_\sigma (\theta_0 P_{\sigma'} M + T)\theta^{-1} = (\theta_0 P_{\sigma \sigma'} M+ T ) \theta^{-1} = Q_{\sigma \sigma'}. 
\end{align*}
In particular, $Q_\sigma^{-1} = Q_{\sigma^{-1}}$, where noting that $Q_{\rm id} =(\theta_0 M +T)\theta^{-1}= \theta \theta^{-1}= I_4$.
\end{proof}

For $\tau \in \{ \tau \mid \tau^N=\l \}$ and $\sigma \in \S_4$, put
$$ \tau_\sigma = \tau \cdot \dfrac{\sqrt[N]{d_x}*{}^t(1,1,-1,-1)}{\sqrt[N]{d_{x_\sigma}}*{}^t(1,1,-1,-1)} \in \{ \tau \mid \tau^N = \l_\sigma\}. $$
Then, by the isomorphism in Remark \ref{rem of 2X2 and Xx}, we have the isomorphism 
$$h_\sigma^\tau := (\p_{\theta_0}^{\tau_\sigma})^{-1} \circ g_\sigma \circ \p_{\theta_0}^\tau \colon  {}_2X_{2,\l}^\tau \longrightarrow {}_2X_{2,\l_\sigma}^{\tau_\sigma}.$$
Let ${\rm Gal} = {\rm Gal}(\k_N/\k)$ and put $G = (\k^*)^4$.
The following theorem is a geometric interpretation of the 24 formulas for Gauss' functions over $\k$ mentioned in Remark \ref{rem of Phi(1,1,1,1)}.

\begin{thm}\label{isom between X22 and X22-sigma}
Let $\sigma \in \S_4$ and $\tau \in \{\tau\mid \tau^N = \l\}$.
   \begin{enumerate}
      \item The isomorphism $h_\sigma^\tau$ is given by
      \begin{align*}
(x_i,y_i) 
 \longmapsto & \sqrt[N]{d_{x_{\sigma}}}  \Big( \big( \sqrt[N]{d_{x}}^{-1} (x_1,x_2,y_1,y_2)\big)*Q_\sigma\Big)  \\
 & = \dfrac{\sqrt[N]{d_{x_{\sigma}}}}{\sqrt[N]{d_{x}}* Q_\sigma} \Big( (x_1,x_2, y_1, y_2) * Q_\sigma \Big) .
 \end{align*}
 By the same correspondence, we have the isomorphism
 $$h_\sigma \colon {}_2X_{2,\l} \otimes \k_N \longrightarrow {}_2X_{2,\l_{\sigma}} \otimes \k_N.$$
%      \begin{equation*}
%         \begin{diagram}
%          % \node{X_{\D,z} \otimes \k_N} \arrow{s,l}{\p} \arrow{e,t}{f_\sigma} \node{X_\D(z_\sigma) \otimes \k_N} \arrow{s,r}{\p} \\
%            \node{X_{ x} \otimes \k_N} \arrow{e,t}{g_\sigma} \node{X_{x_{\sigma}} \otimes \k_N} \\
%            \node{{}_2X_{2,\l} \otimes \k_N} \arrow{n,l}{\p_\theta} \arrow{e,t}{h_\sigma} \node{X_{2,2}(\l_{\sigma}) \otimes \k_N} \arrow{n,r}{\p_\theta}
%         \end{diagram}.
%      \end{equation*

      \item For any $\sigma'\in \S_4$, we have $h_{\sigma'}^{\tau_\sigma} \circ h_{\sigma}^\tau = h_{\sigma \sigma'}^\tau$.
      
       \item Put $d_\sigma := d_{x_\sigma}/  (d_x * Q_\sigma) \in G$. 
       Define the automorphism $\pi_\sigma$ of $G\times {\rm Gal}$ by
       $$ (g, e) \mapsto ( K_{d_\sigma}(e)(g*Q_\sigma), e).$$
Then, for $g \in G$ and $e \in {\rm Gal}$, 
\begin{align*}
h_\sigma \circ (g,e)  = \pi_\sigma(g, e) \circ h_\sigma.
\end{align*}

      \item \label{cor of isom between X22 and X22-sigma} 
      For $\chi \in \widehat {G}$, we have
      $$ \chi(d_\sigma) N({}_2X_{2,\l}; \chi* {}^tQ_\sigma ) = N({}_2X_{2,\l_\sigma} ; \chi). $$  
   \end{enumerate}
\end{thm}
\begin{proof}
      (i) Noting $(u_i:v_i)*M=(u_i:v_i)$ for $(u_i:v_i) \in \P^3$ and \eqref{tau' invariant}, we have
      \begin{align*}
       g_\sigma \circ \p_{\theta_0}^\tau (x_i,y_i) 
       & = (\sqrt[N]{d_x} ^{-1}(x_i:y_i) )*\theta_0 P_\sigma M\\
       & = (1,1,1,\tau')\big( (\sqrt[N]{d_x}^{-1}(x_i:y_i) )*(\theta_0 P_\sigma M + T) \big),
      \end{align*}
      where $\tau'$ is as in Remark \ref{rem of 2X2 and Xx}.
      Therefore,
      \begin{align*}
       h_\sigma^\tau (x_i,y_i) 
       & = \dfrac{\sqrt[N]{d_{x_\sigma}}}{(1,1,1,\tau_\sigma')*\theta^{-1}} \cdot (1,1,1,\tau')*\theta^{-1}\cdot  \big( (\sqrt[N]{d_x}^{-1}(x_i,y_i) )*Q_\sigma \big),
      \end{align*}   
      where $\tau_\sigma' := \tau_\sigma ( \sqrt[N]{d_{x_\sigma}}*{}^t(1,1,-1,-1) )$.     
      Noting $\tau' = \tau_\sigma'$ by definition, we obtain 
      $$h_\sigma^\tau (x_i,y_i) = \sqrt[N]{d_{x_\sigma}}( (\sqrt[N]{d_x}^{-1}(x_i,y_i))*Q_\sigma ).$$
      Clearly, the isomorphisms $h_\sigma^\tau$ ($\tau \in \{ \tau \mid \tau^N=\l\}$) induce the isomorphism $h_\sigma$.
      
       (ii) 
       Noting $(x_{\sigma})_{\sigma'} = x_{\sigma \sigma'}$ and Lemma \ref{lem of Qsigma},
       we have
       \begin{align*}
       h_{\sigma'}^{\tau_\sigma} \circ h_{\sigma}^\tau (x_i,y_i)
       & = \sqrt[N]{d_{x_{\sigma\sigma'}}} \Big( \big(\sqrt[N]{d_{x_{\sigma}}}^{-1} h_{\sigma}^\tau(x_i,y_i) \big ) *Q_{\sigma'} \Big)\\
       & = \sqrt[N]{d_{x_{\sigma\sigma'}}} \Big( \big(\sqrt[N]{d_x}^{-1} (x_i,y_i)\big) *Q_{\sigma\sigma'} \Big) = h_{\sigma\sigma'}^{\tau}(x_i,y_i).
       \end{align*}

       (iii) We can prove this by computing on the structure sheaves.
%Here, note that $\sqrt[N]{d_\sigma}/ e^{-1}(\sqrt[N]{d_\sigma}) = e(\sqrt[N]{d_\sigma}/ e^{-1}(\sqrt[N]{d_\sigma}) ) = e(\sqrt[N]{d_\sigma})/\sqrt[N]{d_\sigma} = K_{d_\sigma}(e)$.

(iv) By Lemma \ref{Comparison of N by isom} and (iii), we have for $(\chi,\rho) \in \widehat{G\times {\rm Gal}}$,
$$ N({}_2X_{2,\l} \otimes \k_N ; \pi_\sigma^*(\chi, \rho) )  = N({}_2X_{2,\l_\sigma}\otimes \k_N ; (\chi, \rho)) .$$
Noting \eqref{character*matrices}, we have $\pi_\sigma^*(\chi,\rho) = (\chi*{}^tQ_\sigma, \eta\rho)$, where $\eta = \chi \circ K_{d_\sigma}$.
Thus, we complete the proof by \eqref{N of base change} and $K_{d_\sigma}(F) = d_\sigma$.
\end{proof}

For example, when $\sigma = (1 \, 3) \in \S_4$, we have
$$ Q_\sigma = \bmat{ 1 & 0 & 0 & 0 \\ 0 & 1 & 0 & 0 \\ -1 & -1 & -1 & 0 \\ 0 & 0 & 0 & 1 }. $$
Hence, for $\chi = (\a_1,\a_2,\b_1,\b_2)$, 
$$\chi * {}^tQ_\sigma = (\a_1, \a_2, \ol{\a_1\a_2\b_1}, \b_2).$$
If we take $x = \bmat{1 & 1 \\ -1 & -\l}$, then $x_\sigma = \bmat{1 & 1 \\ 1 & 1-\l}$ and $\l_\sigma = 1-\l$.
Therefore, 
$$ d_\sigma = \dfrac{d_{x_\sigma}}{d_x * Q_\sigma} = \dfrac{( 1, 1, 1, 1-\l)}{(-1, 1, 1, -\l)*Q_\sigma} = (-1, 1, 1, \dfrac{\l-1}{\l} ).$$
By Theorem \ref{N of X_mn} and Theorem \ref{isom between X22 and X22-sigma} \ref{cor of isom between X22 and X22-sigma}, we have 
$$  \hF{\a_1, \a_2}{\a_1\a_2\b_1, \ol{\b_2}}{\l} = \a_1(-1)\dfrac{j(\a_1,\b_1)}{j(\a_1, \ol{\a_1\a_2\b_1})} \b_2\Big( \dfrac{\l}{\l-1} \Big)\hF{\a_1,\a_2}{\ol{\b_1},\ol{\b_2}}{1-\l}, $$
when $\a_1\b_1,\a_2\b_2, \a_2\b_1  \neq \e$ (the identity is given in \cite[Theorem 3.15]{Otsubo} when $\b_2 = \e$).

\begin{rem}[The reducible case $\l=1$]
When $\l=1$, the following analogue of the Euler-Gauss summation formula is well known (cf. \cite[Theorem 4.3]{Otsubo}):
$$ \hFF{2}{1}{\a,\b}{\c}{1} = \dfrac{j(\a, \b\ol\c)}{j(\a,\ol\c)}\quad (\a \neq \c \mbox{ and } \b \neq \e).$$
We now give a geometric interpretation of this identity using \eqref{N of Fer}.
For $a,b \in \k^*$, put
$$ {}_2^{}X_{2,1}^{a,b}  := \{ (x_i,y_i) \in {}_2X_{2,1} \mid y_2=ax_1, y_1=bx_2\} \subset {}_2X_{2,1}.$$
If $(x_i,y_i) \in {}_2X_{2,1}$, then we have $x_1^N = y_2^N$ and $x_2^N = y_1^N$, and hence 
$${}_2X_{2,1} = \bigsqcup_{a,b \in \k^*} {}_2^{} X_{2,1}^{a,b}.$$
We have the isomorphism
$$ i^{a,b} \colon     Fer_2^* \longrightarrow {}_2^{} X_{2,1}^{a,b} \, ;\, (u,v) \longmapsto  ( u, v, bv, au ).$$
The subgroup $G' = \{ (\xi,\xi',\xi',\xi) \mid \xi,\xi' \in \k^* \} \cong (\k^*)^2$ of $G$ acts on ${}_2^{} X_{2,1}^{a,b}$.
We have, for $\chi \in \widehat G$,
\begin{align*}
N({}_2X_{2,1};\chi)
& = \dfrac{1}{\# G}\sum_{g \in G} \chi(g) \#\{ P \in {}_2X_{2,1}(\ol\k) \mid {\rm Frob} (P) = g \cdot P\} \\
 &= \dfrac{1}{\# G}\sum_{g \in G} \chi(g) \sum_{a,b \in \k^*} \#\{ P \in {}_2^{} X_{2,1}^{a,b}(\ol\k) \mid {\rm Frob}(P)= g\cdot P \} \\
 &= \dfrac{1}{\# G}\sum_{a,b } \sum_{g' \in G'} \chi|_{G'}(g') \#\{ P \in {}_2^{} X_{2,1}^{a,b} (\ol\k) \mid  {\rm Frob}(P) = g'\cdot P\}\\
 &= \dfrac{\# G'}{\# G} \sum_{a,b} N({}_2^{} X_{2,1}^{a,b}; \chi|_{G'}).
\end{align*}
Here, $\chi|_{G'} \in \widehat{G'}$ is the restriction of $\chi$.
Let $\pi \colon (\k^*)^2 \rightarrow G' \, ;\, (\xi , \xi' ) \mapsto (\xi, \xi', \xi', \xi)$ be an isomorphism.
Then, one shows $i^{a,b} \circ (\xi,\xi') = \pi(\xi,\xi') \circ i^{a,b}$ for all $(\xi,\xi') \in (\k^*)^2$.
Thus, by Lemma \ref{Comparison of N by isom} and the equalities above, we have 
$$N({}_2X_{2,1};\chi) = N(Fer_2^*; \pi^*(\chi|_{G'})).$$
If we put $\chi = (\a, \b, \ol\c,\e)$ ($\a\neq \c, \b \neq \e$), then the identity is equivalent to the Euler-Gauss summation formula because $\pi^*(\chi|_{G'}) = (\a, \b\ol\c)$.
\end{rem}

\begin{rem}
Except for the point-counting statements, the argument of this subsection works in characteristic zero.
Then, we have to replace $\k_N$ with $\k(\{ \sqrt[N]{a} \mid a \in \k^* \})$.
The same is true for Subsections \ref{subsec. of X_D} and \ref{subsec. of X_A}.
\end{rem}

\subsection{Kummer's function ($d=2, n=4, \D=(1,1,2)$)}\label{subsec. of X12}
Recall that
$$W_\D = \left\{ {\rm diag}(P_\sigma, \mu(c)) \in GL_4(\k) \middle|\, \sigma \in \S_2, c \in \k^* \right\},\quad \mu(c) = \bmat{1 & 0 \\ 0 & c},$$
and that, for $z = (z_1, z_2, z_3, z_4) \in M(2,4;\k)$, $X_{\D,z} \subset \mathbbm{A}^6 = \{(t_1,t_2,t_3,u,s_1,s_2)\}$ is defined by 
$$ (t_1^N, t_2^N, t_3^N, t_3^N(u^q - u) ) = (sz_1, sz_2, sz_3, sz_4), \quad \prod_{i=1}^3 t_i \neq 0. $$
For $ \l \in \k^*$, fix an element $x=(x_{ij}) \in M(2,2;\k)$ such that $\l= \dfrac{x_{21}x_{12}}{x_{11}}$, and let $X_x \subset \P^2\times \mathbbm{A}^1 = \{(u_1:u_2:u_3, v)\}$ be a variety defined by
$$ \begin{cases} u_1^N = x_{11}u_2^N + x_{21}u_3^N \vspace{3pt} \\ u_3^N(v^q-v) = x_{12}u_2^N + x_{22}u_3^N \\ u_1u_2u_3 \neq 0. \end{cases}$$
For $z = \bmat{x_{11} & 1 & 0 & x_{12} \\ x_{21} & 0 & 1 & x_{22}}$, there is a morphism
$$ \p \colon X_{\D,z} \longrightarrow X_x \, ;\, ((t_i),u,s) \longmapsto ( t_1:t_2:t_3, u). $$
Recall
$$ {}_1X_{2,\l} \colon \begin{cases} x^N + y^N =1 \\ t^q-t = z^N \\ \l x^N = y^Nz^N \\ xyz \neq 0. \end{cases}$$
Define
$$ \theta_0 = \bmat{0 & 1 &0\\ -1 & -1 &0\\ 0 & 0&0}, \quad d_x = ( x_{11}, x_{21}, x_{12}).$$
Then, we have the morphism $ \p_{\theta_0} \colon {}_1X_{2,\l} \otimes \k_{pN} \longrightarrow X_x \otimes \k_{pN}$ given by 
\begin{align*}
(x,y,z,t) \longmapsto (u_1: u_2:u_3, v) 
& = \big( ( \sqrt[N]{d_x}^{-1} (x:y:z) )*\theta_0, t + r(x_{22}) \big).
%& = \Big( \dfrac{\sqrt[N]{x_{21}}}{ y}: \dfrac{\sqrt[N]{x_{21}}}{\sqrt[N]{x_{11}}}\cdot \dfrac{x}{y}:1, t + r(x_{22}) \Big).
\end{align*}

\begin{rem}\label{rem of 1X2 and Xx}
We have the decomposition
$$ {}_1X_{2,\l} \otimes \k_{pN} = \bigsqcup_{\tau^N=\l} {}_1X_{2,\l}^\tau,$$
where
$$ {}_1X_{2,\l}^\tau := \{ (x,y,z,t) \in {}_1X_{2,\l} \otimes \k_{pN} \mid \tau x = yz\}.$$
As in Remark \ref{rem of 2X2 and Xx}, the restriction of $\p_{\theta_0}$ induces the isomorphism 
$$\p_{\theta_0}^\tau \colon {}_1X_{2,\l}^\tau \rightarrow X_x \otimes \k_{pN}\, ;\, (x,y,z,t) \mapsto (u_1:u_2:u_3,v),$$
where
$$(u_1:u_2:u_3) = (1,1,\tau')\big( (\sqrt[N]{d_x}^{-1} (x:y:z) )*\theta \big),\quad v = t + r(x_{22}).$$
Here, 
$$ \theta := \bmat{0 & 1 & 1 \\ -1 & -1 & -1 \\ 0 & 0 & -1}, \quad \tau' := \tau ( \sqrt[N]{d_x}*\bmat{1\\-1\\-1} ).$$
Then, $(\p_{\theta_0}^\tau)^{-1} (u_1:u_2:1, v) = (x,y,z,t)$, where
$$ (x,y,z) = \dfrac{\sqrt[N]{d_x}}{(1,1,\tau')*\theta^{-1}} ( (u_1,u_2,1)*\theta^{-1} ),\quad t = v - r(x_{22}).$$
\end{rem}

For $w = {\rm diag}(P_\sigma, \mu(c)) \in W_\D$, write $(z_{ij}^w)=zw$.
Let $x_w:=(x_{ij}^w)$ be the matrix given by 
$$z_w:=\bmat{z_{12}^w & z_{13}^w \\ z_{22}^w & z_{23}^w}^{-1}zw = \bmat{x_{11}^w & 1 & 0 & x_{12}^w \\ x_{21}^w & 0 & 1 & x_{22}^w}.$$
Note that $x_{11}^w, x_{21}^w, x_{12}^w \neq 0$.
By Theorem \ref{isom fw} and Theorem \ref{isom between X_D} \ref{isom Lg}, we have the isomorphism
$$ f_w \colon X_{\D,z} \longrightarrow X_{\D,z_w} \, ;\, ((t_i),u,s) \longmapsto ( (t_1,t_2)*P_\sigma, t_3, cu, s') \quad (s'= s \bmat{z_{12}^w & z_{13}^w \\ z_{22}^w & z_{23}^w}). $$
We also have the isomorphism
$$ g_w \colon X_x \longrightarrow X_{x_w}\, ;\,  (u_1:u_2:u_3, v) \longmapsto ( (u_1:u_2)*P_\sigma:u_3, cv),$$
and it satisfies  $g_w \circ \p = \p \circ f_w$.
Put
$$ \l_w := \dfrac{x_{21}^w x_{12}^w}{x_{11}^w} = {\rm sgn}(\sigma) c\l, \quad Q_\sigma = \begin{cases} I_3 & (\sigma = {\rm id}) \vspace{3pt} \\ \bmat{ -1 & -1 & -1 \\ 0 & 1 & 0 \\ 0 & 0 & 1} & (\sigma = (1\, 2)). \end{cases}$$

%
%For $(x,y,z,t) \in {}_1X_{2,\l}$, define $h_w(x,y,z,t)$ by
%$$ (x,y,z) \mapsto \dfrac{\sqrt[N]{d_{x_w}}}{\sqrt[N]{d_x} * Q_\sigma} \big( ( x,y,z )*Q_\sigma \big), \quad t \mapsto ct + r(c x_{22}-x_{22}^w).  $$
For $\tau \in \{ \tau \mid \tau^N=\l\}$ and $w \in W_\D$, put
$$ \tau_w = \tau \cdot \dfrac{\sqrt[N]{d_x}*{}^t (1,-1,-1)}{\sqrt[N]{d_{x_w}}*{}^t (1,-1,-1)} \in \{ \tau \mid \tau^N = \l_w\}.$$
By the isomorphism in Remark \ref{rem of 1X2 and Xx}, we have the isomorphism
$$ h_w^\tau := (\p_{\theta_0}^{\tau_w})^{-1} \circ g_w \circ \p_{\theta_0}^\tau \colon {}_1X_{2,\l}^\tau \longrightarrow {}_1X_{2,\l_w}^{\tau_w}.$$
Put $G = (\k^*)^3 \times \k$ and ${\rm Gal} = {\rm Gal}(\k_{pN}/\k)$.
For $\sigma = (1\, 2)$, the following theorem gives a geometric interpretation of the formula \eqref{Kummer prod}.
\begin{thm}\label{N of X12}
Let $w={\rm diag}(P_\sigma, \mu(c)) \in W_\D$ and $\tau \in \{\tau \mid \tau^N=\l\}$.
\begin{enumerate}
\item \label{N of X12 (i)} The isomorphism $h_w^\tau \colon (x,y,z,t) \mapsto (x',y',z',t')$ is given by
$$ (x',y',z') = \dfrac{\sqrt[N]{d_{x_w}}}{\sqrt[N]{d_x} * Q_\sigma} \big( ( x,y,z )*Q_\sigma \big), \quad t' = ct + r(c x_{22}-x_{22}^w).  $$
By the same correspondence, we have the isomorphism
$$ h_w \colon {}_1X_{2,\l} \otimes \k_{pN} \longrightarrow {}_1X_{2,\l_w} \otimes \k_{pN}.$$

\item For any $w' \in W_\D$, we have $h_{w'}^{\tau_w} \circ h_w^{\tau} = h_{ww'}^\tau$.

\item \label{N of X12 (iii)} Put $d_w = d_{x_w}/(d_x*Q_\sigma) \in (\k^*)^3$.
Define the automorphism $\pi_w $ of $G \times {\rm Gal}$ by 
$$ \big( \bm\xi,a , e \big) \longmapsto \big(   K_{d_w}(e)(\bm\xi*Q_\sigma), \, ca+A_{cx_{22}-x_{22}^w}(e) , e\big) \quad (\bm\xi \in (\k^*)^3, a \in \k).$$
Then, for $(g,e) \in G \times {\rm Gal}$,
$$ h_w \circ (g,e) = \pi_w(g,e) \circ h_w. $$

\item  \label{N of X12 (iv)} For $\chi = ((\a_i)_i,\psi) \in \widehat{G}$, put
$$ \chi_w = \big( (\a_i)_i*{}^t Q_\sigma, \psi_c \big).$$
Then, we have
$$ \chi(d_w, cx_{22}-x_{22}^w) N({}_1X_{2,\l}; \chi_w) = N({}_1X_{2,\l_w}; \chi). $$
\end{enumerate}
\end{thm}
\begin{proof}
(i) Put $T = \theta - \theta_0$ and put
$$ M = \bmat{1 & 0 & 0 \\ 0 & 1 & 0 \\ -1 & -1 & 0}, \quad P_\sigma' = \left( \begin{array}{c|c}
P_\sigma & \mat{0\\0} \\
\hline 
\mat{0 & 0} & 1
\end{array} \right) .$$
One shows $\theta_0 P_\sigma' M + T = Q_\sigma$ for each $\sigma \in \S_2$.
This proves the assertion by the same argument as in the proof of Theorem \ref{isom between X22 and X22-sigma} (i).

(ii) This follows from $(x_w)_{w'} = x_{ww'}$ and $Q_\sigma Q_{\sigma'} = Q_{\sigma \sigma'}$ ($\sigma, \sigma' \in \S_2$).

(iii) We can check this  by computing on the structure sheaves.

(iv) By Lemma \ref{Comparison of N by isom} and (iii), we have
$$ N({}_1X_{2,\l} \otimes \k_{pN}; \pi_w^*(\chi,\rho)) = N({}_1X_{2,\l_w} \otimes \k_{pN};(\chi,\rho)) \quad ( (\chi,\rho) \in \widehat{G\times {\rm Gal}}).$$
Noting \eqref{character*matrices}, one shows
\begin{align*}
\pi_w^*(\chi,\rho) (g,e)
= \chi_w(g) \times \chi (K_{d_w}(e), A_{cx_{22}-x_{22}^w}(e))  \times \rho(e).
\end{align*}
Hence, if we put $\eta = \chi \circ (K_{d_w}, A_{cx_{22}-x_{22}^w}) \in \widehat{\rm Gal}$, then we have $\pi_w^*(\chi,\rho) = (\chi_w, \eta\rho)$.
Thus, we obtain the identity of the theorem by \eqref{N of base change}.
\end{proof}

\subsection{Lauricella's $F_D$ ($d=2, n\geq 5, \D=(1,\dots,1)$)}\label{subsec. of X_D}
We can generalize the result in Gauss' case ($d=2, n=4, \D=(1,1,1,1)$) to $n \geq 5$, which corresponds to Lauricella's $F_D^{(n-3)}$. 
Now, recall that $W_\D  =  \S_{n}$ and the variety $X_{\D,z}$ is defined by
$$ \begin {cases} t_i^N = s_1z_{1i} + s_2z_{2i} & (i=1, \dots, n) \\ t_1\cdots t_n \neq 0. \end{cases}$$
For $\l_1, \dots, \l_m \in \k^*$ ($m:=n-3$), fix an element $x:= (x_{ij}) \in M(2, m+1;\k)$ such that
$$ \l_i = \dfrac{x_{11}x_{2(i+1)}}{x_{21}x_{1(i+1)}}\quad (i=1, \dots, m).$$
Put $z=\left( \begin{array}{c|c}x& I_2 \end{array}\right)$.
Define a variety $X_x \subset \mathbbm{P}^{n-1}$ by the equations
$$\begin{cases}
 u_i^N = x_{1i}v_1^N + x_{2i}v_2^N & (i=1, \dots, n-2) \vspace{3pt} \\ u_1\cdots u_{n-2}v_1v_2 \neq 0. 
 \end{cases}$$
Then, we have the morphism $\p \colon X_{\D,z} \rightarrow X_x$ given by
$$(t, s) \longmapsto (u_1 : \cdots : u_{n-2}:v_1:v_2) = (t_1:\cdots : t_{n}).$$
Define a matrix
%$$ \theta = -( e_1, e_{m+3}, e_{m+4}, \dots, e_{2m+2}, e_1-e_{m+2}, \bm 0) \in M(2m+2, m+3;\Z). $$
$$\theta_0 = \left(\begin{array}{c|c|c|c}
-1 & 0 \quad \cdots \quad  0& -1 & 0\\
\hline
\mat{0 \\ \vdots \\ 0} & O_m & \mat{0 \\ \vdots \\ 0} & \mat{0 \\ \vdots \\ 0} \\
\hline 
0 & 0 \quad \cdots \quad 0 & 1 & 0 \\
\hline 
\mat{0 \\ \vdots \\ 0} & -I_m & \mat{0 \\ \vdots \\ 0} & \mat{ 0 \\ \vdots \\ 0}
\end{array}\right) \in M(2m+2, n; \Z),$$
and put
$$ d_x = (x_{21},  x_{12}, \dots, x_{1(m+1)}, x_{11}, x_{22}, \dots, x_{2(m+1)}).$$
Then, we obtain a morphism $\p_{\theta_0} \colon  X_{F_D^m,\bm\l} \otimes \k_N \rightarrow X_x \otimes \k_N$ given by
\begin{align*}
 (x_i,y_i) 
 \longmapsto & \big(  \sqrt[N]{d_x}^{-1} (x_i : y_i) \big)*\theta_0 ,
%&  = ( \dfrac{\sqrt[N]{x_{21}}}{x_0} : \dfrac{\sqrt[N]{x_{22}}}{y_1} : \cdots : \dfrac{\sqrt[N]{x_{2(m+1)}}}{y_m} : \dfrac{\sqrt[N]{x_{21}}}{\sqrt[N]{x_{11}}}\cdot \dfrac{y_0}{x_0}:1),
\end{align*}
where we wrote $(x_i,y_i) = (x_0, \dots, x_m, y_0, \dots, y_m)$.
As in  Remark \ref{rem of 2X2 and Xx}, we have the decomposition
$$ X_{F_D^m, \bm\l} \otimes \k_N = \bigsqcup_{\bm \tau^N = \bm\l} X_{F_D^m,\bm\l}^{\bm\tau} \quad (\bm\tau=(\tau_1,\dots,\tau_m) ),$$
where
$$ X_{F_D^m,\bm\l}^{\bm\tau} := \{ (x_i,y_i) \in X_{F_D^m,\bm\l}\otimes\k_N \mid \tau_i x_0x_i = y_0y_i \mbox{ for all }i \}.$$
By restriction, $\p_{\theta_0}$ induces the isomorphism $ \p_{\theta_0}^\tau \colon X_{F_D^m,\bm\l}^{\bm\tau} \rightarrow X_x \otimes \k_N $, where the inverse morphism is given below.
For $j=1, \dots, m$,
$$ T_j :=\left( \begin{array}{c|c} 
   \mat{O} & \mat{1 \\ e_j \\ -1 \\ -e_j} 
    \end{array}\right) \in M(2m+2,n;\Z) , \quad \tau_j' := \tau_j \cdot ( \sqrt[N]{d_x} * \bmat{ 1 \\ e_j \\ -1 \\ -e_j}),$$
and put 
$$\theta_j = \theta_0 + T_j,\quad T_0=O,\quad \tau_0'=1.$$
Define $\rho_j \in M(n,2m+2;\Z)$ ($j=0,\dots, m$) by
$$\small \rho_0 =\left( \begin{array}{c|c|c|c}
-1 & \mat{0 & \cdots & 0} & -1 & \mat{0 & \cdots & 0 }\\
\hline
\mat{0 \\ \vdots \\ 0} & -I_m & \mat{0 \\ \vdots \\ 0} & -I_m \\
\hline
\mat{0 \\ 0} & \mat{0 & \cdots & 0\\ 0& \cdots & 0} & \mat{1 \\ 0} &  \mat{0 & \cdots & 0\\ 0 & \cdots & 0}
\end{array} \right),\quad 
\rho_j = \left( \begin{array}{c|c|c|c}
0 & \mat{0 & \cdots & 0} & 0 & \mat{0 & \cdots & 0 }\\
\hline
\mat{0 \\ \vdots \\ 0} & O_m & \mat{0 \\ \vdots \\ 0} & O_m \\
\hline
\mat{0 \\ 0} & \mat{{}^t e_j \\ {}^t e_j} & \mat{0 \\ 0} &  \mat{0 & \cdots & 0\\ 0 & \cdots & 0}
\end{array} \right), $$
and put 
$$\rho= \sum_{j=0}^m \rho_j.$$
The inverse morphism is given by
$$ (\p_{\theta_0}^{\bm\tau})^{-1} (u_1:\cdots : u_{n-2}:v_1:1) = \dfrac{\sqrt[N]{d_x}}{\prod_{j=1}^m (  (1,\dots, 1, \tau_j')*\rho_j)} ( (u_1,\dots,u_{n-2},v_1,1 )*\rho ).$$
We can show this as follows.
Noting that
$
(1,\dots,1,\tau_j') \big( (\sqrt[N]{d_x}^{-1}(x_i,y_i))*T_j \big) = (1,\dots,1)
$
for $(x_i,y_i) \in X_{F_D^m,\bm\l}^{\bm\tau}$,
 we have
 \begin{equation}
(  \sqrt[N]{d_x}^{-1} (x_i,y_i) ) *A = (1, \dots, 1, \tau_j') \big( (\sqrt[N]{d_x}^{-1}(x_i,y_i) )*(A+T_j) \big) \label{eq for Tj}
 \end{equation}
 for $A \in M(2m+2, n;\Z) $.
Therefore, one shows
\begin{align*}
( \sqrt[N]{d_x}^{-1} (x_i, y_i))*\theta_0\rho
& = \prod_{j=0}^m ( \sqrt[N]{d_x}^{-1} (x_i, y_i))*\theta_0\rho_j\\
& = \prod_{j=0}^m \Big( (1,\dots, 1,\tau_j')\big( (\sqrt[N]{d_x}^{-1}(x_i, y_i))*\theta_j\big) \Big)* \rho_j \\
& = (\prod_{j=1}^m  (1,\dots, 1,\tau_j')*\rho_j) \cdot \big( (\sqrt[N]{d_x}^{-1}(x_i, y_i) )* \sum_{j=0}^m \theta_j \rho_j\big).
\end{align*}
Using this identity and the fact that $\sum_{j=0}^m \theta_j \rho_j = I_{2m+2}$, we see $(\p_{\theta_0}^{\bm \tau})^{-1} \circ \p_{\theta_0}^{\bm \tau} = {\rm id}$.
On the other hand, noting $\rho \theta_0 = I_n-\left(\begin{array}{c|c} O & e_n\end{array}\right)$ and $\rho_j\theta_0 = O_n$ for any $1\leq j \leq m$, we see $\p_{\theta_0}^{\bm \tau}  \circ (\p_{\theta_0}^{\bm \tau})^{-1} = {\rm id}$.

Suppose that $\l_i \neq 1$ and $\l_i \neq \l_j$ for any $i$ and $j \neq  i$.
Then, $[i\, j] \neq 0$ for $z = \left( \begin{array}{c|c}
x & I_2
\end{array} \right)$ when $ i \neq j $.
For $\sigma \in \S_n$, write $(z_{ij}^\sigma) = zP_\sigma$ and put
$$ z_\sigma = \bmat{z_{1(n-1)}^\sigma & z_{1n}^\sigma  \\ z_{2(n-1)}^\sigma & z_{2n}^\sigma }^{-1} zP_\sigma.$$
It can be written as
$$ z_\sigma = \left( \begin{array}{c|c}
x_\sigma  & I_2
\end{array} \right), \quad x_\sigma := (x_{ij}^\sigma) \in M(2, n-2;\k),$$
where $x^\sigma_{ij} \in \k^*$.
Put
$$ \l_{\sigma,i} = \dfrac{x_{11}^\sigma x_{2(i+1)}^\sigma}{x_{21}^\sigma x_{1(i+1)}^\sigma}, \quad \bm\l_\sigma =(\l_{\sigma,i})_i.$$
By Theorem \ref{isom fw} and Theorem \ref{isom between X_D} \ref{isom Lg}, we have the isomorphism
$$ f_\sigma \colon X_{\D,z} \longrightarrow X_{\D,z_\sigma}\, ;\,  (t,s) \longmapsto (t*P_\sigma , s') \quad (s' = s\bmat{z_{1(n-1)}^\sigma & z_{1n}^\sigma  \\ z_{2(n-1)}^\sigma & z_{2n}^\sigma }).$$
Clearly, we have the isomorphism
$$ g_\sigma \colon X_x \longrightarrow X_{x_\sigma} \, ;\, (u_1 : \cdots : v_2) \longmapsto (u_1 : \cdots : v_2)*P_\sigma,$$
and it satisfies $\p \circ f_\sigma = g_\sigma \circ \p$.

Define
$$Q_\sigma = \sum_{j=0}^m ( \theta_0 P_\sigma M + T_j)\rho_j ,$$
where
$$ M := \left( \begin{array}{c|c}
I_{n-1} & \mat{0 \\ \vdots \\ 0} \\
\hline
\mat{-1 & \cdots & -1} & 0
\end{array}\right) \in M(n,n;\Z).$$
\begin{lem}\label{lem of Q_sigma for FD}
For any $\sigma, \sigma' \in \S_n$, we have 
$$ Q_\sigma Q_{\sigma'}=Q_{\sigma\sigma'}.$$
In particular, $Q_\sigma \in GL_{2m+2}(\Z)$.
\end{lem}

\begin{proof}
One shows
$M \rho_0 \theta_0 = M$, $MP_\sigma M=P_\sigma M$ and $T_j\rho_j\theta_0=O$, and hence, 
$$ Q_\sigma \sum_j \theta_0 P_{\sigma'}M \rho_j = \sum_j \theta_0 P_{\sigma\sigma'}M\rho_j.$$
We have $Q_\sigma T_i= T_i$ for any $i = 0, \dots, m$, where note that
$ \rho_j T_i = O$ for $j \not\in\{ 0,i\}$ and that $(\rho_0 + \rho_i)T_i = \left(\begin{array}{c|c} O & e_n \end{array} \right).$
These identities imply
\begin{align*}
Q_\sigma Q_{\sigma'} 
& = Q_\sigma \sum_j \theta_0 P_{\sigma'}M \rho_j + Q_\sigma  \sum_j T_j \rho_j 
 = \sum_j \theta_0 P_{\sigma\sigma'}M\rho_j + \sum_j T_j\rho_j =  Q_{\sigma\sigma'}.
\end{align*} 
In particular, $Q_{\sigma^{-1}} = Q_\sigma^{-1}$.
\end{proof}

For $\bm\tau = (\tau_j)_j \in \{ \bm\tau \mid \bm\tau^N = \bm\l \}$ and $\sigma \in \S_n$, put $ \bm\tau_\sigma = (\tau_{\sigma, j})_j \in \{ \bm\tau \mid \bm\tau^N = \bm\l_\sigma\}$, where
$$ \tau_{\sigma,j} := \tau_j \cdot \dfrac{ \sqrt[N]{d_x}* {}^t(1, {}^t e_j , -1, -{}^t e_j) }{\sqrt[N]{d_{x_\sigma}}* {}^t(1, {}^t e_j , -1, -{}^t e_j)}.$$
Then, we have the isomorphism
$$h_\sigma^{\bm\tau} :=  (\p_{\theta_0}^{\bm\tau_\sigma})^{-1} \circ g_\sigma \circ \p_{\theta_0}^{\bm\tau} \colon X_{F_D^m,\bm\l}^{\bm\tau} \longrightarrow X_{F_D^m,\bm\l_\sigma}^{\bm\tau_\sigma}.$$
Let $G = (\k^*)^{2m+2}$ and ${\rm Gal} = {\rm Gal}(\k_N/\k)$.
\begin{thm}\label{isom between X_FD}
Let $\sigma \in \S_n$ and $\bm\tau \in \{ \bm\tau \mid \bm\tau^N = \bm\l\}$.
\begin{enumerate}
\item The isomorphism $h_\sigma^{\bm\tau}$ is given by
$$ (x_i,y_i) \longmapsto \dfrac{\sqrt[N]{d_{x_\sigma}}}{\sqrt[N]{d_x} * Q_\sigma} \big( (x_i, y_i ) * Q_\sigma \big).$$
By the same correspondence, we have the isomorphism
$$h_\sigma \colon X_{F_D^m,\bm\l} \otimes \k_N \longrightarrow X_{F_D^m,\bm\l_\sigma} \otimes \k_N.$$

\item For any $\sigma' \in \S_n$, we have $h_{\sigma'}^{\bm\tau_{\sigma}} \circ h_{\sigma}^{\bm\tau}=h_{\sigma \sigma'}^{\bm\tau}$.

\item Put $d_\sigma = d_{x_\sigma}/ (d_x * Q_\sigma) \in G$.
Define the automorphism $\pi_\sigma$ of $G\times {\rm Gal}$ by
$$ (g,e) \mapsto (K_{d_\sigma}(e) (g*Q_\sigma) , e).$$
Then, we have
$$ h_\sigma \circ (g,e) = \pi_\sigma(g,e) \circ h_\sigma.$$

\item For $\chi \in \widehat{ G}$, 
$$ \chi(d_\sigma) N(X_{F_D^m,\bm\l}; \chi * {}^tQ_\sigma ) = N(X_{F_D^m,\bm\l_\sigma}; \chi). $$
\end{enumerate}
\end{thm}
\begin{proof}
(i) Note that
\begin{align*}
g_\sigma \circ \p_{\theta_0}^{\bm\tau}(x_i,y_i)
& = ( \sqrt[N]{d_x}^{-1} (x_i : y_i) ) * \theta_0 P_\sigma M.
%& = ( \sqrt[N]{d_x}^{-1} (x_i : y_i) ) * Q_\sigma \theta_0.
%& = (1,\dots,1,\bm\tau') \big( (\sqrt[N]{d_x}^{-1} (x_i: y_i) ) * (Q_\sigma \theta_0 + T_j) \big)\quad (j=1,\dots,m).
\end{align*} 
Thus, (put $\tau_{\sigma , j}' = \tau_{\sigma ,j} \cdot (\sqrt[N]{d_{x_\sigma}}*{}^t(1, {}^t e_j , -1 , -{}^t e_j)$)
\begin{align*}
 h_\sigma^{\bm\tau} (x_i,y_i) 
& = \dfrac{\sqrt[N]{d_{x_\sigma}}}{\prod_j (1,\dots,1,\tau_{\sigma,j}')*\rho_j}  \big( (\sqrt[N]{d_x}^{-1}(x_i,y_i))*\theta_0 P_\sigma M \rho \big).
\end{align*}
Noting \eqref{eq for Tj}, we have
\begin{align*}
(\sqrt[N]{d_x}^{-1}(x_i,y_i))*\theta_0 P_\sigma M \rho
& = \prod_{j=0}^m ( \sqrt[N]{d_x}^{-1}(x_i,y_i) )*\theta_0 P_\sigma M \rho_j\\
& = \prod_{j=0}^m \Big( (1,\dots,1,\tau_j')  \big( ( \sqrt[N]{d_x}^{-1}(x_i,y_i) )*(\theta_0 P_\sigma M + T_j) \big) \Big)*\rho_j \\
& = ( \prod_{j=1}^m (1,\dots,1,\tau_j')*\rho_j )\cdot \big( (\sqrt[N]{d_x}^{-1}(x_i,y_i))*Q_\sigma.
\end{align*}
Noting $\tau_j' = \tau_{\sigma,j}'$, we obtain (i).
(ii)--(iv) can be proved by the same argument as in the proof of Theorem \ref{isom between X22 and X22-sigma} (ii)--(iv) and using Lemma \ref{lem of Q_sigma for FD}.
\end{proof}

By (iv) above and Theorem \ref{N of X_Lauricella} \ref{N of X_Lauricella-D}, we can obtain $n!$ transformation formulas for $F_D^{(m)}$.
%\begin{exa}[$n=5, m=2, \sigma = (1\, 2\, 3\, 4\, 5)$]
%We have
%$$ Q_\sigma = \bmat{ -1 & 1 & 2 & 0 & -1 & 0 \\ 0 & 1 & 0 & 0 & 0 & 0  \\ 0 & 0 & 1 & 0 & 0 & 0 \\ 0 & -1 & -2 & 0 & 0 & -1 \\ 1 & -1 & 0 & 1 & 0 & 0 \\ 0 & 1 & -1 & 0 & 1 & 0 }.$$
%If we take $x = \bmat{1 & 1 & 1 \\ -1 & -\l_1 & -\l_2}$, then $x_\sigma = \bmat{ 1 - \l_1 & 1 - \l_2 & 1 \\ 1 & 1 & 1}$,  and hence, we have 
%$$  \bm \l_\sigma = ( \dfrac{1-\l_1}{1-\l_2} , 1-\l_1), \quad d_\sigma = (  \dfrac{1}{\l_1}, \dfrac{\l_1(\l_2-1)}{\l_2}, -\l_2, \dfrac{\l_1-1}{\l_1}, \dfrac{1}{\l_2}, 1).$$
%By Theorem \ref{N of X_Lauricella} \ref{N of X_Lauricella-D} and Theorem \ref{isom between X_D} (iii), we have
% 
%\end{exa}

\begin{rem}
Let us consider the reducible case $\l_{m-1} =  \l_m $.
Then, $F_D^{(m)} (\bm\l)$ reduces to $F_D^{(m-1)}(\bm\l')$ (cf. \cite[Theorem 3.23 (i)]{N2}), where $\bm\l' = (\l_1,\dots,\l_{m-1})$.
Analogously, $X_{F_D^m,\bm\l}$ can be reduced to $X_{F_D^{m-1},\bm\l'}$ as follows.
For $a,b \in \k^*$, put
$$ X_{F_D^m,\bm\l}^{a,b} := \left\{ (x_i, y_i) \in X_{F_D^m,\bm\l} \middle| \, (x_m, y_m) = (ax_{m-1}, by_{m-1}) \right\}.$$
Now,  for $(x_i,y_i) \in X_{F_D^m,\bm\l}$, noting that $y_{m-1}^N /x_{m-1}^N  = y_m^N / x_m^N$ and $x_i^N + y_i^N = 1$, we have $x_{m-1}^N = x_m^N$ and $y_{m-1}^N = y_m^N$.
Hence, we have the decomposition
$$ X_{F_D^m,\bm\l} = \bigsqcup_{a,b \in \k^*} X_{F_D^m,\bm\l}^{a,b}.$$
Clearly, we have the isomorphism $ i^{a,b} \colon X_{F_D^{m-1},\bm\l'}  \rightarrow X_{F_D^{m},\bm\l}^{a,b} $ by
$$ (x_i, y_i) \longmapsto (x_0, \dots, x_{m-1}, ax_{m-1}, y_0, \dots, y_{m-1}, by_{m-1}).$$
The reduction formula \cite[Theorem 3.23 (i)]{N2} can be recovered as follows.
The subgroup $G' = \{ (a_0, \dots, a_{m-1}, a_{m-1}, b_0, \dots, b_{m-1}, b_{m-1})\} \subset (\k^*)^{2m+2}$ acts on $X_{F_D^m,\bm\l}^{a,b}$.
Let $\pi \colon (\k^*)^{2m}  \rightarrow G'$ be the natural isomorphism.
By Lemma \ref{Comparison of N by isom}, for $\chi \in \widehat{(\k^*)^{2m+2}}$, we have
\begin{align*}
N(X_{F_D^m,\bm\l}; \chi) 
&= \dfrac{1}{(q-1)^2} \sum_{a,b \in \k^*} N(X_{F_D^m,\bm\l}^{a,b} ; \chi |_{G'})\\
& =  N(X_{F_D^{m-1},\bm\l'}; \pi^*(\chi|_{G'})).
\end{align*}
If we write $\chi = ( (\a_i)_i, (\b_i)_i)$, then 
$$\pi^*(\chi|_{G'}) = (\a_0, \dots, \a_{m-2}, \a_{m-1}\a_m, \b_0,\dots,\b_{m-2},\b_{m-1}\b_m).$$
By Theorem \ref{N of X_Lauricella} \ref{N of X_Lauricella-D}, if we take $\b_1=\cdots=\b_m=\e$, we can recover the reduction formula.
\end{rem}

\subsection{Humbert's $\Phi_1$ ($d=2, n= 5, \D=(1,1,1,2)$)}\label{subsec. of X_Humbert}
For $z = (z_1, \dots, z_5) \in M(2,5;\k)$, recall that the variety $X_{\D,z}$ is defined by
$$ \begin{cases} t_i^N = sz_i & (i=1, \dots, 4) \\ t_4^N(u^q-u) = sz_5 \\ t_1\cdots t_4 \neq 0. \end{cases}$$
For $\bm\l=(\l_1, \l_2) \in (\k^*)^2$, fix an element $x=(x_{ij}) \in M(2,3;\k)$ such that
$$ (\l_1, \l_2) = \big( \frac{x_{21}x_{12}}{x_{11}x_{22}}, \frac{x_{21}x_{13}}{x_{11}} \big).$$
Put
$$ z = \bmat{ x_{11} & x_{12} & 1 & 0 & x_{13} \\ x_{21} & x_{22} & 0 & 1 & x_{23}}. $$
Note that $[i\, 4]\neq 0$ ($i = 1,2,3$) for $z$.
Define a variety $X_x \subset \P^3 \times \mathbbm{A}^1 =\{ (u_1:\cdots:u_4, v) \}$ by
$$ \begin{cases} u_1^N = x_{11}u_3^N + x_{21}u_4^N\\ u_2^N = x_{12}u_3^N + x_{22}u_4^N \\ u_4^N(v^q-v) = x_{13}u_3^N + x_{23}u_4^N \\ u_1u_2u_3u_4 \neq 0. \end{cases}$$
There is a morphism $\p \colon X_{\D,z} \rightarrow X_x; ((t_i),u,s) \mapsto  (t_1:t_2:t_3:t_4, u)$.
Put $d_x = (x_{11}, x_{22}, x_{21},x_{12}, x_{13})$ and
$$ \theta_0 = \bmat{0&0&1&0\\ 0&-1&0&0 \\ -1&0&-1&0\\ 0&0&0&0 \\ 0&0&0&0}.$$
We have the morphism $ \p_{\theta_0} \colon X_{\Phi_1,\bm\l} \otimes \k_{pN} \rightarrow X_x \otimes \k_{pN}$ given by
$$ (x_i, y_i, z, t) \longmapsto \big( (\sqrt[N]{d_x}^{-1}(x_i:y_i:z))*\theta_0, t + r(x_{23}) \big).$$
As in Remark \ref{rem of 2X2 and Xx}, $X_{\Phi_1,\bm\l} \otimes \k_{pN}$ decomposes into the $\bm\tau$-components
$$X_{\Phi_1, \bm\l}^{\bm \tau} := \{ (x_i,y_i,z,t) \in X_{\Phi_1,\bm\l}\otimes \k_{pN} \mid \tau_1 x_1x_2 = y_1y_2, \tau_2 x_1 = y_1z\}\quad (\bm\tau^N=\bm\l).$$
The restriction of $\p_{\theta_0}$ induces the isomorphism $\p_{\theta_0}^{\bm\tau} \colon  X_{\Phi_1,\bm\l}^{\bm\tau} \rightarrow X_x\otimes \k_{pN}$.
Put $\tau_j' = \tau_j( \sqrt[N]{d_x} * (T_j e_4))$ and $\theta_j = \theta_0 + T_j$, where
$$ T_1 := \left(\begin{array}{c|c}
\mat{O} & \mat{1\\1\\-1\\-1\\0}
\end{array}\right), \quad
T_2 := \left(\begin{array}{c|c}
\mat{O} & \mat{1\\0\\-1\\0\\-1}
\end{array}\right) \quad \in M(5,4;\Z).$$
Put $\rho = \sum_{j=0}^2 \rho_j \in M(4,5;\Z)$, where
$$ \rho_0= \left( \begin{array}{c|c}
\mat{-1&0&-1\\ 0 & -1 & 0 \\ 1 & 0 & 0 \\ 0 & 0 & 0} & O
\end{array}\right),\quad 
\rho_1 =  \left( \begin{array}{c|c}
\mat{O} & \mat{0 & 0 \\ -1 & 0 \\ 1 & 0 \\ -1 & 0 }
\end{array}\right),\quad
\rho_2 =  \left( \begin{array}{c|c}
\mat{O} & \mat{0 & 0 \\ 0 & 0 \\ 0 & 1 \\ 0 & -1 }
\end{array}\right).$$
Noting $\sum_{j=0}^2 \theta_j \rho_j = I_5$, $\rho \theta_0 = I_4-\left(\begin{array}{c|c} O&e_4\end{array}\right)$ and $\rho_j\theta_0 = O_4$ for each $j=1,2$, one shows that the inverse morphism of $\p_{\theta_0}^{\bm\tau}$ is given by $(\p_{\theta_0}^{\bm \tau})^{-1}(u_1:u_2:u_3:1, v)=(x_i,y_i,z,t)$, where 
$$ (x_i,y_i,z) = \dfrac{\sqrt[N]{d_x}}{\prod_{j=1}^2 (1,1,1,\tau_j')*\rho_j} \big( (u_1,u_2,u_3,1)*\rho \big), \quad t = v-r(x_{23}).$$

Now, $W_\D = \S_3 \times W(2)$.
For $w = {\rm diag}(P_\sigma, \mu(c)) \in W_\D$, write $(z_{ij}^w) = zw$. 
Let $x_w = (x_{ij}^w) \in M(2,3;\k)$ be the matrix given by
$$ z_w := \bmat{z_{13}^w & z_{14}^w \\ z_{23}^w & z_{24}^w}^{-1}zw = \bmat{x_{11}^w & x_{12}^w & 1 & 0 & x_{13}^w \\ x_{21}^w & x_{22}^w & 0 & 1 & x_{23}^w}.$$
Suppose that $\l_1 \neq 1$. 
Then all entries $x_{ij}^w$, except possibly $x_{23}^w$, are non-zero.
Put
$$ \bm\l_w = \big( \frac{x_{21}^w x_{12}^w}{x_{11}^w x_{22}^w}, \frac{x_{21}^w x_{13}^w }{x_{11}^w} \big).$$
By  Theorems \ref{isom fw} and \ref{isom between X_D} \ref{isom Lg}, we have the isomorphism
$$f_w \colon X_{\D,z} \rightarrow X_{\D,z_w}\, ; \, (t,s) \mapsto \big( (t_1,t_2,t_3)*P_\sigma, t_4, cu, s') \big), \ (s' = s\bmat{z_{13}^w & z_{14}^w \\ z_{23}^w & z_{24}^w}), $$
and clearly we also have the isomorphism
$$ g_w \colon X_x \longrightarrow X_{x_w}\, ;\, (u_1:u_2:u_3:u_4,v) \longmapsto ((u_1:u_2:u_3:u_4)*P_\sigma', cv),$$
where
$$P_\sigma' = \left( \begin{array}{c|c}
P_\sigma & \mat{0\\0\\0} \\
\hline 
\mat{0 & 0 & 0} & 1
\end{array} \right) .$$
Define 
$$ Q_\sigma = \sum_{j=0}^2 (\theta_0 P_\sigma' M + T_j) \rho_j,$$
where $M$ is as in Subsection \ref{subsec. of X22}.
We can prove the following lemma by the same argument as in the proof of Lemma \ref{lem of Q_sigma for FD}.
\begin{lem}\label{lem for Qsigma of Phi_1}
For any $\sigma, \sigma' \in \S_3$, we have
$$ Q_\sigma Q_{\sigma'} = Q_{\sigma\sigma'}.$$
In particular, $Q_\sigma \in GL_5(\Z)$.
%\begin{enumerate}
%\item  $Q_\sigma \in GL_5(\Z)$.
%
%\vspace{3pt}
%\item \label{eq of Q-sigma for Humbert-1}
%$\theta P_\sigma = Q_\sigma \theta$.
%
%\vspace{3pt}
%\item \label{eq of Q-sigma for Humbert-2}
%$
%\left\{ \begin{array}{c}
%Q_\sigma {}^t(1, 1, -1,-1, 0) = {}^t(1, 1, -1,-1, 0) ,\vspace{7pt}\\
%Q_\sigma{}^t(1,0,-1,0,-1)  = {}^t(1,0,-1,0,-1).
%\end{array} \right.
%$
%\end{enumerate}
\end{lem}

We have the isomorphism $h_w^{\bm\tau} := (\p_{\theta_0}^{\bm\tau_w})^{-1} \circ g_w \circ \p_{\theta_0}^{\bm\tau} \colon X_{\Phi_1,\bm\l}^{\bm\tau} \rightarrow X_{\Phi_1,\bm\l_w}^{\bm\tau_w} $, where
$$ \bm\tau_w := (\tau_{w,1}, \tau_{w,2}), \quad  \tau_{w,j} := \tau_j \cdot \dfrac{\sqrt[N]{d_x}*(T_je_4)}{\sqrt[N]{d_{x_w}}*(T_je_4)}.$$
Put $G=(\k^*)^5 \times \k$ and ${\rm Gal} = {\rm Gal}(\k_{pN}/\k)$.
\begin{thm}
Let $w = {\rm diag}(P_\sigma, \mu(c)) \in W_\D$.
\begin{enumerate}
\item The isomorphism $h_w^{\bm\tau}$ is given by 
$$ (x_i,y_i,z,t) \longmapsto \big( \dfrac{\sqrt[N]{d_{x_w}}}{\sqrt[N]{d_x} * Q_\sigma}( (x_i,y_i,z)*Q_\sigma), ct+r(cx_{23}-x_{23}^w)  \big).$$
By the same correspondence, we have the isomorphism
$$ h_w \colon X_{\Phi_1,\bm\l} \otimes \k_{pN} \longrightarrow X_{\Phi_1,\bm\l_w} \otimes \k_{pN}.$$

\item We have $h_{w'}^{\bm\tau_w} \circ h_{w}^{\bm\tau} = h_{ww'}^{\bm\tau}$ for any $w' \in W_\D$.

\item Put $d_w := d_{x_w}/(d_x*Q_\sigma) \in (\k^*)^5$.
Define the automorphism $\pi_w$ of $G \times {\rm Gal}$ by
$$(\bm \xi, a ,e) \mapsto (K_{d_w}(e) (\bm\xi *Q_\sigma), ca+A_{cx_{23}-x_{23}^w}(e), e) \quad (\bm\xi \in (\k^*)^5, a \in \k).$$
Then, for $(g,e) \in G \times {\rm Gal}$, 
$$h_w \circ (g,e) = \pi_w(g,e) \circ h_w.$$

\item For $\chi = ( (\a_i)_i, \psi) \in \widehat{G}$, we have
$$ \chi(d_w, cx_{23}-x_{23}^w) N(X_{\Phi_1,\bm\l}; \chi_w) = N(X_{\Phi_1,\bm\l_w};\chi), $$
where $\chi_w := ( (\a_i)*{}^t Q_\sigma, \psi_c)$.
\end{enumerate}
\end{thm}
\begin{proof}
(i) We can prove this by a similar argument as in the proof of Theorem \ref{isom between X_FD} (i).

(ii) The proof is the same as that of Theorem \ref{isom between X22 and X22-sigma} (ii), by Lemma \ref{lem for Qsigma of Phi_1}.

(iii) and (iv) We can prove this by the same argument as in the proof of Theorem \ref{N of X12} \ref{N of X12 (iii)} and \ref{N of X12 (iv)}.
\end{proof}

\subsection{Humbert's $\Phi_3$ ($d=2, n=5, \D=(1,2,2)$)}
For $X_{\Phi_3,\bm\l}$, we can make observations similar to those in the previous subsection.
Recall $W_\D \cong W(2)^2 \rtimes \S_2$ and $ w\in W_\D$ can be written as 
$$w=(c_1, c_2, \sigma):= \left( \begin{array}{c|c} 1 & \mat{0 & \cdots & 0} \\ \hline \mat{0 \\ \vdots \\ 0} & {\rm diag}(\mu(c_1), \mu(c_2)) \widetilde P_\sigma \end{array}\right) \quad ( c_i \in \k^*, \, \sigma \in \S_2).$$
For $\bm\l = (\l_1, \l_2) \in (\k^*)^2$, fix a matrix $x=(x_{ij}) \in M(2,3;\k)$ such that
$$ \l_1 = \dfrac{x_{11} x_{22}}{x_{21}}, \quad \l_2 = x_{22} x_{13},$$
and put $d_x = (x_{21}, x_{11}, x_{22}, x_{13})$.
For 
$$z = \bmat{x_{11} & 1 & x_{12} & 0 & x_{13} \\ x_{21} & 0 & x_{22} & 1 & x_{23}}$$
and $w = (c_1, c_2, \sigma) \in W_\D$, one shows that $P_\sigma z w$ can be written as
$$ P_\sigma z w = \bmat{x_{11}^w & 1 & x_{12}^w & 0 & x_{13}^w \\ x_{21}^w & 0 & x_{22}^w & 1 & x_{23}^w } \quad (x_{11}^w, x_{21}^w, x_{22}^w, x_{13}^w \neq 0).$$
Put 
$$ x_w = (x_{ij}^w), \quad \l_{w,1} = \dfrac{x_{11}^w x_{22}^w}{x_{21}^w}, \quad \l_{w,2}= x_{22}^w x_{13}^w.$$
We have the decomposition
$$ X_{\Phi_3,\bm\l} \otimes \k_N = \bigsqcup_{\bm\tau^N  = \bm\l} X_{\Phi_3}^{\bm\tau},$$
where
$$  X_{\Phi_3}^{\bm\tau} = \{ (x,y,z_i,t_i) \in  X_{\Phi_3,\bm\l} \otimes \k_N \mid \tau_1 x=yz_1, \tau_2= z_1z_2\}. $$
For $\bm\tau= (\tau_1, \tau_2) \in \{ \bm\tau \mid \bm\tau^N=\bm\l\}$, put $\bm\tau_w = (\tau_{w,1},\tau_{w,2}) \in \{\bm\tau \mid \bm\tau^N=\bm\l_w\}$, where
$$ \tau_{w,1} = \tau_1 \cdot \dfrac{ \sqrt[N]{d_x} *{}^t (1,-1,-1,0)}{\sqrt[N]{d_{x_w}} *{}^t (1,-1,-1,0)}, \quad \tau_{w,2} = \tau_2 \cdot \dfrac{ \sqrt[N]{d_x} *{}^t (0,0,-1,-1)}{\sqrt[N]{d_{x_w}} * {}^t(0,0,-1,-1)}.$$
\begin{rem}
One shows
$$ \bm\tau_w^N = \bm\l_w  = \begin{cases} (c_1 \l_1, c_2\l_2) & (\sigma = {\rm id}) \\ (c_2 \dfrac{\l_2}{\l_1}, c_1c_2\l_2) & (\sigma = (1\, 2)). \end{cases}$$
\end{rem}

For $\sigma \in \S_2$, put
$$Q_\sigma =  \left( \begin{array}{c|c} P_\sigma & O_2 \\ \hline O_2 & P_\sigma \end{array} \right) \in GL_4(\Z).$$
Put $G := (\k^*)^4 \times \k^2$ and ${\rm Gal} = {\rm Gal}(\k_N/\k)$.
\begin{thm}
Let  $ w = (c_1, c_2, \sigma) \in W_\D $ and $\bm\tau \in \{ \bm\tau \mid \bm\tau ^N=\bm\l\}$.
\begin{enumerate}
\item We have the isomorphism
$$ h_w^{\bm\tau} \colon X_{\Phi_3,\bm\l}^{\bm\tau} \longrightarrow X_{\Phi_3,\bm\l_w}^{\bm\tau_w}\, ;\, (x,y,z_i,t_i) \mapsto (x', y', z_i', t_i'),$$
where
$$ (x',y',z_i') = \dfrac{\sqrt[N]{d_{x_w}}}{\sqrt[N]{d_x} * Q_\sigma} ( (x,y,z_i) *Q_\sigma) , \quad (t_1', t_2') = (c_1t_1, c_2t_2) * P_\sigma.$$
By the same correspondence, we have the isomorphism
$$ h_w \colon X_{\Phi_3,\bm\l} \otimes \k_N \longrightarrow X_{\Phi_3,\bm\l_w} \otimes \k_N.$$

\item We have $h_{w'}^{\bm\tau_w} \circ h_{w}^{\bm\tau} = h_{ww'}^{\bm\tau}$ for $w' \in W_\D$.

\item Put $d_w := d_{x_w}/(d_x*Q_\sigma) \in (\k^*)^4$.
Define the automorphism $\pi_w$ of $G \times {\rm Gal}$ by 
$$ ( \bm\xi, \bm a, e) \mapsto ( K_{d_w}(e) (\bm\xi * Q_\sigma), (a_1c_1, a_2c_2)*P_\sigma, e ),$$
where $\bm\xi \in (\k^*)^4, \bm a =(a_i)\in \k^2$.
Then, for $(g,e) \in G\times {\rm Gal}$, 
$$ h_w \circ (g,e) = \pi_w(g,e)\circ h_w.$$

\item For $\chi = ( (\a_i)_i, \psi, \psi) \in \widehat{G}$, we have
$$ \chi(d_w, 0, 0) N(X_{\Phi_3,\bm\l};\chi_w) = N(X_{\Phi_3,\bm\l_w};\chi),$$
where $\chi_w := ( (\a_i)*Q_\sigma, \psi_{c_{1}}, \psi_{c_{2}})$.
\end{enumerate}
\end{thm}
\begin{proof}
(i) We can easily check that the morphism is well-defined and $(h_w^{\bm\tau})^{-1} = h_{w^{-1}}^{\bm\tau_w}$.

(ii) This follows from $(x_w)_{w'} = x_{ww'}$, where noting that $\sigma \sigma' = \sigma'\sigma$ for $\sigma,\sigma' \in \S_2$.

(iii) and (iv)  We can prove the assertion by the same argument as in the proof of Theorem \ref{isom between X22 and X22-sigma} (iii) and \ref{cor of isom between X22 and X22-sigma}, where noting that ${}^t Q_\sigma = Q_\sigma$.
\end{proof}

\subsection{ Lauricella's $F_A$ ($d \geq 3, n=2d, \D=(1,\dots, 1)$)}\label{subsec. of X_A}
For $\bm\l:=(\l_i) \in (\k^*)^m$ ($m \in \Z_{\geq 2}$), fix a matrix
$$  x = \bmat{ x_{00} & x_{01} &  \cdots & x_{0m} \\ 
                             x_{10} & x_{11} &               & \\
                             \vdots  &            & \ddots &     \\
                             x_{m0} &            & & x_{mm}} \quad (x_{ij}=0\mbox{ if } 1\leq i \neq j \leq m), $$
where
$$ \l_i = \dfrac{x_{0 i} x_{i 0}}{x_{00}{x_{ii}}}.$$
Let $z = \left( \begin{array}{c|c}
x & I_{m+1} 
\end{array} \right) \in M(m+1, 2m+2;\k)$.
In this specialized case of $z$, the function $\Phi_\D(\chi;z)$ can be expressed in terms of Lauricella's $F_A^{(m)}(\bm\l)$.
In this subsection, we consider this case.
Let $X_x \subset \P^{2m+1}=\{ (u_i:v_i):=(u_0 : \cdots : u_m : v_0 : \cdots : v_m)\}$ be the variety defined by the equations
$$ \begin{cases}
u_0^N = x_{00} v_0^N +\cdots + x_{m0} v_m^N \vspace{3pt} \\
u_i ^N = x_{0i} v_0^N + x_{ii} v_i ^N & (i = 1, \dots, m) \vspace{3pt}\\
u_0 \cdots u_m v_0 \cdots v_m \neq 0. \end{cases}$$
Put $d_x = (x_{00}, \dots , x_{m0}, x_{11} , \dots, x_{mm}, x_{01}, \dots, x_{0m}) \in (\k^*)^{3m+1}$.
Define matrices $\theta_j, T_j \in M(3m+1, 2m+2; \Z)$ and $\rho_j, \rho \in M(2m+2, 3m+1;\Z)$ $(0\leq j \leq  m)$ as follows:
$$\theta_0 = \left( \begin{array}{c|c|c|c}
\mat{0\\\vdots\\0} & {}^t \Lambda & I_m & \mat{0\\\vdots\\0} \\
\hline
-1 & \mat{-1 & \cdots & -1 & 0} & \mat{-1 & \cdots & -1 } & 0 \\
\hline
\mat{0\\\vdots\\0} & -I_m & O_m &\mat{0\\\vdots\\0}\\
\hline
\mat{0\\\vdots\\0} & O_m & O_m & \mat{0\\\vdots\\0}
\end{array}\right), \quad  
T_j := \left( \begin{array}{c|c}
O & \mat{ 1 \\ -e_j \\ e_j \\ -e_j }
\end{array}\right),$$
($\Lambda$ is the shift matrix of size $m$ as in Section 3), 
$ T_0 = O , \quad \theta_j := \theta_0 + T_j$, and $\rho = \sum_{j=0}^m \rho_j$, where
$$\rho_0 = \left( \begin{array}{c|c|c}
-1  \ \cdots \  -1 & 0 \  \cdots \  0 & 0\   \cdots \ 0 \\
\hline 
O & -I_m & -I_m \\
\hline
I_{m+1} & \mat{0 \ \cdots \ 0\\ I_m} & \mat{1\  \cdots \ 1\\ O_m}
\end{array}\right), \quad 
\rho_j = \left( \begin{array}{c|c}
O_{2m+1} & O \\
\hline
0\ \cdots \ 0&  -{}^t e_j
\end{array}\right)\ (j \geq 1).$$
By a similar argument as in the Lauricella $F_D$ case, we have the isomorphism
$$\p_{\theta_0}^{\bm \tau} \colon X_{F_A^m,\bm\l}^{\bm \tau} \longrightarrow X_x \otimes \k_N\, ;\, (x_i, y_i, z_i) \longmapsto ( \sqrt[N]{d_x}^{-1} (x_i:y_i:z_i) )*\theta_0,$$
where $X_{F_A^m, \bm\l}^{\bm \tau} := \{ (x_i, y_i, z_i) \in X_{F_A^m,\bm\l} \otimes \k_N \mid \tau_i x_0 y_i = x_i z_i (\mbox{for all }i) \}$ for $\bm\tau^N = \bm\l$.
The inverse morphism is given by
$$(\p_{\theta_0}^{\bm \tau})^{-1} (u_i:v_i) = \dfrac{\sqrt[N]{d_x}}{\prod_{j=1}^m (1, \dots, 1, \tau_j')*\rho_j} ( (u_i,v_i) * \rho),$$
%{}^t \Lambda =  \bmat{ 0  & & & & \\ 1 & \ddots & & & \\ 0 & \ddots & \ddots & & \\ \vdots & \ddots & \ddots & \ddots & \\ 0 & 0 & \cdots & 1 & 0}.$$
where $\tau_j ' := \tau_j (\sqrt[N]{d_x}* (T_j e_{2m+2}) )$.

Put $\S' := \langle (i\ \,  m+i+1) \mid 2 \leq i \leq m+1 \rangle \subset \S_{2m+2}$.
Suppose that $1 \not\in \{ \sum_j \e_j \l_j \mid \e_j = 0, 1\}$.
Then, for any $\sigma \in \S'$, one shows that the submatrix $g$ formed by  the rightmost $m+1$ columns of $zP_\sigma$ is invertible and 
$$ g^{-1} zP_\sigma = \left( \begin{array}{c|c} x_\sigma & I_{m+1} \end{array}\right), \quad 
x_\sigma = \bmat{ x_{00}^\sigma & x_{01}^\sigma &  \cdots & x_{0m}^\sigma \\ 
                             x_{10}^\sigma & x_{11}^\sigma &               & \\
                             \vdots  &            & \ddots &     \\
                             x_{m0}^\sigma &            & & x_{mm}^\sigma},$$
where $x_{0j}^\sigma, x_{i0}^\sigma, x_{ii}^\sigma \in \k^*$ and the other $x_{ij}^\sigma$ are $0$.
Put $\bm\l_\sigma = (\l_{\sigma,i})$, where
$$ \l_{\sigma,i} = \dfrac{x_{0i}^\sigma x_{i0}^\sigma}{x_{00}^\sigma x_{ii}^\sigma}. $$
\begin{rem}
For $\sigma = (i\  \, m+i+1) \in \S'$, one can easily check that
$$ \l_{\sigma, i-1} = \dfrac{-\l_{i-1}}{1-\l_{i-1}}, \quad \l_{\sigma, j} = \dfrac{\l_j}{1-\l_{i-1}}\ (j \neq i-1).$$
Thus, we have $1 \not\in \{ \sum_j \e_j \l_{\sigma,j} \mid \e_j = 0,1\}$ and the property is preserved under composition and therefore holds for every $\sigma \in \S'$.
\end{rem}
As in the previous subsections, we have the isomorphism
$$ g_\sigma \colon X_x \longrightarrow X_{x_\sigma}\, ;\, (u_i : v_i) \longmapsto (u_i:v_i)*P_\sigma.$$
Let $M$ denote the matrix of the same form as in  Subsection \ref{subsec. of X_D}, now of size $2m+2$, and define
$$ Q_\sigma = \sum_{j=0}^m (\theta_0 P_\sigma M + T_j)\rho_j.$$
By a similar argument as in the proof of Lemma \ref{lem of Q_sigma for FD}, we have
\begin{equation}
Q_{\sigma}Q_{\sigma'} = Q_{\sigma\sigma'} \quad (\sigma, \sigma' \in \S'), \label{eq. of Qsigma for FA}
\end{equation}
and $Q_\sigma \in GL_{3m+1}(\Z)$.
Put $\bm\tau_\sigma = (\tau_{\sigma,j})$, where 
$$ \tau_{\sigma, j} := \tau_j \cdot \dfrac{\sqrt[N]{d_x} * {}^t (1, -{}^t e_j, {}^t e_j, -{}^t e_j)}{\sqrt[N]{d_{x_\sigma}} * {}^t (1, -{}^t e_j, {}^t e_j, -{}^t e_j)}.$$

The following is a geometric interpretation of well-known transformation formulas for $F_2$ over $\k$ (cf. \cite[Corollary 1.7]{TSB}) and their $m$-variable generalizations.
\begin{thm}
Let $\sigma \in \S'$ and $\bm\tau \in \{ \bm\tau \mid \bm\tau^N = \bm\l\}$.
\begin{enumerate}
\item We have the isomorphism 
$$ h_\sigma^{\bm\tau} \colon X_{F_A^m,\bm\l}^{\bm\tau} \longrightarrow X_{F_A^m,\bm\l_\sigma}^{\bm\tau_\sigma}\, ;\, (x_i,y_j,z_j) \longmapsto \dfrac{\sqrt[N]{d_{x_\sigma}}}{\sqrt[N]{d_{x}}*Q_\sigma}\big((x_i, y_j,z_j)*Q_\sigma \big).$$
Furthermore, we have the isomorphism $h_\sigma \colon X_{F_A^m,\bm\l} \otimes \k_N \rightarrow X_{F_A^m,\bm\l_\sigma} \otimes \k_N$ by the same correspondence.

\item We have $h_{\sigma'}^{\bm\tau_\sigma} \circ h_{\sigma}^{\bm\tau} = h_{\sigma \sigma'}^{\bm\tau}$ for $ \sigma' \in \S' $.

\item Put $d_\sigma = d_{x_\sigma}/ (d_x * Q_\sigma) \in (\k^*)^{3m+1}$. 
Define the automorphism $\pi_\sigma$ of $(\k^*)^{3m+1} \times {\rm Gal}$ by 
$$(g,e) \mapsto (K_{d_\sigma}(e) (g*Q_\sigma), e).$$
Then, for $(g,e) \in (\k^*)^{3m+1} \times {\rm Gal}$, 
$$ h_\sigma \circ (g,e) = \pi_\sigma(g,e) \circ h_\sigma.$$

\item For $\chi \in \widehat{(\k^*)^{3m+1}}$,
$$ \chi(d_\sigma) N(X_{F_A^m,\bm\l} ; \chi * {}^t Q_\sigma) = N(X_{F_A^m,\bm\l_\sigma};\chi).$$
\end{enumerate}
\end{thm}
\begin{proof}
(i) We can prove this by the same argument as in the proof of Theorem \ref{isom between X_FD} (i), where $h_{\sigma}^{\bm \tau} := (\p_{\theta_0}^{\bm \tau_\sigma})^{-1} \circ g_\sigma \circ \p_{\theta_0}^{\bm \tau}$.

(ii)  This is clear by \eqref{eq. of Qsigma for FA}.

(iii) and (iv) We can prove the assertion by the same argument as in the proof of Theorem \ref{isom between X22 and X22-sigma} (iii) and \ref{cor of isom between X22 and X22-sigma}.
\end{proof}

\begin{rem} Let us consider the reducible case $1 \in \{ \l_1, \l_2, \l_1+\l_2\}$.
Then, recall that Appell's $F_2 = F_A^{(2)}$ can be reduced to a ${}_3F_2$-function (cf. \cite[Theorem 3.13]{N2}).
We can obtain a geometric interpretation of these reduction formulas.
We have 
$$ X_{F_2,(\l,1)} = \bigsqcup_{a \in \k^*} X_{F_2,(\l,1)}^a,$$
where $ X_{F_2,(\l,1)}^a := \left\{ (x_i,y_j,z_j) \in X_{F_2,(\l,1)} \middle| \, ax_0y_2=x_2z_2 \right\}.$
We have the isomorphism
$$ i^a \colon {}_3X_{3,\l}\otimes \k_N \longrightarrow X_{F_2,(\l,1)}^a\otimes \k_N \, ;\, (x_i,y_i) \longmapsto \sqrt[N]{d}^a ( (x_i,y_i)*Q),$$
where $d := (-1, 1, 1, 1, 1, 1, -1)$, $\sqrt[N]{d}^a := (\sqrt[N]{-1}, 1, 1, 1, 1, 1, a \sqrt[N]{-1})$ and 
$$Q := \bmat{1&0&1&0&0&0&0\\ 0&0&0&1&0&0&0 \\ 1&0&0&0&0&0&1 \\ 0&1&0&0&0&0&0 \\ 0&0&0&0&0&1&0 \\ -1&0&-1&0&-1&0&-1}.$$
The subgroup 
$$G' = \{ ( (\xi_i)_{i=0}^2, (\zeta_j)_{j=1}^2, (\zeta_j')_{j=1}^2) \in (\k^*)^7 |\, \xi_0\zeta_2=\xi_2\zeta_2' \} \subset (\k^*)^7$$
naturally acts on $X_{F_2,(\l,1)}^a$.
Recall that the group $G := (\k^*)^6$ acts on ${}_3X_{3,\l}$.
Define an isomorphism $ \pi \colon G \times {\rm Gal} \rightarrow G' \times {\rm Gal}$ by $ (g,e) \mapsto (K_d(e)(g*Q), e)$.
Then, we have $i^a  \circ (g,e) = \pi(g,e)\circ i^a$ for any $(g,e) \in G \times {\rm Gal}$.
Thus, by Lemma \ref{Comparison of N by isom}, we have, for $\chi \in (\k^*)^7$ and $\rho \in \widehat{\rm Gal}$, 
\begin{align*}
N(X_{F_2,(\l,1)} \otimes \k_N ; (\chi,\rho) ) 
& = \dfrac{1}{N} \sum_{a \in \k^*} N(X_{F_2,(\l,1)}^a \otimes \k_N ; (\chi|_{G'},\rho))\\
& = N({}_3X_{3,\l} \otimes \k_N ; \pi^*(\chi|_{G'}, \rho)).
\end{align*}
Since
$$ \pi^*(\chi|_{G'},\rho) = (\chi*{}^t Q, \eta \rho) \quad (\eta := \chi \circ K_d),$$
we have, by \eqref{N of base change}, 
$$N(X_{F_2,(\l,1)}; \chi) = \chi(d) N({}_3X_{3,\l} ; \chi*{}^tQ).$$
If we put $\chi=(\a, \e,\e, \b_1,\b_2, \ol{\c_1}, \ol{\c_2})$, then $\chi*{}^t Q = (\a,\b_1,\a\ol{\c_2}, \e, \ol{\c_1}, \ol{\a\b_2}\c_2)$, and hence the identity above is equivalent to the reduction formula \cite[Theorem 3.13 (i)]{N2} by Theorems \ref{N of X_Lauricella} \ref{N of X_Lauricella-A} and \ref{N of X_mn}.

The cases $\l_1=1$ and $\l_1+\l_2=1$ can be treated similarly.
\end{rem}

\section*{Acknowledgements}
The author is deeply grateful to Noriyuki Otsubo for many helpful discussions and suggestions, and in particular for suggesting the construction of varieties associated with confluent hypergeometric functions. The author also thanks Fang-Ting Tu and Ryojun Ito for valuable comments, and Atsuhira Nagano and Shunya Adachi for helpful discussions on general and confluent hypergeometric functions over the complex numbers.
This work was supported by JST FOREST Program (JPMJFR2235).


\begin{thebibliography}{99}
%\bibitem{Ahlgren-Ono} S. Ahlgren and K. Ono, A Gaussiann hypergeometric series evaluations and Ap\'ery number congruences, J. reine angew. Math. {\bf 518} (2000), 187--212.
\bibitem{Archinard} N. Archinard, Hypergeometric abelian varieties, Canad. J. Math. {\bf 55} (5) (2003), 897--932.
\bibitem{Asakura-Otsubo} M. Asakura and N. Otsubo, On the adelic Gaussian hypergeometric function, arXiv:2408.080-\\12v1 (2024).
\bibitem{BCM} F. Beukers, H. Cohen and A. Mellit, Finite hypergeometric functions, Pure Appl. Math. Q., {\bf 11}(4) (2015), 559--589.
\bibitem{CK}A. S. Chetry and G. Kalita, Lauricella hypergeometric series $F_A^{(n)}$ over finite fields, The Ramanujan J., {\bf 57} (2022), 1335--1354.
\bibitem{FST}S. Frechette, H. Swisher and F.-T. Tu, A cubic transformation formula for Appell-Lauricella hypergeometric functions over finite fields, Res. Number Theory, {\bf 4}:27 (2018).
\bibitem{FLRST} J. Fuselier, L. Long, R. Ramakrishna, H. Swisher and F.-T. Tu, {\it Hypergeometric functions over finite fields}, Memoirs of the AMS (2022).
\bibitem{GGR} I. M. Gel'fand, M. I. Graev and V. S. Retakh, Hypergeometric functions over an arbitrary field, Russian Math. Surveys, {\bf 59}:5 (2004), 831--905.
\bibitem{Greene}J. Greene, Hypergeometric functions over finite fields, Trans. Amer. Math. Soc., {\bf 301} (1) (1987), 77--101.
\bibitem{He} B. He, A Lauricella hypergeometric series over finite fields, arXiv:1610.04473v3 (2017).
%\bibitem{He2} B. He, A finite field analogue for Appell series $F_3$, arXiv:1704.03509v2 (2017).
%\bibitem{He-Li-Zhang} B. He, L. Li and R. Zhang, An Appell series over finite fields, Finite Fields and Their Applications, {\bf 48} (2017), 289--305.
\bibitem{Humbert} P. Humbert, The confluent hypergeometric functions of two variables, Proc. Royal Soc. Edinburgh, {\bf 41} (1920), 73--96.
\bibitem{IKNN} R. Ito, S. Kumabe, A. Nakagawa and Y. Nemoto, Kamp\'e de F\'eriet hypergeometric functions over finite fields, Res. Number Theory, {\bf 9}:52 (2023).
\bibitem{Katz}N. M. Katz, {\it Exponential sums and differential equations}, Annals of Mathematics Studies, Volume 124, Princeton University Press, Princeton, NJ, 1990.
\bibitem{K-T} H. Kimura and M. Taneda, Analogue of flat basis and cohomological intersection numbers for general hypergeometric functions, J. Math. Sci. Univ. Tokyo, {\bf 6} (1999), 415--436.
\bibitem{K-H-T} H. Kimura, Y. Haraoka and K. Takano, The generalized confluent hypergeometric functions, Proc. Japan Acad., {\bf 68} Ser. A (1992), 290--295.
\bibitem{K-K} H. Kimura and T. Koitabashi, Normalizer of maximal abelian subgroups of $GL(n)$ and general hypergeometric functions, Kumamoto J. Math., {\bf 9} (1996), 13--43.
\bibitem{Koblitz} N. Koblitz, The number of points on certain families of hypersurfaces over finite fields, Compositio Math., {\bf 48} (1983), 3--23.
% \bibitem{K.-S.} T.-H. Koornwinder and J.-V. Stokman (Eds.), {\it Encyclopedia of special functions the Askey-Bateman project, Vol. II: Multivariable special functions}, Cambridge University Press, 2021.
\bibitem{Li-Li-Mao} L. Li, X. Li and R. Mao, Appell series $\mathbbm{F}_1$ over finite fields, Int. J. Number Theory, {\bf 14} (3) (2018), 727--738.
%\bibitem{Ma} H. Ma, Some properties for Appell series $F_2$ over finite fields, Integral Transforms and Special Functions, {\bf 30} (12) (2019), 992--1003.
\bibitem{Macdonald} I. Macdonald, {\it Symmetric functions and Hall polynomials}, Oxford Univ. Press, 1995.
\bibitem{Mc} D. McCarthy, Transformations of well-poised hypergeometric functions over finite fields, Finite Fields and Their Applications, {\bf 18} (6) (2012), 1133--1147.
%\bibitem{Redbook} D. Mumford, {\it The Red Book of Varieties and Schemes} (2nd. expanded edition), Springer-Verlag Berlin Heidelberg 1999.
%\bibitem{Nakagawa Diagonal} A. Nakagawa, Artin $L$-functions of diagonal hypersurfaces and generalized hypergeometric functions over finite fields, Hiroshima Math. J,, {\bf 54} (3) (2024), 375--399.
\bibitem{N} A. Nakagawa, Appell-Lauricella hypergeometric functions over finite fields and algebraic varieties, Hokkaido Math. J., {\bf 53} (2) (2024), 307--347.
\bibitem{N2} A. Nakagawa, Sum representations of Appell-Lauricella functions over finite fields using confluent hypergeometric functions and their applications, Res. Number Theory, {\bf 10}:74 (2024).
\bibitem{Ohyama} Y. Ohyama, A unified approach to $q$-special functions of the Laplace type, arXiv:1103.5232 (2011).
\bibitem{Otsubo} N. Otsubo, Hypergeometric functions over finite fields, Ramanujan J., {\bf 63} (2024), 55--104.
%\bibitem{Otsubo-Senoue}  N. Otsubo and T. Senoue, Product formulas for hypergeometric functions over finite fields, Res. Number Theory, {\bf 8}:80 (2022).
\bibitem{Serre} J.-P. Serre,  Zeta and $L$-functions, {\it Arithmetical Algebraic Geometry}, Harper and Row, New York (1965), 82--92.
\bibitem{Slater} L. J. Slater, {\it Generalized hypergeometric functions}, Cambridge University Press, 1966.
\bibitem{TB} M. Tripathi and R. Barman, A finite field analogue of the Appell series $F_4$, Res. Number Theory, {\bf 4}:35 (2018). 
\bibitem{TSB} M. Tripathi, N. Saikia and R. Barman, Appell's hypergeometric series over finite fields, Int. J. Number Theory, {\bf 16} (4) (2020), 673--692.
\bibitem{Vidunas} R. Vid\={u}nas, Specialization of Appell's functions to univariate hypergeometric functions, J. Math. Analysis and Applications {\bf 355} (2009), 145--163.
\end{thebibliography}
\end{document}